\RequirePackage{fix-cm}
\documentclass[oneside,a4paper,11pt,numbers=noenddot,headinclude,footinclude=true,cleardoublepage=empty]{scrartcl}

\usepackage[main=english]{babel}
\usepackage[T1]{fontenc}

\usepackage{amsmath,amssymb,amsthm,amsfonts,amstext,amsbsy,
amscd}

\usepackage{hyperref}
\usepackage{comment}
\DeclareOldFontCommand{\bf}{\normalfont\bfseries}{\mathbf} 

\usepackage[dvipsnames]{xcolor}
\hypersetup{
    colorlinks,
    linkcolor={blue!80!black},
    citecolor={green!50!black},
    urlcolor={blue!80!black}
  }

\usepackage{enumitem}
\setlist{  
  listparindent=\parindent,
  parsep=0pt,
}

\usepackage{caption}
\usepackage[labelfont=bf]{caption}

\usepackage{graphicx}
\usepackage{subcaption}
\graphicspath{{../}{plots/}} 
\usepackage{wrapfig}
\usepackage[normalem]{ulem}
\usepackage{booktabs}
\usepackage{tabularx}

\usepackage{amsmath,amssymb,amstext,amsthm,bm,dsfont,mathtools}
\newtheorem{prop}{Proposition}

\newtheorem{theorem}{Theorem}
\newtheorem{lemma}[theorem]{Lemma}
\newtheorem{corollary}[theorem]{Corollary}

\newtheorem{mdef}[theorem]{Definition}

\theoremstyle{definition}
\newtheorem{ass}[theorem]{Assumption}
\newtheorem{remark}[theorem]{Remark}

\newtheorem{thmx}{Question}[section]

\newcommand{\rr}[0]{\mathbb{R}}
\newcommand{\R}[0]{\mathbb{R}}

\newcommand{\nn}[0]{\mathbb{N}}
\newcommand{\zz}[0]{\mathbb{Z}}
\newcommand{\E}[0]{\mathbb{E}}
\newcommand{\ad}{\mathfrak{a}} 
\newcommand{\bd}{\mathfrak{b}} 
\newcommand{\pino}[2]{%
\ensuremath e_{#1 ,#2}}
\newcommand{\sigf}[1]{%
\ensuremath \mathcal{F}_{#1}}
\newcommand{\sigfdouble}[2]{%
\ensuremath \mathcal{F}_{#2}^{#1}}
\DeclareMathOperator{\argmin}{armin}
\newcommand{\inPlong}{\overset{\mathbb{P}}{\longrightarrow}}
\newcommand{\inLplong}[1]{\overset{L^{#1}}{\longrightarrow}}
\newcommand{\weakly}{\overset{w}{\longrightarrow}}

\newcommand{\disteq}{%
  \ensuremath \overset{d}{=}
}

\newcommand{\proofof}[1]{Proof (#1)}
\newcommand{\I}[4]{%
\ensuremath E_{#1,#2}^{#3,#4}}
\newcommand{\II}[3]{%
\ensuremath e_{#1}^{#2,#3}}
\renewcommand{\P}{\mathbb{P}}

\newcommand{\MJ}{\color{blue}}
\newcommand{\GK}{\color{red}}

\usepackage{ifthen}
\newcommand{\cformat}[3]{%
  \ifthenelse{\boolean{#1}}{#2}{#3}%
}
\newcommand{\defbool}[2]{%
  \newboolean{#1}%
\setboolean{#1}{#2}}

\defbool{main-lemma-prep-0}{false} 
\defbool{main-lemma-prep-1}{false} 
\defbool{main-lemma-stepI-1}{false}
\defbool{main-lemma-stepI-2}{false}
\defbool{main-lemma-stepI-3}{false}
\defbool{main-lemma-stepI-4}{false}
\defbool{main-lemma-stepII-1}{false}
\defbool{main-lemma-stepIII2-1}{false}
\defbool{main-lemma-stepIII2-2}{false}
\defbool{main-lemma-stepIV1-1}{false}
\defbool{main-lemma-stepIV1-2}{false}
\defbool{main-lemma-stepIV1-3}{false}
\defbool{main-lemma-stepIV1-4}{false}
\defbool{main-lemma-stepIV2-1}{false}
\defbool{main-lemma-stepIV2-2}{false}
\defbool{main-lemma-stepIV2-3}{false}

\title{Sharp oracle inequalities and universality of the AIC and FPE}
\author{
  Jirak, Moritz\\
  \texttt{University of Vienna}
  \and
  Köstenberger, Georg\\
  \texttt{University of Vienna}
}

\begin{document}
\frenchspacing
\raggedbottom

\maketitle
\begin{abstract}
In two landmark papers, Akaike introduced the AIC and FPE, demonstrating their significant usefulness for prediction. In subsequent seminal works, Shibata developed a notion of asymptotic efficiency and showed that both AIC and FPE are optimal, setting the stage for decades-long developments and research in this area and beyond. Conceptually, the theory of efficiency is universal in the sense that it (formally) only relies on second-order properties of the underlying process $(X_t)_{t \in \zz}$, but, so far, almost all (efficiency) results require the much stronger assumption of a linear process with independent innovations. In this work, we establish sharp oracle inequalities subject only to a very general notion of weak dependence, establishing a universal property of the AIC and FPE. A direct corollary of our inequalities is asymptotic efficiency of these criteria. Our framework contains many prominent dynamical systems such as random walks on the regular group, functionals of iterated random systems, functionals of (augmented) Garch models of any order, functionals of (Banach space valued) linear processes, possibly infinite memory Markov chains, dynamical systems arising from SDEs, and many more. 


\end{abstract}

\section{Introduction}

Consider a time series $X_1, \ldots, X_n \in \rr$. How do we predict the future $X_{n+1}$, based on past observations? This fundamental problem has seen decades of research, see for instance \cite{An:hannan:1982:Aos}, \cite{bhansali:1986}, \cite{hannan-deistler-book}, \cite{fuller:1981:jasa}, \cite{parzen:book:1977}, \cite{davis:brockwell:book}, \cite{brillinger:book:1981}, \cite{davisson:1965}, \cite{hurvich1989regression}, \cite{hannan:quinn:1979:JrssB}, \cite{shibata:1981:biometrika} for earlier accounts, and \cite{lim2021time}, \cite{ing-wei-main}, \cite{ing:2007:APE:AOS}, \cite{ing:2020:Aos}, \cite{wu:EJS:2016}, \cite{gosh:michailidis:AoS:2021}, \cite{kreiss:aos:2011}, \cite{zhou:AoS:2023}, \cite{Politis:2021} for some more recent advances. In landmark papers, Akaike introduced the FPE (finite prediction error) in \cite{akaike1969fitting}, and his Information Criterion in \cite{akaike:1974}, now famously known as AIC, to address this problem in the context of linear prediction. As Akaike pointed out himself in \cite{akaike1969fitting}, related concepts to the FPE were also obtained in \cite{davisson:1966} at roughly the same time. Other highly influential ideas were offered by Schwarz \cite{schwarz:1978} and Rissanen \cite{RISSANEN1978}, which are, however, less relevant to our aims here.

Already shortly after their introduction, the AIC and FPE were reported to perform exceptionally well in practice, and have had a huge impact on both theory and practice ever since. In a subsequent seminal work, Shibata \cite{shibata} introduced a notion of asymptotic efficiency, which informally reads as

\begin{align}\label{defn:asymp:eff}
\frac{\text{Prediction based on Modelselection}}{\text{Optimal Oracle-based Prediction}} \xrightarrow[n \to \infty]{\mathbb{P}} \text{Efficiency} \in [0,1].
\end{align}
He showed that both AIC and FPE (among others) fulfill the criterion of optimality (asymptotic efficiency equal to one), offering an explanation for their huge success. Since then, many more adjustments, extensions and further properties were obtained, see for instance \cite{hannan:quinn:1979:JrssB}, \cite{hannan-deistler-book}, \cite{ing-wei-main}, \cite{ing:2007:APE:AOS}, \cite{bhansali:1986}, \cite{bhansali1996asymptotically}, \cite{karagrigoriou2001}, \cite{karagrigoriou1995}, \cite{Politis:2021} for a small fraction of the vast literature in this area. While \eqref{defn:asymp:eff} has its merits, we take a more modern viewpoint in the sequel and present our main results in terms of (sharp) finite sample oracle inequalities (cf. \cite{barron:1999ptrf}, \cite{Candes:2006}). More precisely, we will show that with high probability, 
\begin{align}\nonumber
\text{Prediction based} \,&\, \text{on Modelselection} \\&\leq \text{Optimal Oracle-based Prediction} \times \big(1 + o(1) \big). \label{defn:oracle:intro}
\end{align}
The desired property of asymptotic efficiency is then an immediate consequence of \eqref{defn:oracle:intro}.
\\

A fundamental drawback of all the results regarding and related to (asymptotic) efficiency, including in particular those of Akaike and Shibata, are the underlying assumptions. It appears that in the literature\footnote{We are not aware of any substantial exceptions here.}, the core condition is always that the (stationary) process $(X_t)_{t \in \zz}$ is a (weakly dependent) AR($\infty$)-process with independent innovations, that is, we have the (formal) recurrence equation
\begin{align}\label{eq:intro:ar}
    X_t = a_1 X_{t-1} + a_2 X_{t-2} + \ldots + \varepsilon_t = \sum_{j = 1}^{\infty} a_j X_{t-j} + \varepsilon_t,
\end{align}
where $(\varepsilon_t)_{t \in \zz}$ is a sequence of i.i.d. random variables with zero mean\footnote{The general case $\E (X_t) \neq 0$ can be readily reduced to the zero mean case, see for instance \cite{karagrigoriou1995}.} and $\sum_{j \geq 1} j^{\alpha}|a_j| < \infty$ for some $\alpha \geq 0$, reflecting the weak dependence. These assumptions are justified from a technical point of view, since the derivation of (efficiency) results such as those in \cite{shibata}, \cite{ing-wei-main}, \cite{ing:2007:APE:AOS} is highly non-trivial, a barrier being the unavoidable, rather complex underlying dependence structures arising in the proofs. 


On the other hand, from a more conceptual point of view, a (very) strong restriction like $(\varepsilon_t)_{t \in \zz}$ being i.i.d. seems to be rather unnecessary. Indeed, all the basic quantities are defined in terms of (empirical) autocovariance functions
\begin{align}
  \hat{\gamma}_n(h) = \frac{1}{n-h} \sum_{t = h+1}^{n} X_t X_{t-h}, \quad h \in \zz,
\end{align}
which, by the ergodic theorem, are consistent estimates (for fixed $h$) if $(X_t)_{t \in \zz}$ is ergodic, stationary, and $\E (X_t^2) < \infty$. This indicates that the concepts and ideas laid down by Akaike \cite{akaike1969fitting} and Shibata \cite{shibata} are somewhat universal and should work in a (much) more general framework. In addition, in practice one 'just' applies the AIC/FPE (and related criteria) anyway, assuming said 'universality', since it is essentially impossible to decide whether a linear process with i.i.d. innovations are present or not. It is clear that ergodicity and moment conditions alone will not be sufficient, rather, some notion of weak dependence will be required for more quantitative results. We are thus led to the following question.

\begin{thmx}\label{QA}
\textit{Is there a general notion of weak dependence such that Shibata's asymptotic efficiency results regarding AIC and FPE in \cite{shibata} persist, that is, do we have a 'universality' property?}
\end{thmx}

Answering this question and thus establishing 'universality' is precisely the aim of this work. Subject only to a very general notion of weak dependence, we will show that the asymptotic efficiency results of Shibata \cite{shibata} still hold, including in particular Akaike's AIC and FPE. Our framework contains a large and diverse number of prominent dynamical systems and time series models used in econometrics, finance, physics, and statistics. Among others, this includes random walks on the general linear group, functionals of (augmented) Garch models of any order, functionals of dynamical systems arising from stochastic differential equations (SDEs), functionals of infinite order Markov chains, linear processes, iterated random systems, and many more, see Sections \ref{sec:prelim} and \ref{sec:examples} for more details and examples.\\

The tools we develop can be used to generalize many more results from the linear process setting to our much more general setup. To exemplify this further, we also establish 
asymptotic normality for our underlying estimators and discuss some further applications.

This note is structured as follows. In Section \ref{sec:prelim}, we introduce some notation and our basic model assumptions. Here, we also briefly discuss our notion of weak dependence and how it relates to (weakly) dependent AR$(\infty)$ processes. In Section \ref{sec:indep}, we show that AIC and its variants are asymptotically efficient model selection criteria for predicting arbitrary, weakly dependent, stationary processes, and some prominent examples are discussed in Section \ref{sec:examples}. Section \ref{sec:sim} contains a small simulation study,
and the proofs are relegated to Section \ref{sec:proofs}.

\section{Notation and main results}\label{sec:main}

To begin with, let us introduce some general notation. All random variables are defined on a common, rich enough probability space $(\Omega,\mathcal{A},\mathbb{P})$. 
For a real-valued random variable $X$, we write $\|X\|_{p} = \mathbb{E}(|X|^{p})^{1/p}$ for its $L^{p}$-norm, and denote with $L^{p}$ the corresponding space. 
We write $X \disteq Y$ if two random variables $X$ and $Y$ have the same distribution. For a sequence $x=(x_{n}) \in \mathbb{R}^{\infty}$ and $p\geq 1$, we write $\|x\|_{\ell^{p}} = \bigl(\sum_{n=1}^{\infty}|x_{n}|^{p}\bigr)^{1/p}$ for its $\ell^{p}$ norm, and denote with $\ell^{p} = \{x \in \mathbb{R}^{\infty}\mid \|x\|_{\ell^{p}}< \infty\}$.
For two functions $f,g:\mathbb{N}\to \mathbb{R}$, we write $f(n) \lesssim g(n)$ if there is an absolute constant $C > 0$, such that $f(n) \leq Cg(n)$ for all $n\in \mathbb{N}$. If $f(n) \lesssim g(n)$ and $g(n) \lesssim f(n)$, we write $f(n) \asymp g(n)$. If a constant $A$ is \textit{swallowed} by $\lesssim$, that is, if we use $Ag(n) \lesssim g(n)$, we sometimes write $\lesssim_{A}$ to highlight the overall dependence on $A$. Additionally, we set $x \wedge y = \min(x,y)$ and $x\vee y = \max(x,y)$. Given a stationary process $(X_t)_{t \in \zz}$, we denote with 
\begin{align*}
\gamma(h) = \E (X_t X_{t+h}), \quad   R(k) = \big(\gamma(i-j)\big)_{1 \leq i,j,\leq k}, \quad k \in \nn, 
\end{align*}
the autocovariance function and the covariance matrix of $(X_{t},\dots, X_{t-k+1})^{\top}$.

\subsection{Preliminaries}\label{sec:prelim}

According to classical Yule-Walker theory, the orthogonal projection of $X_{t}$ onto the linear subspace of $L^{2}$ spanned by the last $k$ observations $X_{t-1}, \dots, X_{t-k}$, is given by 
\begin{align} \nonumber
&P_k(X_{t}) = -\sum_{i=1}^k a_i(k) X_{t-i}, \quad \text{where} \\ 
&a(k) = - R(k)^{-1}r(k), \quad 
r(k) = (\gamma(1), \gamma(2),\dots ,\gamma(k)).    
\end{align}
Thus, the candidate model we try to mimic  
for prediction is the $\text{AR}(k)$ model
\begin{equation}\label{eq:fit-k}
  X_{t} + a_1(k) X_{t-1} + \cdots + a_k(k) X_{t-k} = \pino{t}{k},
\end{equation}
with residuals $\pino{t}{k}$ that we may also view as innovations. The optimal oracle model for comparison is the case $k = \infty$, where we use the abbreviations
\begin{align}
 a =  (a_i)_{i \in \nn} = \big(a_i(\infty)\big)_{i \in \nn}, \quad (e_{t})_{t \in \zz} =  (\pino{t}{\infty})_{t \in \zz}, \quad  \text{and} \quad \sigma^2 = \E(e_{t}^{2}). 
\end{align}

Let us note here that if $\sum_{h \geq 0} |\gamma(h)| < \infty$ and the spectral density $f_X $ of the process $(X_t)_{t \in \zz}$ satisfies $f_X  \neq 0$, then all the above quantities are well-defined, see for instance \cite{brillinger:book:1981}. 
In particular, $\sup_{k\geq1} \|R(k)\|, \sup_{k\geq 1}\|R(k)^{-1}\| < \infty$.
Now, given only a finite number of observations, we cannot fit $\text{AR}(k)$ models of arbitrary order $k$. Thus, for any given $n$, we restrict ourselves to candidate models of order at most $K_n < n$. 
Based on the set of observations $X_1, \dots, X_n$, we obtain estimators
\begin{align}\label{defn:hat:R:and:hat:a}
  \hat{R}(k) = \frac{1}{N}\sum_{t=K_n}^{n-1} X_t(k)X_t(k)^T, \quad \text{and} \quad \hat{a}(k) = -\hat{R}(k)^{-1}\frac{1}{N}\sum_{t=K_n}^{n-1}X_t(k)X_{t+1}
\end{align}
of $R(k)$ and $a(k)$ respectively, where $X_t(k) = (X_t, \dots, X_{t-k+1})$ and $N=n-K_n$.
Note that $\hat{a}(k)$ is not the least-squares estimator of $a(k)$, but the difference between $\hat{a}(k)$ and the least squares estimator of $a(k)$ will be asymptotically negligible subject to our assumptions, see below. 

In order to measure the quality of our potential predictors, we use the classical \textit{independent realization setting} (cf. \cite{akaike1969fitting}, \cite{davis:brockwell:book}, \cite{shibata}). Given an independent copy $Y_1, \ldots, Y_{n+1}$ of $X_1,\ldots,X_{n+1}$, the (empirical) risk is defined as
\begin{align}\label{defn:Q_n(k):first}
Q_n(k) = \E\big[(Y_{n+1} - \hat{Y}_{n+1}^{(k)})^2\big| \mathcal{X}_n\big] - \sigma^2,
\end{align}
where $\mathcal{X}_n = \sigma(X_1,\ldots,X_n)$ denotes the $\sigma$-algebra generated by $X_1,\ldots,X_n$, and $\hat{Y}_{n+1}^{(k)}$ is the linear predictor of $Y_{n+1}$ of order $k$, that is, with $Y_n,\ldots,Y_{n-k+1}$ as predictors (covariates) and $\hat{a}_i(k)$ measurable with respect to $\mathcal{X}_n$, defined above in \eqref{defn:hat:R:and:hat:a}. In other words, we have
\begin{align*}
\hat{Y}_{n+1}^{(k)} = \hat{a}_1(k)Y_n + \hat{a}_2(k) Y_{n-1} + \ldots + \hat{a}_k(k)Y_{n-k+1},
\end{align*}
where $\hat{a}(k) = \big(\hat{a}_i(k)\big)_{1 \leq i \leq k} \in \mathcal{X}_n$. In \eqref{defn:Q_n(k):second} below, we provide an alternative, more explicit expression for $Q_n(k)$.

This independent realization setting has been mainly studied in the literature, and we stick to this setting. The opposing concept of \textit{same realization prediction} (cf. \cite{ing:wei:JMVA:2003}, \cite{ing-wei-main}) has been introduced more recently, and arguably may be more realistic for many applications in practice. On the other hand, it is shown for instance in \cite{ing-wei-main} that the FPE and AIC are (asymptotically) efficient in both settings, somewhat reconciling both concepts. Note, however, that the same realization setting appears to be even more difficult to handle from a technical level than the independent realization setting, see for instance \cite{ing-wei-main} or \cite{ing:neg:mom:2021}.

Let us now turn to one of our key aspects, namely a weak dependence measure we want to employ to establish universality and thus answer Question \ref{QA}. Over the past decades, a great number of different ways to define and quantify weak dependence have been established in the literature, see for instance \cite{wu-physical-dependence}, \cite{wu2011asymptotic} and Section \ref{sec:examples}. We take up the following viewpoint. 

Consider a sequence of real-valued random variables $X_1,\ldots,X_n$. It is well known (cf. \cite{rosenblatt:book:2000}), that, on a possibly larger probability space, this sequence can be represented as
\begin{align}\label{eq_structure_condition}
{X}_{t} = g_t\bigl(\varepsilon_{t}, \varepsilon_{t-1}, \ldots \bigr), \quad 1 \leq t \leq n,
\end{align}
for some measurable functions $g_t$\footnote{In fact, $g_t$ can be selected as a map from $\rr^t$ to $\rr$, and this extends to Polish spaces.}, where $(\varepsilon_t)_{t \in \zz}$ is a sequence of independent and identically distributed random variables taking values in some measurable space. 
This motivates the following setup: 
Denote the corresponding $\sigma$-algebra with $\mathcal{F}_t = \sigma( \varepsilon_j, \, j \leq t)$, and write $X_t = g_t(\xi_t)$, where $\xi_t = (\varepsilon_t,\varepsilon_{t-1}, \ldots)$. 
For future use, we also define the $\sigma$-algebras $\sigfdouble{t}{s} = \sigma(\varepsilon_{s},\dots,\varepsilon_{t})$, for $s\leq t$.
Given a real-valued stationary sequence $({X}_t)_{t \in \zz}$, we always assume that $X_t$ is adapted to $\mathcal{F}_{t}$ for each $t \in \zz$. 
Hence, we implicitly assume that $X_{t}$ can be written as in \eqref{eq_structure_condition}. 
Over the past decades, an important question in the dynamical systems literature has been whether a stationary process $(X_t)_{t \in \zz}$ satisfies a representation like \eqref{eq_structure_condition} or not\footnote{The problem arises due to $\zz$ having infinite cardinality, whereas $n$ in \eqref{eq_structure_condition} is finite.} (e.g. \cite{emery_schachermayer_2001}, \cite{vershik_doc_transl}), and if so, whether the function $g_t$ depends on $t$ or not (e.g. \cite{feldman_rudolph_1998}). 
Both questions are, however, irrelevant to our cause, our key requirement rather being that $X_t$ and $X_{t - m}$ are close to being independent for $m$ large, see below. 
Although we will always write and express our conditions in terms of $(X_t)_{t \in \zz}$ for simplicity, we effectively only work with $X_1,\ldots,X_n$. 
Similarly, we will always assume that $(X_t)_{t \in \zz}$ is strictly stationary, although this may be readily relaxed to local (weak) stationarity or quenched setups. 

A useful feature of representation \eqref{eq_structure_condition} is that it allows to give simple, yet efficient and general dependence conditions. Following \cite{wu-physical-dependence} and his notion of physical dependence, let $(\varepsilon_t')_{t \in \zz}$ be an independent copy of $(\varepsilon_t)_{t \in \zz}$ on the same, rich enough probability space. Slightly abusing notation, let
\begin{align}
X_t^{(l)} = g_t\big(\varepsilon_t, \ldots,\varepsilon_{l+1}, \varepsilon_{l}', \xi_{l-1}\big), \quad  t,l \in \zz.
\end{align}
For $q \geq 1$, we then measure weak dependence in terms of the distance
\begin{align}\label{defn:dep:measure:general}
\delta_{q}^{X}(l) = \sup_{t \in \zz}\big\|X_t - X_t^{(t-l)}\big\|_q, \quad l \in \nn.
\end{align}

Observe that if the functions $g_t$ satisfy $g_t = g$, that is, they do not depend on $t$, the above simplifies to
\begin{align*}
\delta_{q}^{X}(l) = \big\|{X}_{l} - {X}_{l}^{\prime}\big\|_q, \quad \text{with $X_l' = X_l^{(0)}$ for $l \in \nn$.}
\end{align*}
In this case, the process $(X_t)_{t \in \zz}$ is typically referred to as (time homogenous) Bernoulli-shift process. Note that in general we may not even have $\delta_{q}^{X}(l) \to 0$ as $l$ increases, as can be seen from the simple example $X_k = X$ for $k \in \zz$, with $\|X\|_q < \infty$.

In the sequel, we will require more than $\delta_{q}^{X}(l) \to 0$, namely a certain minimal amount of polynomial decay for $\delta_{q}^{X}(l)$, and express this as
\begin{align}\label{defn:Theta}
D_q^{X}(\alpha) = \big\|X_0\big\|_q +  \sum_{l = 1}^{\infty} l^{\alpha}\delta_{q}^{X}(l), \quad \alpha \geq 0.
\end{align}

On the other hand, measuring (weak) dependence in terms of \eqref{defn:Theta}, that is, demanding $D_q^{X}(\alpha) < \infty$, is still quite general, easy to verify in many prominent cases, and has a long history going back at least to \cite{billingsley_1968}, \cite{ibraginov_1966}, we refer to Section \ref{sec:examples} for a brief account and some references. Among others, we specifically discuss the cases of random walks on the general linear group, functionals of Garch models of any order, functionals of dynamical systems arising from SDEs, functionals of linear processes, and infinite order Markov chains.

\subsection{Sharp oracle inequalities, asymptotic efficiency and weak dependence}\label{sec:asympt:efficiency:weak:dep}

If the spectral density $f_X$ is bounded and bounded away from zero, then

\begin{equation*}
  \langle x, y \rangle_R = \sum_{i,j=0}^\infty x_iy_j R_{ij} \quad  \|x\|_R^2 = \sum_{i,j=0}^\infty x_ix_j R_{ij}, \quad x,y \in \ell^2,
\end{equation*}
are well defined, that is, $\langle \cdot,\cdot \rangle_{R}$ is positive definite and $\|\cdot\|_{R}$ is a norm. A nice feature of the independent realization setting is that the risk $Q_n(k)$, defined in \eqref{defn:Q_n(k):first}, can be re-expressed as
\begin{align}\label{defn:Q_n(k):second}
  Q_n(k) \coloneqq \|\hat{a}(k) - a\|_R^2.
\end{align}

We consider $\mathbb{R}^k$ to be the subspace of $\ell^2$ consisting of vectors with non-zero values appearing only in the first $k$ entries.  
If $x, y \in \mathbb{R}^k \subseteq \ell^2$, note that $\langle x,y \rangle_R = \langle x,R(k) y\rangle$, where $\langle \cdot , \cdot \rangle$ denotes the standard scalar product on $\mathbb{R}^k$. 
The scalar product $\langle \cdot,\cdot \rangle_{R}$ is chosen in such a way, that $a(k)$ is the orthogonal projection of $a$ onto $\mathbb{R}^{k}\subseteq \ell^{2}$. Thus we have 
\begin{equation*}
  \|\hat{a}(k) - a\|_R^2 = \|\hat{a}(k) - a(k)\|_R^2 + \|a(k) - a\|_R^2.
\end{equation*}
The first term becomes small as $n$ gets large and $k$ remains small (relative to $n$), while the second term vanishes as $k$ goes to infinity. 
Replacing the variance term $\|\hat{a}(k) - a(k)\|_R^2$ with its optimal rate $\sigma^2 k/N$ (cf. \cite{ing-wei-main}, \cite{shibata}), we set
\begin{align}\label{defn:L_n(k)}
  L_n(k) \coloneqq \frac{k}{N}\sigma^2 + \|a-a(k)\|_R^2.
\end{align}


We can now make the following definition.

\begin{mdef}\label{defn:oracle:inequality}
We say that a (stochastic) sequence $k_n$ of model orders satisfies \textmd{a sharp oracle inequality}, if there exist (absolute) constants $C,c,\gamma>0$ and a sequence $\mathfrak{a}_n \to 0$, such that with probability at least $1 - C(k_n^*)^{-\gamma}$, we have
\begin{align}\label{eq:defn:oracle:inequality}
\big|Q_n(k_n) - L_n(k^*_n)\big| \leq c \mathfrak{a}_n L_n(k^*_n),
\end{align}
where $k_n^* \in \{1, \dots,K_n\}$ is a sequence of minimizers of $L_n(k)$.  
\end{mdef}

\begin{remark}
Note that although this definition is expressed in terms of sequences, the view point we take here is of a finite sample perspective. In typical examples - see also the short discussion below Corollary \ref{cor:lower:bound} - one has $k_n^* \asymp n^{1/(2\alpha + 1)}$ together with the familiar rates $L_n(k^*_n) \asymp n^{-2\alpha/(2\alpha +1)}$, in which case \eqref{eq:defn:oracle:inequality} resembles classical (sharp) nonparametric oracle inequalities.
\end{remark}

\begin{remark}\label{def:asymp:efficiency}
If $k_n^* \to \infty$, then a sequence $k_n$ that satisfies a sharp oracle inequality implies in particular 
\begin{equation*}
  \frac{Q_n(k_n)}{L_n(k^*_n)}  \xrightarrow[n \to \infty]{\mathbb{P}} 1.
\end{equation*}
We call such sequences \textit{asymptotically efficient}, which is in line with \cite{shibata}. Condition $k_n^* \to \infty$ is standard in the literature and an immediate consequence of Assumption \ref{G1} below (cf. Lemma 4.1 in \cite{karagrigoriou1995}).
\end{remark}

A consequence of Corollary \ref{cor:lower:bound} below is that Definition \ref{defn:oracle:inequality} is justified and universal subject to mild weak dependence conditions only.

We are now ready to state our main conditions. Throughout this paper, we impose two sets of assumptions on the process $(X_t)_{t \in \zz}$ and the maximal order $K_n$. The first set \ref{G1}-\ref{G3} consists of general assumptions that ensure the well-posedness of the problem, and are therefore the (standard) minimal requirements in the literature.

\begin{ass}\label{ass:main}
Let $(X_t)_{t \in \zz}$ be a stationary process. Then
\begin{enumerate}[label=(G\arabic*),ref=(G\arabic*),leftmargin=*,align=left]
\item \label{G1} $X_t$ does not degenerate to a finite order autoregressive process. 
\item \label{G2} The spectral density $f_X$ is (uniformly) bounded away from zero.
\item \label{G3} $K_{n}\in \{1,\dots,n-1\}$ is a divergent sequence of integers, and there is $\kappa>0$ such that $K_{n}^{2+\kappa}/n$ is bounded.
\end{enumerate}
\end{ass}

Let us briefly discuss these conditions. Note first that \ref{G1} and \ref{G2} are only based on second-order properties (the ACF) of the process $(X_t)_{t \in \zz}$. While we make use of strong stationarity in the proofs, this only has notational reasons and (significantly) increases readability. It is therefore straightforward to adapt the proofs subject only to weak stationarity (and Assumption \ref{ass:weak:dep} below). In addition,  one may actually show that stationarity can be slightly relaxed to a quenched (and hence non-stationary setup). Regarding specifically \ref{G1}, we recall that non-degeneracy in terms of the order is necessary. It is well-known that in the degenerate case of a finite order autoregressive process, the BIC (and related criteria) are optimal, but the AIC is not. Note that we can express non-degeneracy in terms of $\|a(k) - a\|_R > 0$ for any $k \in \nn$, hence the bias term never vanishes, but gets arbitrarily small as $k$ increases. Similarly to \ref{G1}, \ref{G2} is absolutely standard in the literature and rules out certain degenerate processes. We refer to \cite{bradley:BJ:2002} for discussions and equivalent conditions. Finally, condition \ref{G3} imposes a very mild growth constraint on the maximally fitted order, and is standard in the literature \cite{shibata, ing-wei-main}. 
Note that this assumption in particular implies that $K_{n}^{2}/n \to 0$, which was used in \cite{shibata}. 
Furthermore, it implies that $K_{n}^{2+\kappa}/N$ is bounded for $N=n - K_{n}$, as $n/N \to 1$. 
For further comparison with classical assumptions, we refer to Remark \ref{rem:compare-assumptions}.
By modifying (regularizing) the estimators a bit, it appears that one may relax \ref{G3}.
However, such an endeavor is not so relevant (and beyond the scope of this work), since it turns out the selected order is significantly less than $K_n$, see \eqref{eq:Ln:equality} and the attached discussion below, as well as the results in Figure \ref{fig:m-dep} for a numerical illustration.\\
\\
The second set of assumptions lays out the framework for our universality results. They replace the independence conditions imposed on the innovations in the literature so far (e.g. \cite{ing-wei-main}, \cite{karagrigoriou2001}, \cite{shibata}).  


\begin{ass}\label{ass:weak:dep} 
Let $(X_t)_{t \in \zz}$ be a stationary process, $q > 8$, and $\alpha \geq 5/2$. Then
\begin{enumerate}[label=(I\arabic*),ref=(I\arabic*),leftmargin=*,align=left]
\item \label{I1} $\E(X_t) = 0$ and $X_t \in L^q$.
\item \label{I2} $X_{t}$ satisfies $D^X_q(\alpha) < \infty$.
\end{enumerate}
\end{ass}

As already explained at the end of Section \ref{sec:prelim}, conditions such as those given in Assumption \ref{ass:weak:dep} above a very general, and in many cases also rather easy to verify, see Section \ref{sec:examples} for some prominent examples. In particular, they constitute a very significant generalization compared to the linear process setup considered in the literature in the context of efficiency results. 
Moreover, to the best of our knowledge, there are no readily available methods for verifying whether the innovations of a process $X$ are truly independent or not. 
Put differently, there does not seem to be any straightforward, reliable, empirical way to verify (or falsify) the assumptions of the classical results.
The attractive interplay between AR models and physical dependence has already been observed in \cite{wang-politis}, where it has been used to establish consistent estimates for large covariance matrices. 
This provided a non-linear generalization of a seminal paper by K. Berk \cite{berk} on consistent, autoregressive spectral estimates.
It is well known, that autoregressive, spectral estimates and asymptotic efficiency of AIC (and its variants) are closely related. 
In \cite{shibata:1981:AoS:spec:density} R. Shibata used his previously established efficiency results \cite{shibata} to derive optimal autoregressive spectral estimates. 
The present work, together with \cite{wang-politis} show that this connection persists in the non-linear case.

\subsection{Main results}\label{sec:indep}

The main goal of this section is to establish sharp oracle inequalities regarding the efficiency of AIC and its variants subject only to the conditions presented in Section \ref{sec:asympt:efficiency:weak:dep} above, that is, Assumption \ref{ass:main} and Assumption \ref{ass:weak:dep}.

As our first main result, we establish that the empirical risk $Q_n(k)$ (defined in \eqref{defn:Q_n(k):first}) is an optimal finite sample approximation of $L_n(k)$ (defined in \eqref{defn:L_n(k)}), and thus captures the complexity of the prediction problem.

\begin{theorem}\label{thm:main}
  Grant Assumption \ref{ass:main} and Assumption \ref{ass:weak:dep}. Then for every $b>0$, there exists a constant $C>0$ (depending on $b$), such that for every $0 < \delta < \min\{1/4 - 1/(2p),\kappa/8\}$, we have 
\begin{align}
  \bigl|Q_n(k) - L_n(k)\bigr| \leq b L_n(k)/(k_n^{*})^{\delta}, \quad 1 \leq k \leq K_n,
\end{align}
with probability at least $1-C(k_{n}^{*})^{-\gamma}$, where $\gamma = \min\{p/2-1-2 \delta p, 2 \delta p\}>0$.
\end{theorem} 

The proof of this result is involved and lengthy, and is split into many different parts, see Section \ref{sec:proof:thm:main} for more details.

Based on Theorem \ref{thm:main}, we are now in a position to obtain our first result regarding optimality in terms of a lower bound, justifying in particular Definition \ref{def:asymp:efficiency}.

\begin{corollary}\label{cor:lower:bound}
Let $\tilde{k}_n \in \nn$ be any (stochastic) order selection bounded by $K_n$, measurable with respect to $\mathcal{X}_n =\sigma(X_1,\ldots, X_n)$. Then, subject to the conditions of Theorem \ref{thm:main}, for any $\delta > 0$, we have
\begin{align*}
\lim_{n \to \infty}\P\Big(\frac{Q_n(\tilde{k}_n)}{L_n(k_n^*)} \geq 1 - \delta \Big) = 1,
\end{align*}
where $k_n^* \in \nn$ is any minimizer of $L_n(\cdot)$.
\end{corollary}

Recall condition \ref{G3}, which implies $K_n/\sqrt{n} = o(1)$, and is standard in the literature. While heuristically clear, in the present setup, we cannot per se rule out that $Q_n$ is potentially minimized for some value bigger than $K_n$. On the other hand, subject to Assumptions \ref{ass:main} and \ref{ass:weak:dep}, we have
\begin{align}\label{eq:Ln:equality}
\min_{1 \leq k \leq K_n} L_n(k) = \min_{k \in \nn}L_n(k).  
\end{align}
Indeed, due to Lemma \ref{lemma:basic-props-AR(oo)}, it follows that $\|a - a(k)\|_R \lesssim k^{-\alpha}$, and hence we have
\begin{align}
L_n(k) \leq \sigma^2\frac{k}{N} + C k^{-2\alpha}, \quad C > 0.
\end{align}
This bound is familiar from classic nonparametric statistics and selecting $k = N^{1/(2\alpha +1)}$, we get 
\begin{align}
L_n(k) \lesssim N^{-2\alpha/(2\alpha +1)} \asymp n^{-2\alpha/(2\alpha +1)}. 
\end{align}
Since $L_n(k) \geq \sigma^2 k/N$, \eqref{eq:Ln:equality} follows. While this does not prove that $Q_n$ is always minimized at an order $k \leq K_n$ (with high probability), it provides decent hints that this is indeed the case.

We now turn to our ultimate goal, namely showing asymptotic efficiency for the AIC, FPE, and more. To this end, we adopt Shibata's approach in \cite{shibata}, Section 4. First, we pick a model order $\hat{k}_n$ as a minimizer of 
\begin{align}\label{defn:S_n:model:selection}
S_n(k) = (N+2k)\hat{\sigma}_k^2,
\end{align}
where the empirical variance estimator $\hat{\sigma}_k^2$ is defined as
\begin{align}\label{defn:hat:sigma}
    \hat{\sigma}_k^2 = \frac{1}{N} \sum_{t=K_n+1}^n (X_t - \hat{a}_1(k)X_{t-1} - \cdots - \hat{a}_k(k) X_{t-k})^2.
\end{align}

We wish to show that the sequence $\hat{k}_n$ is asymptotically efficient in the sense of Definition \ref{def:asymp:efficiency}. Once this has been established, we may adapt Shibata's perturbation argument to establish analogous results for the AIC, FPE, and related, see Theorem \ref{thm:irs-stable} below.

\begin{theorem}[Universality I] \label{thm:irs-asymp-eff} Grant Assumption \ref{ass:main} and Assumption \ref{ass:weak:dep}. Then any sequence of minimizers $\hat{k}_n$ of $S_n(k) = (N+2k)\hat{\sigma}_k^2$ satisfies
      \begin{equation*}
        |Q_{n}(\hat{k}_{n}) - L_{n}(k_{n}^{*})| \leq 8(k_{n}^{*})^{-\delta}L_{n}(k_{n}^{*}),
      \end{equation*}
      with probability at least $1 - C(k_{n}^{*})^{-\gamma}$, for some $C,\delta,\gamma>0$, where $k_n^* \in \nn$ is any minimizer of $L_n(\cdot)$.
\end{theorem}

As was already shown by Shibata in \cite{shibata} (cf. Theorem 4.2), Theorem \ref{thm:irs-asymp-eff} together with the following perturbation result provides asymptotic efficiency for a big number of model selection criteria.

\begin{theorem}[Universality II]\label{thm:irs-stable} 
Grant Assumption \ref{ass:main} and Assumption \ref{ass:weak:dep} and a sequence of (potentially random) functions $\rho_n:\nn \to \rr$. If there is a $\delta_{0}>0$, such that for every $0 < \delta< \delta_{0}$ and $b>0$ there are constants $C_{1},C_{2}, \gamma_{1},\gamma_{2}$ (depending on $\delta$ and $b$) with
 \begin{equation*}
   \begin{split}
     \mathbb{P}\biggl(\max_{1\leq k \leq K_{n}}\frac{|\rho_{n}(k)|}{N}> b(k_{n}^{*})^{-\delta}\biggr) &\leq C_{1}(k_{n}^{*})^{-\gamma_{1}}, \quad \text{and}\\
     \mathbb{P}\biggl(\max_{1\leq k\leq K_{n}}\frac{|\rho_{n}(k) - \rho_{n}(k_{n}^{*})|}{NL_{n}(k)} > b(k_{n}^{*})^{-\delta}\biggr) &\leq C_{2}(k_{n}^{*})^{-\gamma_{2}},
   \end{split}
 \end{equation*}
 then any sequence $\hat{k}_{n}(\rho)$ of minimizers of $S_n^\rho(k) = (N + \rho_n(k) + 2k)\hat{\sigma}^2_k$ satisfies 
 \begin{equation*}
   |Q_{n}(\hat{k}_{n}(\rho)) - L_{n}(k_{n}^{*})| \leq 8(k_{n}^{*})^{-\delta}L_{n}(k_{n}^{*}),
 \end{equation*}
 with probability at least $1- C(k_{n}^{*})^{-\gamma}$, for some $C,\gamma>0$.
\end{theorem}

  \begin{remark}\label{rem:compare-assumptions}
  Theorem \ref{thm:main} implies 
  \begin{equation}\label{eq:max-ratio}
    \max_{1\leq k\leq K_{n}}\biggl|\frac{Q_{n}(k)}{L_{n}(k_{n}^{*})} - 1\biggr| \overset{\mathbb{P}}{\longrightarrow} 0.
  \end{equation}
  By a closer inspection of our proofs, one can easily see that \eqref{eq:max-ratio} does not require the existence of a $\kappa>0$ such that $K_{n}^{2+\kappa}/n$ is bounded; it is enough to assume that $K_{n}^{2}/n\to 0$.
  This can be thought of as a very general analog of Theorem 3.1, Theorem 3.2, and in particular Proposition 3.2 in Shibata's seminal work \cite{shibata}.
  Thus, if one is contend with classical asymptotic efficiency in the sense of \cite{shibata, ing-wei-main}, Assumption \ref{G3} in Corollary \ref{cor:lower:bound},  Theorem \ref{thm:irs-asymp-eff}, and Theorem \ref{thm:irs-stable}, may be relaxed slightly to $K_{n}^{2}/N \to 0$. In this case AIC and its variants are asymptotically efficient, but we do not get explicit (oracle) rates.
\end{remark}

As an application, consider now the perturbations 
\begin{align}\nonumber
&S_{n}^{*}(k) = (n + 2k)\hat{\sigma}^{2}_{k}, \quad S_{n}^{\mathrm{AIC}}(k) = n e^{2k/n}\hat{\sigma}^{2}_{k},\\
&S_{n}^{\mathrm{FPE}}(k) = \big(n(n + k)/(n-k)\big) \hat{\sigma}^{2}_{k}.
\end{align}

It is easy to see that these model selection criteria satisfy the assumptions of Theorem \ref{thm:irs-stable}. For example, in the case of AIC, $\rho_{n}(k) = n\exp((2k)/n) - N - 2k$, and a simple Taylor expansion yields
\begin{equation*}
  \biggl| e^{\frac{2k}{n}} - 1 -  \frac{2k}{n}\biggr| \leq \frac{e^{\frac{2k}{n}}k^{2}}{2n^{2}}.
\end{equation*}
This immediately implies
\begin{equation*}
  \begin{split}
    \max_{1\leq k \leq K_{n}}\frac{|\rho_{n}(k)|}{N} &\leq \frac{K_{n}}{N}+ \frac{2n}{N}e^{\frac{2K_{n}}{n}}\biggl(\frac{K_{n}}{N}\biggr)^{2}, \quad \text{as well as}\\
    \max_{1\leq k\leq K_{n}}\frac{|\rho_{n}(k) - \rho_{n}(k_{n}^{*})|}{NL_{n}(k)} &\leq \frac{2n}{\sigma^{2}N} \biggl(e^{\frac{2K_{n}}{n}} \frac{K_{n}}{n} + e^{\frac{2k_{n}^{*}}{n}}\frac{k_{n}^{*}}{n}\biggr),
  \end{split}
\end{equation*}
and the assumptions of Theorem \ref{thm:irs-stable} are satisfied with $\delta_{0}= 1$.
Similar arguments apply in the case of $S_{n}^{\mathrm{FPE}}$ and $S_{n}^{*}$.

Note that $S_n^*$ is just the 'natural' version of $S_n$ in \eqref{defn:S_n:model:selection}, where $N$ is replaced with $n$. With some effort, this can also be done for $\hat{\sigma}_k^2$ in \eqref{defn:hat:sigma}. 
By Theorem \ref{thm:irs-stable}, these are all (universal) asymptotically efficient order selection criteria. 

Regarding the classical AIC, recall that in connection with an AR($\infty$)-process the selected order $k_n^{\mathrm{AIC}}$ is usually defined as
\begin{align}
k_n^{\mathrm{AIC}} = \operatorname{argmin}_{k \in \nn}\big(n \log \hat{\sigma}_k^2 + 2k\big).
\end{align}

By a $\log$-transformation, we now immediately obtain asymptotic efficiency for $k_n^{\mathrm{AIC}}$. A similar reasoning applies to the FPE and the corresponding selected order $k_n^{\mathrm{FPE}}$, and we thus have the following result.

\begin{corollary}\label{cor:aic:fpe:efficient}
Grant the assumptions of Theorem \ref{thm:irs-asymp-eff}. Then both $k_n^{\mathrm{AIC}}$ and $k_n^{\mathrm{FPE}}$ are both (universal) asymptotically efficient.    
\end{corollary}

Finally, as mentioned earlier in the introduction, it is possible to generalize many more results based on our approach. Using the tools developed in Section \ref{sec:proofs}, we may, for instance, generalize the results presented in \cite{hannan:quinn:1979:JrssB}, \cite{Jirak:2012:AoS}, \cite{jirak:2014:jmva}  or \cite{shibata:1981:AoS:spec:density}. To exemplify this further, we present a CLT for the coefficient vector $\hat{a}(k)$, $k$ finite.

\begin{theorem}\label{thm:normal:distribution}
  Given Assumptions \ref{ass:main} and \ref{ass:weak:dep}, we have for any fixed $k \in \nn$
\begin{align*}
\sqrt{n}\big(\hat{a}(k) - a(k)\big) \xrightarrow[]{w} \mathcal{N}\big(0,\Sigma(k)\big),
\end{align*}
  where $\Sigma(k) = \sum_{h\in \mathbb{Z}} R(k)^{-1} \mathbb{E}(e_{0}e_{h}X_{0}(k)X_{h}(k)^{T})R(k)^{-1}$.

\end{theorem}

\section{Examples}\label{sec:examples}

Throughout this section, we only focus on validating Assumption \ref{ass:weak:dep}. As already previously mentioned, our setup contains a huge variety of prominent dynamical systems and time series models. 
For the following exposition, we heavily borrow from \cite{berkes2014}, \cite{jirak:jems}, \cite{jirak:tams:2021}, \cite{wu2011asymptotic}.  

Before discussing the actual examples, let us briefly touch on a useful property of our setup regarding functionals of the underlying sequence. To be more specific, let $(\mathbb{Y},\mathrm{d})$ be a metric space. In many cases, if $X_t = f(Y_t)$ for $Y_t \in \mathbb{Y}$ and $f: \mathbb{Y} \to \R$, it is easier to control $\mathrm{d}(Y_t,Y_t')$ rather than directly $|X_t-X_t'|$. Of course, this is only useful if the function $f$ is 'nice enough', allowing for a transfer of the rate. More generally, for any finite $d$, consider $\mathbb{Y}^{d}$ equipped with $\mathrm{d}(x,y) = \sum_{k = 1}^{d} \mathrm{d}(x_k,y_k)$, where $x = (x_1,\ldots,x_{d})$ for $x \in \mathbb{Y}^{d}$ and likewise for $y$. Let $\mathcal{G}(C_f,\ad,\bd)$ be the class of functions $f$ that satisfy $f:\mathbb{Y}^{d} \to \R$ with
\begin{align*}
\big|f(x) - f(y)\big| \leq C_f \big(\mathrm{d}(x,y)^{\ad} \wedge 1 \big) \big(1 + \mathrm{d}(x,0) + \mathrm{d}(y,0)\big)^{\bd},
\end{align*}
where $C_f, \bd \geq 0$, $0 < \ad \leq 1$ and $0 \in \mathbb{Y}^{d}$ some fixed point of reference. Define $X_t$ by
\begin{align}\label{ex:gen:g:function}
X_t = f\big(Y_t,Y_{t-1}, \ldots, Y_{t-d+1}\big) - \E f\big(Y_t,Y_{t-1}, \ldots, Y_{t-d+1}\big).
\end{align}
If $q \geq 1 \vee p(\ad + \bd)$ and $\E \mathrm{d}^q(Y_t,0) < \infty$, then straightforward computations reveal the following result.

\begin{prop}\label{prop:function}
Given \eqref{ex:gen:g:function}, there exists $C > 0$ such that
\begin{align}
\sup_{t \in \zz} \big\|X_t - X_t^{(t-l)}\big\|_{p} \leq C \sup_{t \in \zz} \big( \E \mathrm{d}^{q}(Y_t, Y_t^{(t-l)})\big)^{\ad/q}.
\end{align}
\end{prop}

\begin{remark}
Observe that if $\sup_{t \in \zz} \|X_t - X_t^{(t-l)}\|_{p}$ has exponential decay, one may select $\ad > 0$ arbitrarily small and still maintain exponential decay.
\end{remark}

Armed with Proposition \ref{prop:function}, we are now ready to discuss some examples.

\subsection{Banach space valued linear processes}\label{ex:banach:linear}

Suppose that $\mathbb{Y} = \mathds{B}$ is a separable Banach space with norm $\|\cdot\|_{\mathds{B}}$. Slightly abusing notation, we write $\|X\|_p = \| \|X\|_{\mathds{B}}\|_p $ for the Orlicz-norm for a random variable $X \in \mathds{B}$. 
Let $(A_i)_{i \in \nn}$ be a sequence of linear operators $A_i : \mathds{B} \to \mathds{B}$, and denote with
$\Vert A_i \Vert_{\mathrm{op}}$ the corresponding operator norm. For an i.i.d. sequence $(\varepsilon_t)_{t \in \zz} \in \mathds{B}$, consider the linear process
\begin{align*}
Y_t = \sum_{i = 0}^{\infty} A_i \varepsilon_{t-i}, \quad k \in \zz,
\end{align*}
which exists if $\|\varepsilon_0\|_{1} < \infty$ and $\sum_{i \in \nn} \|A_i\|_{\mathrm{op}} < \infty$, which we assume from now on. Recall that autoregressive processes (even of infinite order) can typically be expressed as linear processes, in particular the (famous) dynamical system 2x mod 1 (Bernoulli convolution, doubling map). For the latter, we refer to Example 3.2 in ~\cite{jirak_be_aop_2016} and the references therein for more details. Obviously, we have the bound

\begin{align*}
\big\|Y_t - Y_t'\big\|_p \leq 2\big\|\varepsilon_t\big\|_p \big\|A_t\big\|_{\mathrm{op}}.
\end{align*}

Suppose that
\begin{align*}
\big\|\varepsilon_t\big\|_{q} < \infty, \quad \sum_{l = 1}^{\infty} l^{\alpha} \big\|A_l\big\|_{\mathrm{op}}^{\ad} < \infty, \quad \alpha \geq 5/2,
\end{align*}
for $q \geq 1 \vee p (\ad + \bd)$, $p > 8$. Then clearly $X_t = f(Y_t) - \E f(Y_t)$ with $f \in \mathcal{G}(C_f,\ad,\bd)$ satisfies Assumption \ref{ass:weak:dep}.

\subsection{SDEs}\label{sec:ex:sde}

Consider the following stochastic differential equation on $\mathbb{Y} = \R^d$ equipped with the Euclidian norm $\| \cdot \|_{\R^d}$:
\begin{align}\label{eq:ex:sde:1}
dY_t = a(Y_t) dt +\sqrt{2b(Y_t)} dB_t, \quad Y_0 = \xi,
\end{align}
where $(B_t)_{t\geq0}$ is a standard Brownian motion in $\R^d$, and the functions $a:\R^d \to \R^d$
and $b: \R^d \to \R^{d \times d}$ satisfy the following conditions. For a given matrix
$A$, we define the Hilbert–Schmidt norm $\|A\|_{\mathrm{HS}} = \sqrt{\operatorname{tr}(AA^{\top})}$.
We assume that the following stability condition (cf. \cite{Djellout:AoP:2004}) is satisfied: the functions $a$ and $b$ are Lipschitz continuous, and there exists $\gamma > 0$ such that
\begin{align}\label{ass:sde}
\|b(x) - b(y)\|_{\mathrm{HS}}^2 + \langle x-y,a(x)-a(y) \rangle \leq - \gamma\|x - y\|_{\R^d}^2, \quad x,y \in \R^d.
\end{align}
Regarding the existence of a (strong and pathwise unique) solution, we refer to \cite{Djellout:AoP:2004} Section 4 or \cite{karatzas_shreve_1991}, Chapter 5. Now, for any fixed $\delta > 0$ and $t/\delta \in \nn$, let
$\mathcal{I}_i = ((i-1)\delta, i\delta]$ for $i \geq 1$. We may thus write $Y_t$ as
\begin{align*}
Y_t = g_t\big((B_s - B_{t - \delta})_{s \in \mathcal{I}_{t/\delta}}, (B_s - B_{t - 2\delta})_{s \in \mathcal{I}_{t/\delta - 1}}, \ldots, (B_s)_{s \in \mathcal{I}_{1}}, \xi \big),
\end{align*}
that is, as a map from $\big(C[0,\delta)\big)^{{t/\delta}} \times \R \to \R^d$, where $C[0,\delta)$ denotes the space of continuous functions mapping from $[0,\delta)$ to $\R$. By properties of the Brownian motion,
we can thus set $\varepsilon_i = (B_s - B_{(i-1)\delta})_{s \in \mathcal{I}_{i}}$.

For simplicity, we assume from now on that the initial condition $Y_0 = \xi$ admits a stationary solution $Y_t$. We then have the following result (see Proposition 3.4 in \cite{jirak:jems}).

\begin{prop}\label{prop:sde}
Grant Assumption \ref{ass:sde}, and assume the (stationary) solution $(Y_t)_{t \geq 0}$ satisfies $\E \|Y_t\|_{\R^d}^q < \infty$, $q > 2$. Pick any $\delta > 0$. Then there exist $C, c > 0$, such that for any $2 \leq p < q$, we have
\begin{align*}
\sup_{t \in \zz}\E \big\|Y_t - Y_t^{(t-l)}\big\|_{\R^d}^p \leq C \exp(-c l), \quad t/\delta \in \nn.
\end{align*}
\end{prop}


Suppose now that $q \geq 1 \vee p \bd$, $p > 8$. Then $X_t = f(Y_t) - \E f(Y_t)$ with $f \in \mathcal{G}(C_f,\ad,\bd)$ satisfies Assumption \ref{ass:weak:dep} for any $\ad > 0$. Note that simple conditions for the existence of $\sup_{t \geq 0}\E \exp(-c \|Y_t\|_{\R^d}) < \infty$, $c > 0$, are provided for instance in \cite{Djellout:AoP:2004}. Moreover, discrete-time analogs of \eqref{eq:ex:sde:1} naturally also meet our conditions, see for instance \cite{wu2011asymptotic}.

\subsection[Left random walk on GL(R)]{Left random walk on $GL_d(\R)$}\label{ex:randomwalk}

Cocycles, in particular the random walk on $GL_d(\R)$, have been heavily investigated in the literature, see e.g. \cite{bougerol_book_1985}, \cite{benoist:quint:book:2016} and \cite{benoist2016}, \cite{cuny2017}, \cite{CUNY20181347}, \cite{cuny:2023:aop}, \cite{ion:hui:jems:2022} for some more recent results. In this example, we will particularly exploit ideas of \cite{cuny2017}, \cite{CUNY20181347}. Let $(\varepsilon_k)_{k \in \nn}$ be independent random matrices taking values in $G = GL_d(\R)$, with common distribution $\mu$. Let $A_0 = \mathrm{Id}$, and for every $n \in \nn$, $A_n = \prod_{i = 1}^n \varepsilon_i$. Denote with $\|\cdot\|_{\R^d}$ the Euclidean norm on $\R^d$, and likewise $\|g\|_{\R^d} = \sup_{\|x\|_{\R^d} = 1} \|g x\|_{\R^d}$ the induced operator norm. We adopt the usual convention that $\mu$ has a moment of order $p$, if
\begin{align*}
\int_G \big(\log N(g) \big)^p \mu(d g) < \infty, \quad N(g) = \max\big\{\|g\|_{\R^d}, \|g^{-1}\|_{\R^d}\big\}.
\end{align*}
Let $\mathds{X} = P_{d-1}(\R^d)$ be the projective space of $\R^d\setminus\{0\}$, and write $\overline{x}$ for the projection from $\R^d\setminus\{0\}$ to $\mathds{X}$. We assume that $\mu$ is strongly irreducible and proximal, and hence there exists a corresponding unique, stationary measure $\nu$ on $\mathds{X}$, see \cite{benoist:quint:book:2016} for details. The left random walk of law $\nu$ started at $\overline{x} \in \mathds{X}$ (distributed according to $\nu$) is the Markov chain given by $Y_{0\overline{x}} = \overline{x}$, $Y_{k\overline{x}} = \varepsilon_k Y_{k-1\overline{x}}$ for $k \in \nn$. Following the usual setup, we consider the associated random variables $(X_{k\overline{x}})_{k \in \nn}$, given by
\begin{align*}
X_{k\overline{x}} = h\big(\varepsilon_k, Y_{k-1\overline{x}} \big), \quad h\big(g, {z}\big) = \log \frac{\| g z\|_{\R^d}}{\|z\|_{\R^d}},
\end{align*}
for $g \in G$ and $z \in \R^d\setminus\{0\}$. Following ~\cite{CUNY20181347}, if $p > (2+\ad)q + 1$, then Proposition 3 in ~\cite{cuny2017} implies
\begin{align*}
\sum_{k = 1}^{\infty} k^{\ad} \sup_{l  \geq k} \big\|X_{k\overline{x}}^{(k-l)} - X_{k\overline{x}}^{(k-l)} \big\|_q < \infty.
\end{align*}

Hence, if $p > 37$, Assumption \ref{ass:weak:dep} holds. Finally, let us mention that an analogous claim can be made regarding the random walk on positive matrices, see \cite{cuny:positive:matrix:2023}, \cite{jirak:jems} for details.

\subsection{Iterative random Systems}\label{sec:ex:iterative}

Let $(\mathbb{Y},\mathrm{d})$ be a complete, separable metric space. An iterated random
function system on the state space $\mathbb{Y}$ is defined as
\begin{align}
Y_t = F_{\varepsilon_t}\big(Y_{t-1} \big), \quad t \in \zz,
\end{align}
where $\varepsilon_t \in \mathbb{S}$ are i.i.d. with $\varepsilon \stackrel{d}{=} \varepsilon_t$, where $\mathbb{S}$ is some measurable space. Here, $F_{\varepsilon}(\cdot) = F(\cdot, \varepsilon)$ is the $\varepsilon$-section of a jointly measurable function $F : \mathbb{Y} \times \mathbb{S} \to \mathbb{Y}$. Many dynamical systems, Markov processes, and non-linear time series are within this framework, see for instance \cite{diaconis_freedman_1999}, \cite{wu_shao_iterated_2004}. For $y \in \mathbb{Y}$, let
$Y_t(y) = F_{\varepsilon_t} \circ F_{\varepsilon_{t-1}} \circ \ldots \circ F_{\varepsilon}(y)$, and, given $y,y' \in \mathbb{Y}$ and $\gamma > 0$, we say that the system is $\gamma$-\textit{moment contracting} if there exists $0 \in \mathbb{Y}$, such that for all $y \in \mathbb{Y}$ and $t \in \zz$
\begin{align}\label{ex:iterated:contraction}
\E \mathrm{d}^{\gamma}\big(Y_t(y),Y_t(0)\big) \leq C \rho^t \mathrm{d}^{\gamma}(y,0), \quad \rho \in (0,1).
\end{align}
We note that slight variations exist in the literature. We now have the following result, which is an almost immediate consequence of Theorem 2 in \cite{wu_shao_iterated_2004}.

\begin{prop}\label{prop:iterated:1}
Assume that \eqref{ex:iterated:contraction} holds for some $\gamma > 0$, and that $\E \mathrm{d}^{\gamma}\big(0,Y_1(0)\big) < \infty$ for some $0 \in \mathbb{Y}$. Then there exists $C > 0$, such that
\begin{align*}
\sup_{t \in \zz}\E \mathrm{d}^{\gamma}\big(Y_t,Y_t^{(t-l)}\big) \leq C \rho^l.
\end{align*}
If $\E \mathrm{d}^{s}\big(0,Y_1(0)\big)$ for $s > p \geq 1$ instead, then even
\begin{align*}
\sup_{t \in \zz}\E \mathrm{d}^p\big(Y_t,Y_t^{(t-l)}\big) \leq C \rho^l.
\end{align*}
\end{prop}

Suppose now that $(Y_t)_{t \in \zz}$ is stationary and satisfies \eqref{ex:iterated:contraction} and $\E \mathrm{d}^{q}\big(0,Y_1(0)\big)<\infty$. Then if $q \geq 1 \vee p \bd $, $p > 8$, it follows that $X_t = f(Y_t) - \E f(Y_t)$ with $f \in \mathcal{G}(C_f,\ad,\bd)$ satisfies Assumption \ref{ass:weak:dep} for any $\ad > 0$.

\subsection[GARCH(p, q) processes]{GARCH$(\mathfrak{p},\mathfrak{q})$ processes}\label{sec:ex:garch}

A very prominent stochastic recursion is the GARCH($\mathfrak{p},\mathfrak{q})$ sequence, given through the relations
\begin{align*}
&Y_t = \varepsilon_t L_{t} \quad \text{where $(\varepsilon_{k})_{k \in \zz}$ is a zero mean i.i.d. sequence and}\\
&L_t^2 = \mu + \alpha_1 L_{t - 1}^2 + \ldots + \alpha_\mathfrak{p} L_{t - \mathfrak{p}}^2 + \beta_1 Y_{t - 1}^2 + \ldots + \beta_{\mathfrak{q}} Y_{t - \mathfrak{q}}^2,
\end{align*}
with $\mu > 0$, $\alpha_1,\ldots,\alpha_\mathfrak{p}, \beta_1,\ldots,\beta_{\mathfrak{q}} \geq 0$. We refer to \cite{garch:book:2010} for some general aspects of Garch processes and their importance. Assume first $\|\varepsilon_t\|_q < \infty$ for some $q \geq 2$. A key quantity here is
\begin{align*}
\gamma_C = \sum_{i = 1}^{r} \bigl\|\alpha_i + \beta_i \varepsilon_i^2\bigr\|_{q/2}, \quad \text{with $r = \max\{\mathfrak{p},\mathfrak{q}\}$},
\end{align*}
where we replace possible undefined $\alpha_i, \beta_i$ with zero. If $\gamma_C < 1$, then $(Y_k)_{t \in \zz}$ is (strictly) stationary. In particular, one can show the representation
\begin{align*}
Y_t = \sqrt{\mu}\varepsilon_t\biggl(1 + \sum_{n = 1}^{\infty} \sum_{1 \leq l_1,\ldots,l_n\leq r} \prod_{i = 1}^n\bigl(\alpha_{l_i} + \beta_{l_i}\varepsilon_{t - l_1 - \ldots - l_i}^2\bigr) \biggr)^{1/2},
\end{align*}
we refer to \cite{aue_etal_2009} for comments and references. Using this representation and the fact that $|x-y|^q \leq |x^2 - y^2|^{q/2}$ for $x,y \geq 0$, one can follow the proof of Theorem 4.2 in \cite{aue_etal_2009} to show that
\begin{align*}
\bigl\|Y_t - Y_t'\bigr\|_q \leq C \rho^{t}, \quad \text{where $0 < \rho < 1$.}
\end{align*}

Suppose now that $\gamma_C < 1$ for $q \geq 1 \vee p \bd$, $p > 8$. Then $X_t = f(Y_t) - \E f(Y_t)$ with $f \in \mathcal{G}(C_f,\ad,\bd)$ satisfies Assumption \ref{ass:weak:dep} for any $\ad > 0$.

We note that analogous results can be shown for augmented Garch processes, see \cite{berkes_2008_letter}, \cite{jirak:EJP}. Also note that from a more general perspective, Garch processes may be regarded as iterative random systems.

\subsection{Markov chains of infinite order}\label{ex:winterberger}



Let $\mathbb{S} = \mathds{B}$ be a Banach space with norm $\|\cdot\|_{\mathds{B}}$. Recall that $\|X\|_p = \| \|X\|_{\mathds{B}}\|_p $ denotes the Orlicz-norm for a random variable $X \in \mathds{B}$.
Consider a sequence $(a_k)_{k \geq 1} \in \R^+$ with $\sum_{k \geq 1} a_k < 1$, and let $(\varepsilon_k)_{k \in \zz} \in \mathds{B}$ be i.i.d. Let $F:\mathds{B}^{\nn} \times \mathds{B} \to \mathds{B}$ be such that

\begin{align}
&\big\|F(x,\varepsilon_0) - F(y,\varepsilon_0) \big\|_{q} \leq \sum_{k = 1}^{\infty} a_k \|x_k - y_k\|_{\mathds{B}},\\
&\big\|F(0,0,\ldots,\varepsilon_0)\big\|_q < \infty, \quad x,y \in \mathds{B}^{\nn}, 
\end{align}
where we write $x = (x_1,x_2,\ldots)$ for $x \in \mathds{B}^{\nn}$ and $0 \in \mathds{B}$ is some point of reference. The Markov chain of infinite order is then (formally) defined as
\begin{align}
Y_k = F\big(Y_{k-1},Y_{k-2},Y_{k-3},\ldots, \varepsilon_k\big).
\end{align}

Existence and further properties are established in \cite{doukhan:winterberger:2008:spa}. In particular, if the sequence $(a_k)_{k \geq 1}$ satisfies
\begin{align*}
\sum_{i \geq k} a_i \leq C_a k^{-\ad}, 
\end{align*}
the results in \cite{doukhan:winterberger:2008:spa} imply that for some constant $C > 0$, we have
\begin{align}
\big\|Y_k - Y_k'\big\|_q \leq C  k^{-\ad' + \delta}, 
\end{align}
where $\delta > 0$ can be selected arbitrarily small. Suppose now $q \geq 1 \vee p (\bd + \ad)$, $p > 8$ and $\ad' > 5/2\ad$. Then $X_t = f(Y_t) - \E f(Y_t)$ with $f \in \mathcal{G}(C_f,\ad,\bd)$ satisfies Assumption \ref{ass:weak:dep}.

\section{Simulations}\label{sec:sim}

    AIC has been a staple in the model selection toolkit ever since its inception \cite{akaike:1974}, and its usefulness has been verified in countless examples.
  In the context of asymptotic efficiency, a simulation study has been conducted for instance in \cite{ing-wei-main}. 
  However, the effect of dependence in this context has -- to the best of our knowledge -- not been fully explored.
  In order to assess the potential impact of weak dependence on the performance of AIC, we consider a sequence of processes $(X_{t}^{(m)})_{t\in \mathbb{Z}}$, with $m$-dependent innovations as $m$ grows.
For an i.i.d. sequence of standard normal variables $\varepsilon_{i}$, we define a sequence of uncorrelated (martingale difference type), $m$-dependent innovations $e_{t}^{(m)}$ by
\begin{equation}\label{eq:sim-innov}
  e^{(m)}_{t} = \varepsilon_t\biggl(\prod_{j=1}^{m-1}|\varepsilon_{t-j}|\biggr)^{\frac{1}{m-1}}.
\end{equation}
The process $X^{(m)}_{t}$ is then given by
\begin{equation}\label{eq:sim-proc}
  X_{t}^{(m)} = e_{t}^{(m)} + \sum_{j=1}^{\infty}j^{-4}e^{(m)}_{t-j}.
\end{equation}
Note that any $m$-dependent process ($m$ fixed) automatically satisfies Assumption \ref{ass:weak:dep} for any $\alpha \geq 0$, given $q > 8$ moments. Taking the $m$-th root in \eqref{eq:sim-innov} ensures that the innovations are of the same magnitude for different $m$. Note that $m=1$ corresponds to the classical setting of \cite{shibata, karagrigoriou2001, ing-wei-main}. On the other hand, since 
\begin{align*}
\biggl(\prod_{j=1}^{m-1}|\varepsilon_{t-j}|\biggr)^{\frac{1}{m-1}} = e^{\E \log |\varepsilon|} + O_{\P}\big(m^{-1/2}\big),    
\end{align*}
it is tempting to expect the 'very large $m$' case to (exactly) mirror the i.i.d. ($m = 1$) setting. However, our choices of $m$ are not large enough for such an effect to take place, see the discussion below.

We compute 4000 independent runs of the process \eqref{eq:sim-proc} for $t=41,\dots, 1840$ and $m=1,2,5,10,15,20$ and $25$.
The process $X_{t}^{(m)}$ in \eqref{eq:sim-proc} is approximated by 
\begin{equation*}
  \tilde{X}_{t}^{(m)} = e_{t}^{(m)} + \sum_{j=1}^{400}j^{-4}e_{t-j}^{(m)},
\end{equation*}
since $\sum_{j=401}^{\infty}j^{-4}\approx 5.2\times10^{-9}$ and $e_{t}^{(m)}$ is strongly concentrated on the intervall $[-1,1]$ (see Figure \ref{fig:sim-innov-dist}). 
We observe in Figure \ref{fig:m-dep} that larger values of $m$ tend to cause a quicker increase in the size of the models AIC selects.
However, qualitatively the curves in Figure \ref{fig:m-dep} look very similar.
For example, moving from a $5$-dependent process to a $25$-dependent process increases the average selected model order by less than a factor of two. This is the titular \textit{universality} of AIC. The model orders have been selected on the \textit{unrestricted} domain $k=1,\dots,n-1$, as this is what most implementations of AIC, and hence most practitioners, do. As has been highlighted in the discussion below Corollary \ref{cor:lower:bound}, even for very moderate sample sizes the selected model orders are significantly smaller than $\sqrt{n}$.

In a second set of simulations, we try to asses how 
the rate of decay of the AR coefficients potentially influences the performance of AIC. To this end, we consider a sequence of moving average coefficients, $b_{j}^{(p)} = j^{-p}$, for $p>1$. The innovations $e_{t}$ follow a GARCH($0.25$,$0.25$) model with standard Gaussian innovations, that is,
\begin{align*}
 e_{t} &= \varepsilon_{t}L_{t},
\end{align*}
where $\varepsilon_{t}$ is a sequence of i.i.d. standard Gaussian random variables, and
\begin{align*}
 L_{t}^{2} &= \frac{1}{10} + \frac{1}{4}L_{t-1}^{2} + \frac{1}{4}e_{t-1}^{2}.    
\end{align*}
The process we consider is given by 
\begin{equation*}
  Y_{t}^{(p)} = e_{t} + \sum_{j=1}^{\infty}b_{j}^{(p)}e_{t-j},
\end{equation*}
which is again approximated by 
\begin{equation*}
\tilde{Y}_{t}^{(p)} = e_{t} + \sum_{j=1}^{400}b_{j}^{(p)}e_{t-j}.
\end{equation*}

Note that by Lemma \ref{lemma:dep-rates} and the results of Section \ref{sec:ex:garch}, the AR-process $Y_{t}$ satisfies $D^{Y}_{q}(r) < \infty$ for any $r<p$. 

As before, we select the model order by AIC for various partial samples $(\tilde{Y}_{s})_{s=1}^{t}$, for $t = 41, \dots, 1840$, and record their average in Figure \ref{fig:garch-sim}.  
Slow decay of autoregressive coefficients tends to cause AIC to select larger models. The results should be contrasted with Figure \ref{fig:m-dep}. 
Slow decay in the AR coefficients has a similar effect on the performance of AIC as a higher degree of dependence among the innovations. 
This is in line with Lemma \ref{lemma:dep-rates}. 
In particular, it should be noted that even for smaller $p$, the average selected model order is well below the threshold of $\sqrt{n}$. Note that for $p < 5/2$, the process $Y_{t}$ may not satisfy Assumption \ref{ass:weak:dep}. It is well known \cite{shibata} that for centered i.i.d. Gaussian innovations absolute summability is enough to guarantee the asymptotic efficiency of the model orders selected by AIC. 
The results in Figure \ref{fig:garch-sim} suggest, that for a subclass of very weakly dependent processes, one may be able to relax Assumption \ref{ass:weak:dep}, and move even closer to the i.i.d. case.
\begin{center}
\begin{figure}
    \centering
    \includegraphics[width=1\textwidth]{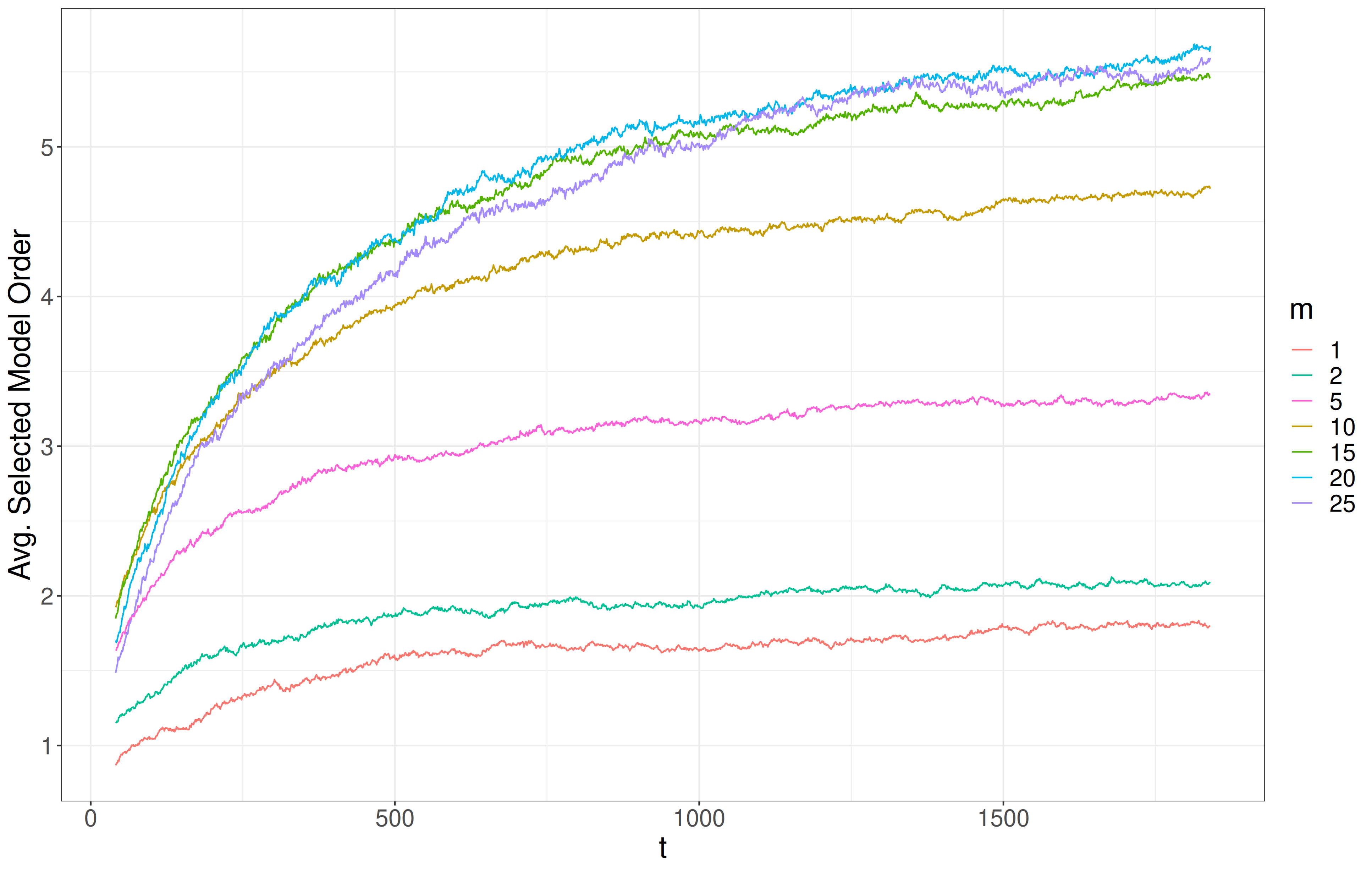}
  \caption{Average model order selected by AIC for $(\tilde{X}_{s}^{(m)})_{s=1}^{t}$ over 4000 runs, for various degrees of dependence $m$ and $t=41,\dots,1840$.}
    \label{fig:m-dep}
\end{figure}
\end{center}

\begin{center}
\begin{figure}
    \centering
    \includegraphics[width=0.6\textwidth]{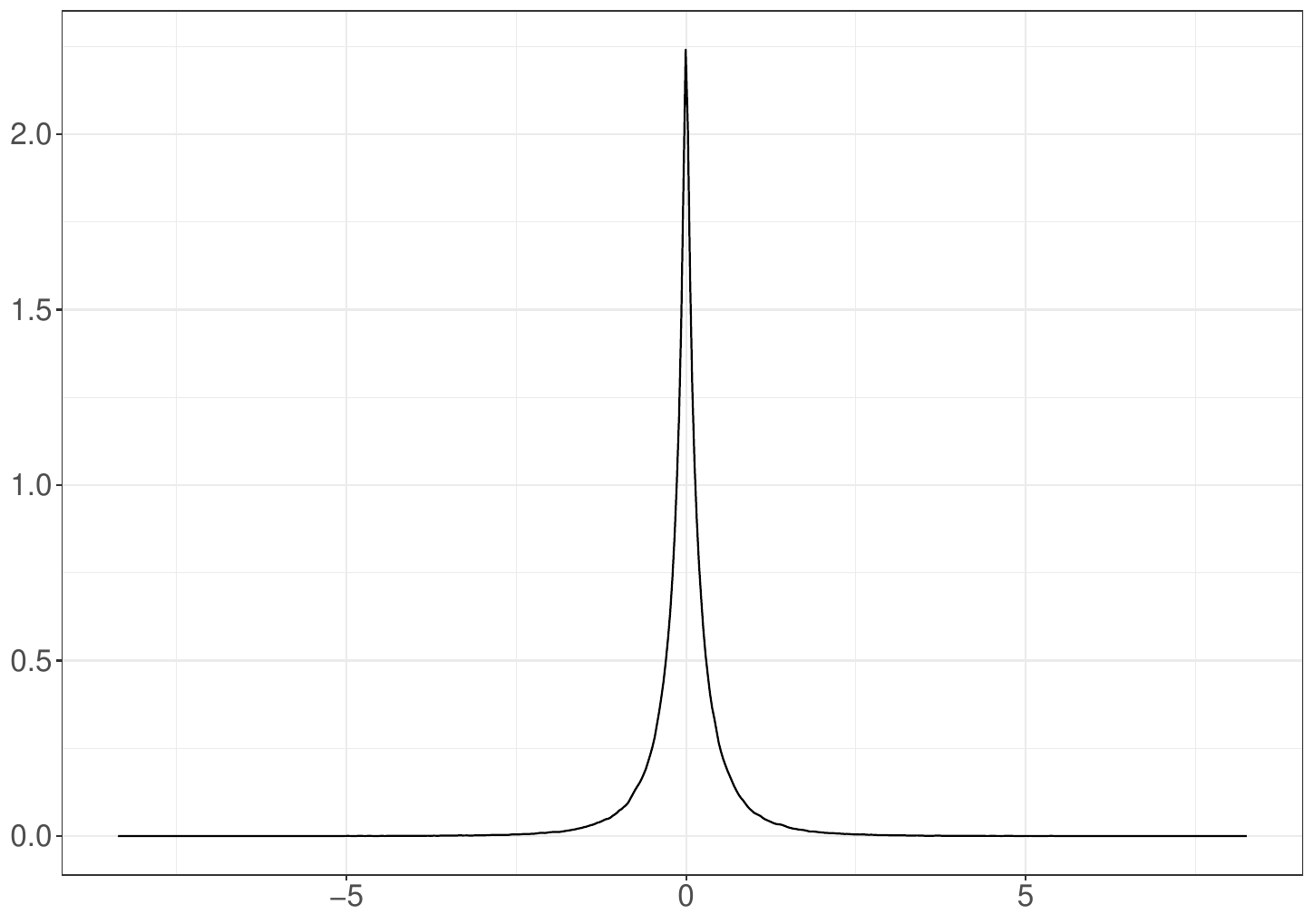}
    \caption{Distribution of $e^{(5)}_{0}$. Most of its mass is concentrated on $[-1,1]$.}
    \label{fig:sim-innov-dist}
\end{figure}
\end{center}

\begin{center}
\begin{figure}
\centering
\includegraphics[width=1\textwidth]{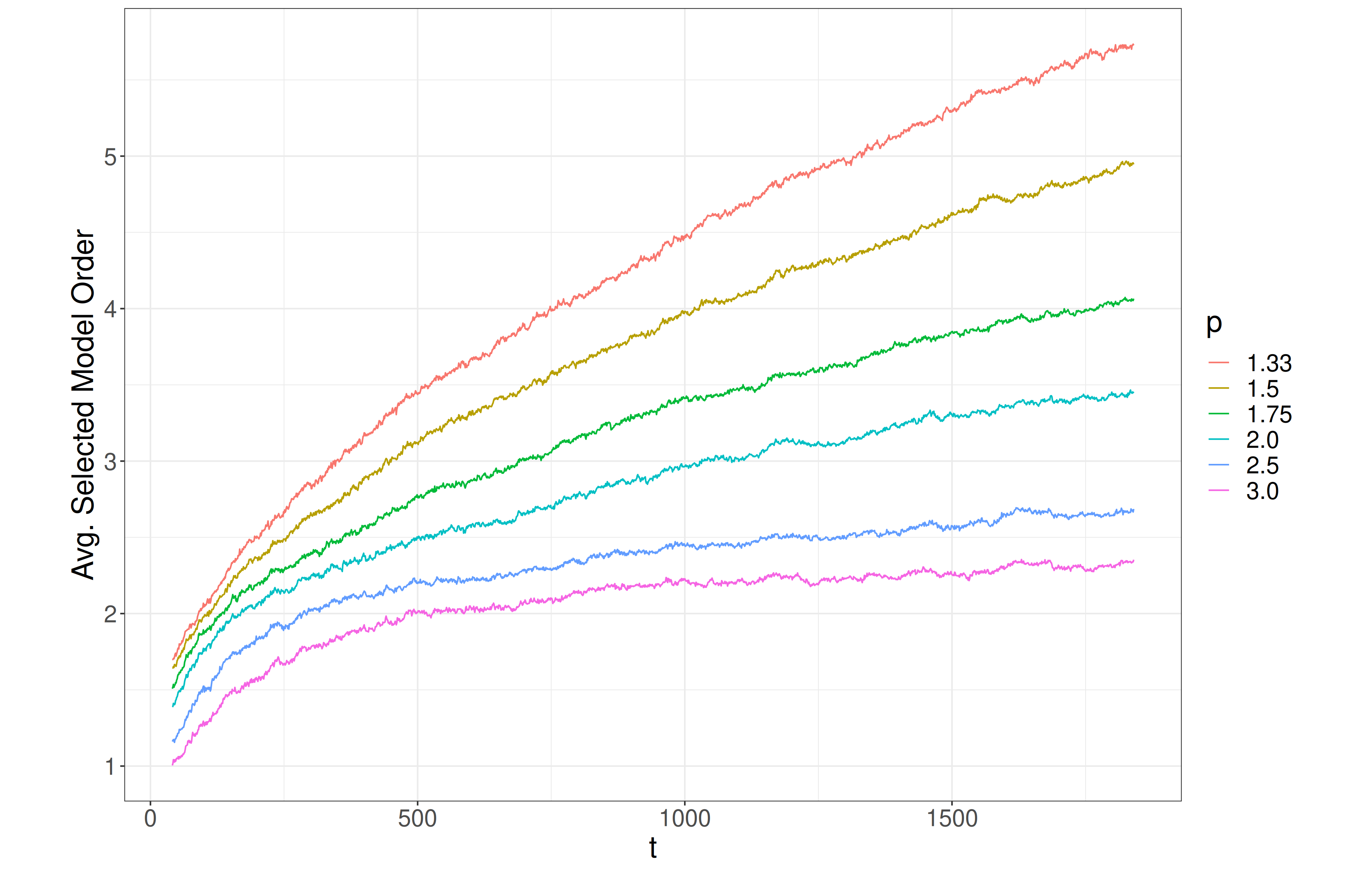}
\caption{Average model orders selected by AIC for $(\tilde{Y}_{s}^{(p)})_{s=1}^{t}$ over 4000 runs, for various AR decay rates $p$ and $t=41,\dots,1840$.}%
\label{fig:garch-sim}
\end{figure}
\end{center}

\section*{Acknowledgements}
Funding by the Austrian Science Fund (FWF) grant I 5485-N is gratefully acknowledged. We would also like to thank D. N. Politis for pointing to the connection between modern autoregressive, spectral estimates and the present work. Furthermore, we would like to thank the participants of the ISOR Privatissimum for providing useful feedback on an early draft. In particular, we would like to thank B. M. Pötscher for suggesting a shortcut in the proof of Lemma \ref{lemma:Bn-Bnk}.

\section{Proofs}\label{sec:proofs}

This section is structured as follows. In Section \ref{sec:pd-ar-props} we introduce the basic techniques that we need for later proofs, and highlight their connections to classical theory of autoregressive processes. In Section \ref{sec:proof:thm:main} we proof Theorem \ref{thm:main}, while Section \ref{sec:proof:thm:irs-asymp_eff} is devoted to the proofs of Theorems \ref{thm:irs-asymp-eff} and \ref{thm:irs-stable}, and Theorem \ref{thm:normal:distribution} is established in Section \ref{sec:proof:normal}.
As a general rule, the proofs in Sections \ref{sec:proof:thm:main} -- \ref{sec:proof:normal} depend on technical lemmas. For the ease of reference, these are formulated in the respective sections where they are first used. 
Sections \ref{sec:2-forms-weakly-proc}, \ref{sec:lemmas-for-main-theorem}, and \ref{sec:lemma-A1} deal with the proofs of these lemmas.

\subsection{Properties of physical dependence and connections to autoregression}\label{sec:pd-ar-props}

In this section, we will briefly touch upon the theory of physically dependent (autoregressive) models, and explain how our assumptions compare to the classical assumptions imposed in the literature. One of the benefits of using physical dependence measures is their amenability to martingale difference approximation arguments. Martingale difference approximations in connection with limit theorems have a very long tradition (cf. \cite{MP:survey} and the references therein) and can be performed on any integrable, $\mathcal{F}_t$ measurable random variable. As a reference, we formulate the following, well-known result (cf. \cite{wu-physical-dependence}, \cite{wu2011asymptotic}).

\begin{lemma}[Martingale difference approximation]\label{lemma:mdsa}
  If $Z \in L^1$ is $\mathcal{F}_t = \sigma(\xi_t)$ measurable (see Section \ref{sec:prelim} for the definition of $\mathcal{F}_{t}$ and $\xi_{t}$), then
  \begin{equation*}
    Z - \mathbb{E}(Z)= \sum_{l=0}^\infty \mathcal{P}_{t-l}(Z) = \sum_{l\in\mathbb{Z}}\mathcal{P}_{t-l}(Z) = \sum_{l\in\mathbb{Z}}\mathcal{P}_{l}(Z),
  \end{equation*}
  where $\mathcal{P}_t(Z) = \mathbb{E}(Z\mid\mathcal{F}_t) - \mathbb{E}(Z\mid\mathcal{F}_{t-1})$.
\end{lemma}

Note that $\mathcal{P}_{l}(Z)$ is a martingale difference sequence. 
If $Z = Z_t = h(\xi_t)$, then we have $\mathbb{E}(Z_t\mid \mathcal{F}_{-1}) = \mathbb{E}(Z_t'\mid \mathcal{F}_{-1}) = \mathbb{E}(Z_t'\mid \mathcal{F}_0)$ and thus 
\begin{equation}\label{eq:P(Z)=E(Z-Z')}
    \mathcal{P}_0(Z_t) = \mathbb{E}(Z_t - Z_t'\mid \mathcal{F}_0).
\end{equation}
Hence we can control the $L^q$ norm of the increments $\mathcal{P}_0(Z_t)$ by $\|\mathcal{P}_0(Z_t)\|_q \leq \delta_q^Z(t)$, if the process is physically dependent.
In the proofs themselves, we frequently use that $\mathcal{P}_{t-l}(Z_{t})\disteq \mathcal{P}_{0}(Z_{l})$, since $Z_{t} = h(\xi_{t})$.  
If the function $h$ depends on $t$ as well, i.e., if $Z_{t} = h_{t}(\xi_{t})$, then we have the estimate
\begin{equation*}
  \|\mathcal{P}_{t-l}(Z_{t})\|_{q} \leq \delta_{q}^{Z}(l) = \sup_{t\in \mathbb{Z}}\|g_{t}(\xi_{t}) - g_{t}(\xi_{t}^{(l)})\|_{q}.
\end{equation*}
In both cases, the physical dependence coefficients control the $L^{q}$ norm of the projections $\mathcal{P}_{t-l}(Z_{t})$.


As we shall now discuss, physical dependence allows for a very tight control of the autocovariance function (ACF). The ACF inherits the rate of decay of $\delta_{2}^{X}(l)$ up to non-decreasing functions. The following result is well-known in the literature. For the sake of completeness, we provide the short proof.

\begin{lemma}\label{lemma:rate-acf}
    If $Z_t$ is a stationary process and $g:\mathbb{N} \to [0,\infty)$ is a non-decreasing function such that 
    \begin{equation*}
        \sum_{l=0}^\infty g(l) \delta_2^Z(l) < \infty,
    \end{equation*}
    then the ACF $\gamma_Z(h)$ of $(Z_t)$ satisfies
    \begin{equation*}
        \sum_{h=0}^\infty g(h) |\gamma_Z(h)| < \infty.
    \end{equation*}
\end{lemma} 
\begin{proof}
The statement is trivial if $g$ is identical to zero. Otherwise, there is a minimal index $n_0$, such that $0 < g(n_0) \leq g(n)$ for all $n\geq n_0$. Hence 
    \begin{equation*}
        \sum_{l=n_0}^\infty \delta_2^Z(l) \leq \frac{1}{g(n_0)} \sum_{l=n_0}^\infty g(l) \delta_2^Z(l) < \infty,
    \end{equation*}
    \textit{i.e.}, $D_2^Z(0) < \infty$. Using Lemma \ref{lemma:mdsa}, we get 
    \begin{equation}\label{eq:acf1}
        \begin{split}
            \sum_{h=0}^\infty g(h) |\mathbb{E}(Z_h Z_0)| & \leq \sum_{h=0}^\infty g(h) \sum_{l=0}^\infty |\mathbb{E}(\mathcal{P}_{l-h}(Z_h)Z_0)|.
        \end{split}
    \end{equation}
    Since $\mathbb{E}(\mathcal{P}_{l-h}(Z_h)Z_0) = \mathbb{E}(\mathcal{P}_{l-h}(Z_h)\mathcal{P}_{l-h}(Z_0))$ and $\delta_q^Z(m) = 0$ for $m<0$, we get
    \begin{equation*}
        \begin{split}
            \eqref{eq:acf1} & \leq \sum_{h=0}^\infty g(h) \sum_{l=0}^\infty \|\mathcal{P}_0(Z_l)\|_2\|\mathcal{P}_0(Z_{h-l})\|_2 \\
            & \leq \sum_{l=0}^\infty \delta_2^Z(l)\sum_{h=0}^l g(h)\delta_2^Z(l-h)\\
            & \leq D_2^Z(0)\sum_{l=0}^\infty g(l) \delta_2^Z(l)  < \infty.
        \end{split}
    \end{equation*}
\end{proof}


It is a well-established fact that precise control over the ACF is already sufficient to control many other second-order properties of a stationary process. To begin with, we recall the following result (Theorem 3.8.4 in \cite{brillinger:book:1981}).

\begin{lemma}\label{lemma:brillinger}
Let $(X_t)_{t \in \zz} \in \rr$ be a stationary, zero mean process with ACF $\gamma_{X}$. Suppose that 
\begin{equation*}
\sum_{h = 1}^{\infty} h^{\alpha} |\gamma_{X}(h)| < \infty, \quad \text{for some } \alpha\geq 0.
\end{equation*}
If the spectral density satisfies $f_X(\lambda) \neq 0$ for $\lambda \in \rr$, then we may write $X_t$ as $X_t = \sum_{j = 0}^{\infty} b_j e_{t-j}$ with $e_t = \sum_{j = 0}^{\infty} a_j X_{t-j}$ for any $t \in \zz$ and $a_{0} = b_{0} = 1$, such that
\begin{align}
\sum_{j = 1}^{\infty} j^{\alpha} \big(|a_j| + |b_j|\big) < \infty.
\end{align}
\end{lemma}

In light of Lemma \ref{lemma:rate-acf}, Assumption \ref{I2} implies
\begin{equation*}
  \sum_{h=0}^{\infty}h^{\alpha}|\gamma_{X}(h)| < \infty, \quad \text{for} \quad \alpha \geq \frac{5}{2},
\end{equation*}
which in turn implies $\sum_{j=1}^{\infty}j^{\alpha}(|a_{j}| + |b_{j}|) < \infty$, for $\alpha \geq \frac{5}{2}$ by Lemma \ref{lemma:brillinger}. For the sake of reference, we include this (important) fact below in Lemma \ref{lemma:basic-props-AR(oo)}. Moreover, this property 
will be used in Lemma \ref{lemma:dep-rates} below. It ensures that the innovations too are weakly dependent to a sufficiently high degree. In light of the discussion after Lemma \ref{lemma:A1} this is crucial for our proof.

As a second key result, we discuss Baxter's inequality. Baxter's inequality will be an indispensable tool for the rest of the paper.
In its original version, Theorem 2.2 of \cite{baxter62} states that the rate at which the ACF decays controls the rate at which the AR coefficients $(a_{i})_{i\geq 1}$ can be approximated by $a(k)$. To be more precise, if $X = (X_t)_{t \in \mathbb{Z}}$ is generated by an AR$(\infty)$ model satisfying \ref{G1} and \ref{G2}, and the ACF $\gamma_X(h)$ of $X$ satisfies 
\begin{equation}\label{eq:cond-baxter-old}
    \sum_{h=0}^\infty h^\alpha |\gamma_X(h)| < \infty,
\end{equation}
for some $\alpha\geq 0$, then there is a $k_\alpha \in \mathbb{N}$ and $M_\alpha >0$ such that for all $k\geq k_\alpha$
\begin{equation*}
    \sum_{m=0}^k(2^{\alpha}+ m^\alpha) |a_m - a_m(k)| \leq M_{\alpha} \sum_{m=k+1}^\infty (2^{\alpha}+ m^\alpha) |a_m|.
\end{equation*}
As it turns out, the case $\alpha=0$ is already sufficient. 
If $g:\mathbb{N}\to [0,\infty)$ is a non-decreasing function, Baxter's inequality for $\alpha=0$ implies (for some constant $M>0$)
\begin{align*}
  \sum_{m=0}^{k}g(m)|a_{m}-a_{m}(k)| &\leq g(k) \sum_{m=0}^{k}|a_{m}-a_{m}(k)| \\& \leq Mg(k)\sum_{m=k+1}^{\infty}|a_{m}|\leq M \sum_{m=k+1}^{\infty}g(m)|a_{m}|.
\end{align*}
In other words, the rate at which $(a_{i})_{i\geq 1}$ can be approximated by $a(k)$ depends only on the rate of decay of $(a_{i})_{i\geq 1}$ as long as $\gamma_{X}\in \ell^{1}$.
In light of Lemma \ref{lemma:brillinger}, condition \eqref{eq:cond-baxter-old} is simply a sufficient condition to guarantee that the right-hand side of Baxter's traditional inequality is finite.
However, for arbitrary weight functions $g(h)$ that are not of the form $h^{\alpha}$, for some $\alpha\geq 0$, Lemma \ref{lemma:brillinger} does not guarantee that the absolute summability of $g(H)\gamma_{X}(h)$ implies the absolute summability of $g(h)a_{h}$. 
If the right-hand side of Baxter's inequality is not finite, the inequality is trivially true.
This is summarized in the following lemma.
\begin{lemma}[Baxter's inequality]\label{lemma:baxter}
Let $(X_t)_{t\in \mathbb{Z}}$ be generated by an AR$(\infty)$ model satisfying \ref{G1} and \ref{G2}. If the ACF $\gamma_X(h)$ of $(X_{t})_{t\in \mathbb{Z}}$ is absolutely summable, then there is a constant M, and a $k_{0}\geq 1$, such that for all $k\geq k_{0}$, and all non-decreasing functions $g:\mathbb{N}\to [0,\infty)$,
\begin{equation*}
  \sum_{m=1}^{k}g(m)|a_{m}-a_{m}(k)| \leq Mg(k)\sum_{m=k+1}^{\infty}|a_{m}|.
\end{equation*}
\end{lemma}

Equipped with Lemmas \ref{lemma:rate-acf} -- \ref{lemma:baxter}, we are ready to compare our set of assumptions to the ones imposed in the classical literature.
In the classical case, assumptions are typically imposed on the innovations, and not directly on the process $X_t$ itself. Clearly $X_{t}\in L^{q}$ for some $q>8$ if and only if $e_{t} \in L^{q}$. 
However, it is a priori not obvious whether \ref{I2} directly controls the degree of dependence among the innovations.
Given the close relationship between a linear process and its innovations, one would hope that this is the case. 
If $X_{t} = g_{t}(\varepsilon_{t},\varepsilon_{t-1},\dots)$, then by Lemma \ref{lemma:brillinger}, $e_{t}= \sum_{j=0}^{\infty}a_{j}X_{t-j} = h_{t}(\varepsilon_{t},\varepsilon_{t-1},\dots)$ for some measurable function $h_{t}$.
This allows us to define $\delta_{q}^{I}(l) = \sup_{t\in \mathbb{Z}}\|e_{t}-e_{t-l}^{(l)}\|_{q}$ (see \eqref{defn:dep:measure:general} and the preceding discussion) and 
\begin{equation*}
  D_{q}^{I}(\alpha) = \|e_{0}\|_{q}+ \sum_{l=1}^{\infty}l^{\alpha}\delta_{q}^{I}(l), \quad \text{for } \alpha\geq 0.
\end{equation*}
The quantity $D^{I}_{q}(\alpha)$ measures the degree of dependence among the innovations.
Lemma \ref{lemma:brillinger} and Lemma \ref{lemma:dep-rates} below ensure that given Assumptions \ref{G1} and \ref{I1}, Assumption \ref{I2} implies
\begin{enumerate}[label=(I\arabic*'),ref=(I\arabic*'),leftmargin=*,align=left]
  \setcounter{enumi}{1}
\item \label{I2'} $\sum_{j=0}^{\infty}j^{\alpha}|a_{j}| < \infty$ and $e_{t}$ is a physically dependent process with $D^I_q(\alpha) < \infty$, for some $\alpha \geq 5/2$. 
\end{enumerate}
In this sense, Assumption \ref{I2} relaxes the assumption of independent innovations imposed in the classical literature.
While it may seem more natural to impose assumptions on the process $X$ itself, rather than on the unobservable innovations, it turns out that \ref{I2} and \ref{I2'} are logically equivalent, if one assumes $\sum_{j=0}^{\infty} j^{\alpha}|a_{j}| < \infty$ for some $\alpha\geq 5/2$ in addition to \ref{G1}, \ref{G2}, and \ref{I1}.

In a similar way we can control the dependence among the errors of the fitted AR($k$) processes $\pino{t}{k}$ (see Equation \eqref{eq:fit-k}). 
Let $\delta_{q}^{I_{k}}(l) = \sup_{t\in \mathbb{Z}}\|\pino{t}{k} -\pino{t}{k}^{(l)}\|_{q}$, and 
\begin{equation*}
  D_{q}^{I_{k}}(\alpha) = \|\pino{0}{k}\|_{q} + \sum_{l=1}^{\infty}l^{\alpha}\delta_{q}^{I_{k}}(l), \quad \text{for } \alpha\geq 0.
\end{equation*}
Again, $D_{q}^{I_{k}}(\alpha)$ measures the degree of dependence among the pseudo-innovations $\pino{t}{k}$.

Using Baxter's inequality, we can show that $D_{q}^{X}(\alpha) <\infty$ if and only if $D_{q}^{I}(\alpha)<\infty$ for all $\alpha\leq \beta$ such that $(n^{\beta}a_{n})_{n\in \mathbb{N}} \in \ell^{1}$. 
Furthermore, we get a uniform bound for $D_{q}^{I_{k}}(\alpha)$ in $k$. 
In other words, we have uniform control over the degree of dependence among all fitted AR($k$) models. 

\begin{lemma}\label{lemma:dep-rates}
  Let $\alpha\geq 0$, $q\geq 1$ and assume \ref{G2}. If $\sum_{m\geq 0} m^\alpha |a_m| < \infty$, then there are constants $b_\alpha, B_\alpha >0$, such that
  \begin{equation*}
    b_{\alpha}D_q^X(\alpha) \leq D_q^I(\alpha) \leq B_\alpha D_q^X(\alpha),
  \end{equation*}
  \textit{i.e.}, $D_q^X(\alpha) < \infty$ if and only if $D_q^I(\alpha) < \infty$.
  If $D_q^I(\alpha) < \infty$ (or equivalently $D_q^X(\alpha) < \infty$), then there is a $C_{\alpha}>0$ such that $D_q^{I_k}(\alpha) < C_{\alpha}$ for all $k\geq 0$.  
\end{lemma}
\begin{proof}
By Lemma \ref{lemma:brillinger}, there a constans $c,C$ such that $c\|e_{0}\|_{q}\leq \|X_{0}\|_{q}\leq C\|e_{0}\|$. Hence it is enough to deal with the term $\sum_{l=1}^{\infty}l^{\alpha}\delta_{q}^{X}(l)$ (and the analogue for $e_{t}$). 
    Note that for $x,y \geq 0$ and for $\alpha \geq 1$ we have $(x+y)^\alpha \leq 2^{\alpha -1}(x^\alpha + y^\alpha)$, while for $0 \leq \alpha < 1$ we have $(x+y)^\alpha \leq x^\alpha + y^\alpha$. Hence $(x+y)^\alpha \leq C_\alpha(x^\alpha + y^\alpha)$ for all $x,y \geq 0$ and some $C_\alpha >0$.
  A simple estimate yields
  \begin{equation*}
    \begin{split}
      \sum_{l=1}^{\infty}l^{\alpha}\delta_{q}^{I}(l) &\leq \sum_{l=1}^\infty l^\alpha \sum_{m=0}^\infty |a_m| \delta_q^X(l-m) \\
      & = \sum_{m=0}^\infty |a_m| \sum_{l=1}^\infty (l+m)^\alpha \delta_q^X(l) \\
      & \leq C_\alpha \biggl(D_q^X(\alpha)\sum_{m=0}^\infty |a_m| + D_q^X(0)\sum_{m=0}^\infty m^\alpha |a_m|\biggr).
    \end{split}
  \end{equation*}
  The bound follows from $D_{q}^{X}(0) \leq D_{q}^{X}(\alpha)$ for $\alpha\geq 0$.  
  The lower bound follows similarly.
  By an almost identical estimate, we have
\begin{equation}\label{eq:delta-Ij-asym}
  \begin{split}
    \sum_{l=0}^\infty l^\alpha \delta_q^{I_{k}}(l) & \leq C_\alpha \biggl(\sum_{m=0}^k |a_m(k)| + \sum_{m=0}^k m^\alpha |a_m(k)|\biggr)D_q^X(\alpha).
  \end{split}  
\end{equation}

By  Assumption \ref{G2}  the continuous spectral density of $X$ is uniformly bounded and bounded away from zero. 
Thus, by Baxter's inequality (Lemma \ref{lemma:baxter}), there are $k_0, M >0$, such that for all $k>k_0$
\begin{equation*}
  \sum_{m=0}^k m^\alpha|a_m(k)-a_m| \leq M \sum_{m=k}^\infty m^\alpha|a_m|.
\end{equation*}
Hence, \eqref{eq:delta-Ij-asym} is bounded by some $B_{\alpha}$ independent of $k$.
\end{proof}

The last lemma in this section collects some useful properties of AR$(\infty)$ models, given our assumptions. These properties are well-known. For the convenience of the reader and ease of reference, they are repeated here.
\begin{lemma}\label{lemma:basic-props-AR(oo)}
  Given Assumptions \ref{G1}, \ref{G2}, \ref{I1} and \ref{I2}, there are constants $C_{1}$ and $C_{2}$ such that (note that $q$ and $\alpha$ are given by Assumption \ref{I1} and \ref{I2} respectively). 

    \begin{enumerate}
        \item $\|a-a(k)\|_{\ell^1}\leq C_1k^{-\alpha}$ for all $k$. 
        \item $\|e_0 - \pino{0}{k}\|_{q} \leq C_2k^{-\alpha}$ for all $k$.         
        \item $\|\pino{0}{k}\|_q$ is uniformly bounded in $k$. 
        \item $\mathbb{E}(\pino{t}{i}\pino{0}{j}) = 0$ for $1 \leq t \leq i-j$.  
        \item $\sigma_{k}^2 \geq \sigma^2$ for all $k$.
        \item $\sum_{j=1}^{\infty}j^{\alpha}(|a_{j}| + |b_{j}|) < \infty$, see Lemma \ref{lemma:brillinger} for the definition of $(b_j)_{j \in \nn}$.
    \end{enumerate}
\end{lemma}

\begin{proof}
    Starting with the first claim, Baxter's inequality (Lemma \ref{lemma:baxter}) yields a $k_0$ and a $M>0$, such that for all $k\geq k_0$
    \begin{equation*}
        \begin{split}
            \|a-a(k)\|_{\ell^1} & = \sum_{m=0}^{k-1} |a_m-a_m(k)| + \sum_{m=k}^\infty |a_m| \\
            & \leq (M+1)\sum_{m=k}^\infty |a_m| \\
            & \leq \frac{M+1}{k^{\alpha}} \sum_{m=k}^\infty m^{\alpha} |a_m| \lesssim k^{-\alpha}.
        \end{split}
    \end{equation*}
  By possibly making the constant infront of $k^{-\alpha}$ larger, we get that $\|a-a(k)\|_{\ell^{1}}\lesssim k^{-\alpha}$ for all $k$. 

The second claim is an immediate consequence of the first one and the stationarity of the process $(X_t)_{t\in\zz}$,
\begin{equation*}
\begin{split}
    \|e_0 - \pino{0}{k}\|_q & \leq \sum_{m=1}^\infty |a_m - a_m(k)| \|X_{-m}\|_q = \|X_0\|_q \|a-a(k)\|_{\ell^1} \lesssim k^{-\alpha}.
\end{split}
\end{equation*}
For the third claim, we simply note that 
\begin{equation*}
    \|\pino{0}{k}\|_q \leq \|e_0\|_q + \|\pino{0}{k} - e_0\|_q \leq \|e_0\|_q + Ck^{-\alpha},
\end{equation*}
for some constant $C>0$, which is uniformly bounded in $k$. 

For the fourth claim, recall that $\pino{t}{i}$ is the projection onto the orthogonal complement of the subspace $V_{i}$ of $L^{2}$ spanned by $X_{t-1},\dots,X_{t-i}$. As $\pino{0}{j}$ is a linear combination of $X_{0}, \dots, X_{-j}$, $\mathbb{E}(\pino{t}{i}\pino{0}{j}) = 0$ if $\pino{0}{j} \in V_{i}$, that is, if $1 \leq t \leq i-j$. 

Next, we are going to show that $\sigma_{k}^2 \geq \sigma^2$ for all $k$. Denote with $V_k$ the $L^2$-subspace spanned by $X_{-1},\dots,X_{-k}$, and with $V_\infty$ the closure of the $L^2$-subspace spanned by $X_{j}$ for $j\leq -1$. Let $P_k$ and $P_\infty$ be the orthogonal projections onto $V_k$ and $V_\infty$ respectively. By the definition of the Yule-Walker equations, we have $\pino{0}{k} = X_0 - P_k(X_0)$. 
Since $V_k \subseteq V_\infty$, we have $P_k = P_\infty P_k$, and thus we may decompose the variance of $\pino{0}{k}$, $\sigma_{k}^2$ into
\begin{equation*}
   \sigma_{k}^2= \|\pino{0}{k}\|_{2}^2 = \|X_0 - P_k(X_0)\|_{2}^2 =  \|X_0 - P_\infty(X_0)\|_{2}^2 + \|P_\infty(X_0 - P_k(X_0))\|_{2}^2.
\end{equation*}
The orthogonal projection $Z = P_\infty(X_0)$ is uniquely characterized by 
\begin{equation*}
    \langle X_0 -Z, Y\rangle_{L^2} = 0 \quad \text{for all } Y \in V_\infty 
\end{equation*}
and since $Z = \sum_{j\geq 1} a_j X_{t-j}$ satisfies this requirement on the dense subspace of $V_\infty$ spanned by $X_j$ for $j\leq -1$ (and hence on all of $V_\infty$), we conclude that $X_0-P_\infty(X_0) = e_0$, i.e., $\sigma_{k}^2 \geq \sigma^2$ for all $k$. The last claim was already shown in the discussion below Lemma \ref{lemma:brillinger}.
\end{proof}
Let us end this section by summarizing some central techniques used in this paper. 
Lemma \ref{lemma:burkholder} to \ref{lemma:Q} contain some ideas that have been introduced by W. B. Wu in \cite{wu-physical-dependence} and refined in subsequent works of his and various coauthors (e.g., \cite{WU-QF,Wu-Shao,Wu-Xiao-Cov-Matrix} and the references therein).
The martingale approximation lemma (Lemma \ref{lemma:mdsa}) together with the version of Burkholder's inequality outlined below will be used many times throughout the paper. 
\begin{lemma}[Lemma 1, \cite{Wu-Shao}]\label{lemma:burkholder}
  If $p>1$, $p' = \min(2,p)$ and $D_i \in L^p$ is a martingale difference sequence, then
  \begin{equation*}
    \bigg\lVert \sum_{i=1}^n D_i\bigg\rVert_p^{p'} \leq B_p^{p'}\sum_{i=1}^n \|D_i\|_p^{p'},
  \end{equation*}
  for some constant $B_p$ depending only on $p$. An application of Fatou's Lemma implies
  \begin{equation*}
    \bigg\lVert \sum_{i=1}^\infty D_i\bigg\rVert_p^{p'} \leq B_p^{p'}\sum_{i=1}^\infty \|D_i\|_p^{p'}.
  \end{equation*}
\end{lemma}

\begin{lemma}\label{lemma:PD-standart-est}
    Let $Z_t\in L^{p}$ be a physically dependent process for $p>1$, $\lambda \in \rr^m$, and $p'=\min\{p,2\}$. Then
    \begin{equation*}
        \biggl\|\sum_{t=1}^m \lambda_t(Z_t - \mathbb{E}(Z_t))\biggl\|_p \leq B_{p}\|\lambda\|_{\ell^{p'}}\sum_{l=0}^{\infty}\|Z_{l}-Z_{l}'\|_{p}\leq B_{p}\|\lambda\|_{\ell^{p'}}D_p^Z(0),
    \end{equation*}
    where $B_p$ is the constant from Burkholder's inequality.
\end{lemma}
\begin{proof}
    An application of Lemma \ref{lemma:mdsa} yields
    \begin{equation*}
        \begin{split}
            \biggl\|\sum_{t=1}^m \lambda_t(Z_t - \mathbb{E}(Z_t))\biggl\|_p &\leq \sum_{l=0}^\infty \biggl\|\sum_{t=1}^m \lambda_t\mathcal{P}_{t-l}(Z_t)\biggl\|_p.
        \end{split}
    \end{equation*}
    Note that $\lambda_t\mathcal{P}_{t-l}(Z_t)$ is a martingale difference sequence in $t$ for fixed $l$. Applying Burkholder's inequality, we get
    \begin{equation*}
        \begin{split}
            \sum_{l=0}^\infty \biggl\|\sum_{t=1}^m \mathcal{P}_{t-l}(Z_t)\biggl\|_p & \leq B_p\sum_{l=0}^\infty \biggl(\sum_{t=1}^m \lambda_t^{p'}\|\mathcal{P}_{t-l}(Z_t)\|_p^{p'}\biggl)^{\frac{1}{p'}}
        \end{split}
    \end{equation*}
    Since $\|\mathcal{P}_{t-l}(Z_{t})\|_{p}\leq \delta_{p}^{Z}(l)$ 
    (see \eqref{eq:P(Z)=E(Z-Z')}), we have
    \begin{equation*}
      \biggl(\sum_{t=1}^m \lambda_t^{p'}\|\mathcal{P}_{t-l}(Z_t)\|_p^{p'}\biggl)^{\frac{1}{p'}} \leq B_{p}D_{p}^{Z}(0)\|\lambda\|_{\ell^{p'}}.
    \end{equation*}
    \end{proof}
\begin{remark}\label{rem:prod-pd}
If $V_{t}$ and $W_{t}$ are physically dependent, then so is $Z_{t} = V_{t-i}W_{t-j}$ for $i,j\geq 0$.
In fact, 
\begin{equation*}
  \delta_{p}^{Z}(l) = \|V_{0}\|_{2p}\delta_{2p}^{W}(l-j) + \|W_{0}\|_{2p}\delta_{2p}^{W}(l-i).
\end{equation*}
This is easily seen by inserting $\pm V_{t-i}W_{t-j}'$ into $\|V_{l-i}W_{l-j} - V_{l-i}'W_{l-j}'\|_{p}$ and using Hölder's inequality,
\begin{equation}\label{eq:pd-prod-est}
       \|V_{l-i}W_{l-j} - V_{l-i}'W_{l-j}'\|_{p} \leq \|V_{0}\|_{2p}\|W_{l-j} - W_{l-j}'\|_{2p} + \|W_{0}\|_{2p}\|V_{l-i}-V_{l-i}'\|_{2p}.
\end{equation}
In particular, Lemma \ref{lemma:PD-standart-est} and the previous estimate yield
\begin{equation*}
  \begin{split}
    \biggl\|\sum_{t=1}^{m} \lambda_{t} (V_{t-i}W_{t-j} - \mathbb{E}(V_{t-i}W_{t-j}))\biggr\|_{p} & \leq B_{p}\|\lambda\|_{\ell^{p'}}\sum_{l=0}^{\infty}\|V_{l-i}W_{l-j} - V_{l-i}'W_{l-j}'\|_{p} \\
                                                                                         &  \leq B_{p}\|\lambda\|_{\ell^{p'}}(\|V_{0}\|_{2p}D^{W}_{2p}(0) + \|W_{0}\|_{2p}D_{2p}^{V}(0)),
  \end{split}
    \end{equation*}
    since $\delta_{2p}^{V}(l) = \delta_{2p}^{W}(l) = 0$ for $l<0$.
\end{remark}
The next lemma allows us to control the difference between a physically dependent process $Z_t$ and the projection onto a finite section of its past $\mathbb{E}(Z_t\mid\sigfdouble{t}{t-s+1})$, where $\sigfdouble{t}{r} = \sigma(\varepsilon_{r},\dots,\varepsilon_{t})$, for $r\leq t$. 
\begin{lemma}\label{lemma:Q}
    Let $Z_t \in L^{p}$ be a physically dependent process for $p>1$, and let $Q^t_s(Z_t) = Z_t - \mathbb{E}(Z_t\mid \sigfdouble{t}{s})$. For $p'=\min(2,p)$ we have 
    \begin{equation*}
        \|Q^t_{t-s+1}(Z_t)\|_p \leq  B_{p}\sqrt[p']{\sum_{l=s}^{\infty}\|Z_{l} - Z_{l}'\|_{p}^{p'}} = B_{p}\sqrt[p']{\sum_{l=s}^{\infty}\delta_{p}^{Z}(l)^{p'}}.
    \end{equation*}
\end{lemma}
\begin{proof}
    Since $\mathbb{E}(Z_t\mid\sigfdouble{t}{t-l})$ converges to $Z_t$ almost surely and in $L^p$ by Doob's martingale convergence theorem, we may expand $Q^t_{t-s+1}(Z_t)$ into a telescoping sum (almost surely and in $L^p$)
    \begin{equation*}
        Q^t_{t-s+1}(Z_t) = Z_t -\mathbb{E}(Z_t\mid\sigfdouble{t}{t-s+1}) = \sum_{l=s}^\infty \mathbb{E}(Z_t\mid\sigfdouble{t}{t-l}) - \mathbb{E}(Z_t\mid\sigfdouble{t}{t-l+1}).
    \end{equation*}
    Next, we want to mimic Equation \eqref{eq:P(Z)=E(Z-Z')}. Here we have $\mathbb{E}(Z_t\mid\sigfdouble{t}{t-l+1}) = \mathbb{E}(Z_t^{(t-l)}\mid\sigfdouble{t}{t-l+1}) = \mathbb{E}(Z_t\mid\sigfdouble{t}{t-l})$ and hence
    \begin{equation*}
        \mathbb{E}(Z_t\mid\sigfdouble{t}{t-l}) - \mathbb{E}(Z_t\mid\sigfdouble{t}{t-l+1}) = \mathbb{E}(Z_t - Z_t^{(t-l)}\mid\sigfdouble{t}{t-l}).
    \end{equation*}
    By an argument similar to Inequality \eqref{eq:P(Z)=E(Z-Z')}, we get
    \begin{equation*}
        \|\mathbb{E}(Z_t\mid\sigfdouble{t}{t-l}) - \mathbb{E}(Z_t\mid\sigfdouble{t}{t-l+1})\|_p 
        \leq \delta_p^Z(l).
    \end{equation*}
    Using Burkholder's inequality (Lemma \ref{lemma:burkholder}) gives
    \begin{equation*}
    \|Q^t_{t-s+1}(Z_t)\|_{p} \leq B_{p} \sqrt[p']{\sum_{l=s}^{\infty}\delta_{p}^{Z}(l)^{p'}},
    \end{equation*}
    finishing the proof. 
\end{proof}
\begin{remark}
    If $D_p^Z(\alpha) < \infty$ for $\alpha \geq 0$ we get 
    \begin{equation}\label{eq:Q-rate}
        \|Q^t_{t-s+1}(Z_t)\|_{p} \leq \sum_{l=s}^\infty \delta_p^Z(l) \leq s^{-\alpha} \sum_{l=s}^\infty l^{\alpha} \delta_p^Z(l) \lesssim s^{-\alpha}.
    \end{equation} 
    Let $V_{t}$ and $W_t$ be physically dependent processes with $D_{2p}^{V}(\beta)< \infty$ for $\beta\geq 0$ and $D_{2p}^W(\gamma) < \infty$ for $\gamma \geq 0$, and let $Z_{t} = V_{t-i}W_{t-j}$ for $i,j<s$.
    By Remark \ref{rem:prod-pd} and Estimate \eqref{eq:Q-rate}, we have
    \begin{equation}\label{eq:Q-prod-rate}
    \begin{split}
      \|Q^t_{t-s+1}(Z_t)\|_{p} &\lesssim\sum_{l=s}^\infty \delta_{2p}^V(l-i) + \delta_{2p}^W(l-j) \lesssim (s-i)^{-\beta} + (s-j)^{-\gamma}.
    \end{split} 
    \end{equation}
    Note that for $Z_t = \pino{t}{k}$ the rate  in\eqref{eq:Q-rate} still holds uniformly in $k$ for $\alpha \leq \frac{5}{2}$,  as $D^{I_k}_p(\frac{5}{2})$ is uniformly bounded in $k$ by Lemma \ref{lemma:dep-rates}. 
    Similarly \eqref{eq:Q-prod-rate} remains valid for $Z_{t} = \pino{t-i}{k}\pino{t-j}{k'}$ uniformly in $k$ and $k'$, as $\|\pino{0}{k}\|$ is uniformly bounded by Lemma \ref{lemma:basic-props-AR(oo)}.
\end{remark}

\subsection{Proof of Theorem \ref{thm:main}}\label{sec:proof:thm:main}

\begin{proof}[Proof of Theorem \ref{thm:main}]
  Using the definition of $L_n(k)$, we get
\begin{equation}\label{eq:thm-main-intro}
\begin{split}
    \max_{1\leq k \leq K_n} \biggl|\frac{Q_n(k)}{L_n(k)} - 1\biggr| &=  \max_{1 \leq k \leq K_n} \frac{|\frac{N}{k} \|\hat{a}(k) - a(k)\|_R^2 - \sigma^2|}{\frac{N}{k}L_n(k)}.
\end{split}
\end{equation}
Next, consider the standard quantities 
\begin{equation*}
  B_n = \frac{1}{N}\sum_{t=K_n+1}^{n} X_{t-1}(k)e_{t} \quad \text{and} \quad B_{nk} = \frac{1}{N}\sum_{t=K_n+1}^{n} X_{t-1}(k)\pino{t}{k},
\end{equation*}
where $X_{t-1}(k) = (X_{t-1}, \dots , X_{t-k})^T$. 
Using this, we can write
\begin{equation*}
\begin{split}
    \hat{a}(k) - a(k) &= - \hat{R}(k)^{-1}B_{nk} \\
    & = (R(k)^{-1} - \hat{R}(k)^{-1})B_{nk} + R(k)^{-1}(B_n-B_{nk}) - R(k)^{-1}B_n.\\
\end{split}
\end{equation*}
By multiple applications of the triangle and Cauchy-Schwarz inequalities, \eqref{eq:thm-main-intro} can be estimated by 
\begin{equation}\label{eq:thm-main-intro-est}
    A_1 + A_2^2 + A_3^2 + 2(A_2A_3 + A_2A_4 + A_3A_4),
\end{equation}
where
\begin{equation*}
    \begin{split}
        A_1 &= \max_{1 \leq k \leq K_n} \frac{|\frac{N}{k}\|R(k)^{-1}B_n\|_R^2 - \sigma^2|}{\frac{N}{k}L_n(k)}, \\  
        A_2^2 &= \max_{1 \leq k \leq K_n} \frac{\|(R(k)^{-1} - \hat{R}(k)^{-1})B_{nk}\|_R^2}{L_n(k)}, \\
        A_3^2 &= \max_{1 \leq k \leq K_n} \frac{\|R(k)^{-1}(B_n-B_{nk})\|_R^2}{L_n(k)}, \\
        A_4^2 &= \max_{1 \leq k \leq K_n} \frac{\|R(k)^{-1}B_n\|_R^2}{L_n(k)}. \\
    \end{split}
\end{equation*}

  For $\delta>0$, consider the events $\mathcal{A}_i = \{A_{i}\leq b_{0}(k_{n}^{*})^{-\delta}/13\}$, for $i=1,2,3$ and $b_{0} = \min(b,\sqrt{b})$. 
  Since $\sigma^{2} \leq (N/k)L_{n}(k)$, we get that $A_{4}^{2} \leq A_{1} + 1$, and thus the expression in \eqref{eq:thm-main-intro-est} (and hence \eqref{eq:thm-main-intro}) is bounded by 
\begin{equation*}
  \big[b_{0} + 2b_{0}^{2} + 2( b_{0}^{2} + 2b_{0}\sqrt{b_{0}+1})\big](k_{n}^{*})^{-\delta}/13 \leq b(k_{n}^{*})^{-\delta}
\end{equation*}
on the set $\mathcal{A}_{1}\cap \mathcal{A}_{2}\cap\mathcal{A}_{3}$. 
By a simple union bound, it is enough to show that
\begin{equation*}
  \mathbb{P}((\mathcal{A}_{1}\cap \mathcal{A}_{2}\cap \mathcal{A}_{3})^{c}) \leq \mathbb{P}(\mathcal{A}_{1}^{c}) + \mathbb{P}(\mathcal{A}_{2}^{c})  +\mathbb{P}(\mathcal{A}_{3}^{c})  \lesssim (k_{n}^{*})^{-\gamma}
\end{equation*}
for some $\gamma>0$. By Lemma \ref{lemma:A1}, we have
\begin{equation*}
  \mathbb{P}(\mathcal{A}_{1}^{c}) \lesssim (k_{n}^{*})^{1-p/2+\delta p},
\end{equation*}
while Lemma \ref{lemma:Bn-Bnk} implies
\begin{equation*}
  \mathbb{P}(\mathcal{A}_{3}^{c}) \lesssim (k_{n}^{*})^{1-p/2+2\delta p},
\end{equation*}
where $p=q/4>2$. 
This already requires $d< 1/4 - 1/(2p)$, for otherwise $1-p/2 + 2 \delta p>0$.

For the set $\mathcal{A}_{2}$, we recall that $\|x\|_R = \sqrt{\langle x , R(k) x\rangle}$ is equivalent to the Euclidean norm $\|\cdot\|_{\ell^2}$ on $\mathbb{R}^k$, since the eigenvalues of $R(k)$ all lie in the interval $[\Lambda_1,\Lambda_2] \subseteq (0,\infty)$, by Assumption \ref{G2}.
We may estimate the term $A_2^2$ by
\begin{equation}\label{eq:main-thm-A_2-est}
    A_2^2 \leq \Lambda_2 \max_{1\leq k \leq K_n} \|\hat{R}(k)^{-1} - R(k)^{-1}\|^2 \max_{1\leq k \leq K_n} \frac{\|B_{nk}\|_{\ell^2}^2}{L_n(k)}.
\end{equation}
Since for all $x \in \mathbb{R}^k$
\begin{equation*}
    \Lambda_2^{-1} \leq \frac{\langle x , R(k)^{-1}x\rangle}{\|x\|_{\ell^2}^2} \leq \Lambda_1^{-1},
\end{equation*}
by Assumption \ref{G2}, and $\langle x, R(k)^{-1} x\rangle = \|R(k)^{-1}x\|_R^2$, we get $\|x\|_{\ell^2}^2 \leq \Lambda_2\|R(k)^{-1}x\|_R^2$.
Using this and $(|a| + |b|)^2 \leq 2(|a|^2+|b|^2)$ for $a,b \in \mathbb{R}$, we can estimate the second factor in \eqref{eq:main-thm-A_2-est} by inserting $\pm B_n$ into $\|B_{nk}\|_{\ell^2}$, yielding
\begin{equation*}
\begin{split}
    \max_{1\leq k \leq K_n} \frac{\|B_{nk}\|_{\ell^2}^2}{L_n(k)} &\leq 2\Lambda_2 \max_{1\leq k \leq K_n} \frac{\|R(k)^{-1}(B_{nk} - B_n)\|_{R}^2}{L_n(k)} + 2 \Lambda_2 \max_{1\leq k \leq K_n} \frac{\|R(k)^{-1}B_n\|_{R}^2}{L_n(k)} \\
    & = 2\Lambda_2(A_3^2 + A_4^2). 
\end{split}
\end{equation*}
On the set $\mathcal{A}_{1}\cap \mathcal{A}_{3}$, we have $A_{4}^{2}\leq b_{0}(k_{n}^{*})^{-\delta}/13 + 1$ and $A_{3}\leq b_{0}(k_{n}^{*})^{-\delta}/13$, and hence
\begin{equation*}
  \max_{1\leq k \leq K_n} \frac{\|B_{nk}\|_{\ell^2}^2}{L_n(k)} \leq 2 \Lambda_{2}(1+2 b_{0}/13)
\end{equation*}
as well as 
\begin{equation*}
 A_{2}^{2}\leq 2\Lambda_{2}(1+2b_{0}/13) \max_{1\leq k \leq K_n} \|\hat{R}(k)^{-1} - R(k)^{-1}\|^2.
\end{equation*}
Thus, the set $\mathcal{A}_{2}^{c}\cap(\mathcal{A}_{1}\cap \mathcal{A}_{3})$ is contained in 
\begin{equation*}
  \mathcal{R} =  \biggl\{ \max_{1\leq k \leq K_n} \|\hat{R}(k)^{-1} - R(k)^{-1}\|^2 > \frac{b_{0}^{2}(k_{n}^{*})^{- 2\delta}}{2 \Lambda_{2}13^{2}(1+2 b_{0}/13 )}\biggr\}.
\end{equation*}
Point \ref{mat-norm3} in Lemma \ref{lemma:mat-norm} implies 
\begin{equation*}
    \mathbb{P}(\mathcal{R}) \lesssim ((k_{n}^{*})^{2 \delta}K_{n}N^{-1/2})^{p}.
\end{equation*}
Since $K_{n}^{1+\kappa/2}/\sqrt{N}$ is bounded, the preceeding display can be estimated for any $\delta< \kappa/8$ by
\begin{equation*}
  \mathbb{P}(\mathcal{R}) \lesssim (k_{n}^{*})^{-2 \delta p}.
\end{equation*}

This in turn yields
\begin{align*}
  \mathbb{P}(\mathcal{A}_{2}^{c}) &\leq \mathbb{P}(\mathcal{A}_{2}^{c}\cap(\mathcal{A}_{1}\cap \mathcal{A}_{3})) + \mathbb{P}(\mathcal{A}_{1}^{c}) + \mathbb{P}(\mathcal{A}_{3}^{c}) 
  \\&\leq \mathbb{P}(\mathcal{R})+ \mathbb{P}(\mathcal{A}_{1}^{c}) + \mathbb{P}(\mathcal{A}_{3}^{c}) \lesssim (k_n^*)^{-\gamma},
\end{align*}
where $\gamma = \min\{p/2-1-2 \delta p, 2 \delta p\}>0$, and $\delta < \min\{\kappa/8, 1/4-1/(2p)\}$.

In total, we obtain a constant $C>0$, such that with probability at least $1-C(k_n^*)^{-\gamma}$, we have
\begin{align*}
\bigl|Q_n(k) - L_n(k)\bigr| \lesssim (k_n^*)^{-\delta} L_n(k), \quad 1 \leq k \leq K_n,
\end{align*}
for any $\delta < \min\{1/4 - 1/(2p), \kappa/8\}$, which completes the proof.

\end{proof}

In the very broad scheme of things, our proof above follows the general ideas put forward by Shibata \cite{shibata}. However, there are many technical and conceptual innovations in this proof, which are not apparent from the rough outline presented above. While the technical details are postponed to Section \ref{sec:proofs} for the sake of readability, we summarize and briefly discuss the lemmata used in the proof of Theorem \ref{thm:main} below.

Lemma \ref{lemma:mat-norm} can be thought of as an $L^{p}$ version of Lemma 6 of \cite{Wu-Xiao-Cov-Matrix}. Analog results in the literature are, for instance, Lemma 2.2 in \cite{karagrigoriou2001} or Lemma 2 in \cite{ing:wei:JMVA:2003}, albeit subject to much stronger conditions. In \cite{Wu-Xiao-Cov-Matrix}, Wu et al. achieve a better rate in probability than we do in $L^{p}$. 
We do, however, require an $L^{p}$ version for our Lemma \ref{lemma:QF-1}. An $L^{p}$-rate for the entries of $\hat{R}(k)-R(k)$ is given in \cite{wu2011asymptotic}. However, the $L^{p}$ convergence of $\max_{1\leq k\leq K_{n}} \|\hat{R}{k}-R(k)\| \leq \|\hat{R}(K_{n}) - R(K_{n})\|_{F}$ or the convergence of $\max_{1\leq k \leq K_n}\|\hat{R}(K_n)^{-1} - R(K_n)^{-1}\|$ does not seem to be recorded in the literature.

\begin{lemma}\label{lemma:mat-norm}
If $X_t \in L^{2p}$ satisfies \ref{G2}, such that $D_{2p}^X(0) < \infty$ for some $p>1$ and $K_n^2 n^{1-p'} \to 0$ for $p'=\min(p,2)$, then
\begin{enumerate}
  \item \label{mat-norm1}
    $\displaystyle \bigl\|\|\hat{R}(K_{n}) - R(K_{n})\|_{F}\bigr\|_{p}\leq C\bigl(K_n^2N^{1-p'} \bigr)^{\frac{1}{p'}} \to 0$, for some $C>0$,
  \item $\displaystyle \max_{1\leq k \leq K_n}\|\hat{R}(k)^{-1} - R(k)^{-1}\| \overset{\mathbb{P}}{\longrightarrow} 0$\label{mat-norm2}, and 
  \item \label{mat-norm3} for every $K>0$, there is a constant $C$, such that for every $c>0$
      \begin{equation*}
        \mathbb{P}\biggl(\max_{1\leq k\leq K_{n}}\|\hat{R}(k)^{-1}-R(k)^{-1}\|>K/(k_{n}^{*})^{c}\biggr) \leq C((k_{n}^{*})^{cp'}K_{n}^{2}N^{1-p'})^{\frac{p}{p'}},
    \end{equation*}

\end{enumerate}
  where $\|A\|_F = \sqrt{\mathrm{tr}(A^TA)}$ denotes the Frobenius norm, and $\|A\| = \sup_{\|x\|\leq 1} \| Ax\|$ is the operator norm. 
  \end{lemma}

 Using Lemma \ref{lemma:mat-norm}, we can generalize several results that are, in one form or another, present in every asymptotic efficiency proof for AIC-like model selection criteria. 
    Lemma \ref{lemma:Bn-Bnk-rate} is a generalization of Lemma 3.1 in the classical work of Shibata \cite{shibata} (and Lemma 3 in \cite{ing:wei:JMVA:2003}). It is used in the proof of Lemma \ref{lemma:Bn-Bnk}.

 \begin{lemma}\label{lemma:Bn-Bnk-rate}
   Under Assumptions \ref{ass:main} and \ref{ass:weak:dep}, we have for $p=q/4$ 
  \begin{equation*}
    \bigl\|\|B_{n} - B_{nk}\|_{R}^{2}\bigr\|_{p} \lesssim \frac{k}{N}\|a-a(k)\|_{\ell^{1}}^{2}.
  \end{equation*}
  \end{lemma}
  Using Lemma \ref{lemma:Bn-Bnk}, we can control the error terms appearing in the proof of Theorem \ref{thm:main}. Once Lemma \ref{lemma:Bn-Bnk} is established, only elementary transformations are necessary to control the terms $A_{2}, A_{3}$ and $A_{4}$.

  \begin{lemma}\label{lemma:Bn-Bnk}
  Under Assumptions \ref{ass:main} and \ref{ass:weak:dep}, we have 
  \begin{equation*}
    A_3^2 = \max_{1 \leq k \leq K_n} \frac{\|R(k)^{-1}(B_n - B_{nk})\|_R^2}{L_n(k)} \overset{\mathbb{P}}{\longrightarrow} 0.
  \end{equation*}
  More precisely, for any constant $c>0$, there is an absolut constant $C >0$ such that for every $\delta>0$, we have
  \begin{equation*}
    \mathbb{P}\bigl(A_{3}^{2}> c(k_{n}^{*})^{-\delta}\bigr) \leq C (k_{n}^{*})^{1-p/2+\delta p}.
\end{equation*}
\end{lemma}

  However, the main technical and conceptual innovations lie in Lemma \ref{lemma:A1}.
  Versions of this lemma are crucial steps in every proof of asymptotic efficiency so far. It generalizes Lemma 3.2 of \cite{shibata}, Lemma 2.3 of \cite{karagrigoriou2001} and Theorem 3 of \cite{ing:wei:JMVA:2003}.

\begin{lemma}\label{lemma:A1}
    Under Assumptions \ref{ass:main} and \ref{ass:weak:dep}, we have
    \begin{equation*}
        A_1 = \max_{1 \leq k \leq K_n} \frac{|\frac{N}{k}\|R(k)^{-1}B_n\|_R^2-\sigma^2|}{\frac{N}{k}L_n(k)} \overset{\mathbb{P}}{\longrightarrow} 0.
    \end{equation*}
   In particular, for every constant $c>0$, there is an absolute constant $C>0$, such that for every $\delta>0$, we have
    \begin{equation*}
      \mathbb{P}(A_{1}\geq c(k_{n}^{*})^{-\delta})\leq C (k_{n}^{*})^{1-p/2+\delta p}.
  \end{equation*}
\end{lemma}

  To prove Lemma 3.2 in \cite{shibata}, Shibata exploited the gaussianity of the innovations and explicit estimates for their joint cumulants. 
  Subsequent authors \cite{karagrigoriou2001,ing-wei-main, ing:wei:JMVA:2003, ing-sin-yu-integrated} split $X_{t}$ into
  \begin{equation}\label{eq:classical-proof-idea}
    X_{t} =  \sum_{j=0}^{\infty}b_{j}e_{t-j} = \sum_{j=0}^{f(n)} b_{j}e_{t-j} + \sum_{j=f(n)+1}^{\infty}b_{j}e_{t-j},
  \end{equation}
  where $b_{0}=1$ and $f:\mathbb{N}\to \mathbb{N}$ is chosen in a suitable way, e.g., $f(n) = K_{n}$. 
  Since they assumed that the innovations are independent, the two summands on the right-hand side of Equation \eqref{eq:classical-proof-idea} are independent as well.
  However, in our case, the innovations are not independent, and we do not benefit from the simplifications taking place in earlier proofs.
Instead, we can only use the intrinsic structure of the autoregression problem.
It is well known, that the autoregressive coefficients $a(k)$ are uniquely determined by the covariance structure of the process (see Equation \eqref{eq:fit-k} and the subsequent discussion). 
However, the converse is also true.
The inverse covariance matrix of the process $X$ can be written as
\begin{equation*}
  R(k)^{-1} = L(k)^{T}D(k)L(k)
\end{equation*}
where
  \begin{equation*}
    L(k) =
    \begin{pmatrix}
      1 & a_1(k-1) &  \cdots& a_{k-2}(k-1)& a_{k-1}(k-1)\\
      0 & 1 & a_1(k-2) &  \cdots& a_{k-2}(k-2)\\
      \vdots &\vdots &\vdots &\vdots &\vdots \\
      0 & \cdots & 0 & 1 & a_1(1)\\
      0 & \cdots & 0 & 0 & 1\\
    \end{pmatrix},
  \end{equation*}
  $D(k) = \mathrm{diag}(\sigma_{k-1}^{-2}, \dots, \sigma_{0}^{-2})$, and $\sigma_{j}^{2} = \mathbb{E}(\pino{0}{j}^{2})$ (see \cite{kromer}, p. 98).
  This allows us to formulate Lemma \ref{lemma:A1} purely in terms of the innovations.
  Using Lemma \ref{lemma:dep-rates}, the physical dependence conditions imposed on the process $X$ give us sufficient control over the innovations (and the pseudo-innovations defined in Equation \eqref{eq:fit-k}).
  It should be noted, that the classical approach outlined in Equation \eqref{eq:classical-proof-idea} is feasible in our case. 
  However, if the innovations are not independent, multiple new cross terms have to be dealt with, that do not appear in previous works. 
  In dealing with these cross terms, we have to perform essentially the same steps that we perform in our proof multiple times.
  Additionally, some non-cross terms are significantly harder to deal with than in the classical case.
  An attentive reader will note that our proof of Lemma \ref{lemma:A1} simplifies drastically once independent innovations are assumed. 
  In this case, many terms simply vanish, and many results that pose a technical hurdle in our setting turn into trivialities (e.g. Lemma \ref{lemma:stepII}).
  Thus, our approach can be thought of as a streamlined version of earlier proofs, that does not depend on the independence of the innovations.

  \subsection{Proof of Theorems \ref{thm:irs-asymp-eff} and \ref{thm:irs-stable}}\label{sec:proof:thm:irs-asymp_eff}

  The proofs of Theorem \ref{thm:irs-asymp-eff} and Theorem \ref{thm:irs-stable} require  Lemma \ref{lemma:irs:sk}. 
  This lemma should be thought of as a more general, quantative version of Proposition 4.1 in \cite{shibata}.
\begin{lemma}\label{lemma:irs:sk}
    
Grant Assumptions \ref{ass:main} and \ref{ass:weak:dep}. Then for every $b>0$ there is a constant $C>0$ (depending on $b$), such that for every $\delta < 1/2 - 2/q$, with a probability of at least $1-C(k_{n}^{*})^{1+\delta q/2 - q/4}$, we have
\begin{equation*}
  |s_{k_{n}^{*}}^{2}-\sigma^{2}_{k_{n}^{*}} - s_{k}^{2}+\sigma_{k}^{2}| \leq b(k_{n}^{*})^{-\delta}L_{n}(k), \qquad \forall k=1,\dots,K_{n}.
\end{equation*}
\end{lemma}

For the proof of \ref{lemma:irs:sk}, we require Lemma \ref{lemma:QF-1}, which acts as a substitute for Lemma 4.2 in \cite{shibata}. Once Lemma \ref{lemma:QF-1} is established, the proof of Lemma \ref{lemma:irs:sk} proceeds by employing an argument of Shibata (cf. Proposition 4.1, \cite{shibata}). Let us mention that if $(i \xi_{i})_{i\geq 1} \in \ell^{1}$, the rates we get even improve upon those of Shibata achieved in Lemma 4.2 of \cite{shibata} for Gaussian processes. The idea of the lemma itself is quite simple. We are interested in controlling the $L^{p}$ norm of the expression
  \begin{equation*}
    \sum_{i,j=1}^{d}\xi_{i}\eta_{j}(\hat{R}_{ij}-R_{ij}),
  \end{equation*}
  for some vectors $\xi,\eta\in \mathbb{R}^{d}$.
  A simple iteration of the Cauchy-Schwarz inequality yields a suboptimal dependence on the dimension $d$ (cf. the proof of Lemma \ref{lemma:mat-norm}).
  On the other hand, one can quite easily derive a \textit{symmetric} estimate of the form
  \begin{equation*}
     \biggl\|\sum_{i,j=1}^{d} \xi_i\eta_j (\hat{R}_{ij} - R_{ij})\biggr\|_p \lesssim N^{\frac{1}{p'}-1}(\|\eta\|_{\ell^{2}}\|\xi\|_{\ell^{1}} + \|\eta\|_{\ell^{1}}\|\xi\|_{\ell^{2}}),
  \end{equation*}
  using Lemma \ref{lemma:PD-standart-est} and the subsequent remark.
  However, in many applications $\eta = \eta(k) \to 0$, while $\xi = \xi(k)$ is only bounded in $k$.
  Hence, we would like to replace $\|\eta\|_{\ell^{1}}$ with $\|\eta\|_{\ell^{2}}$ in the second term, to achieve a faster rate of convergence. 
  On the other hand, we do not care much for $\xi$, as it does not contribute to the rate of convergence. 
  This begs the question: Can we trade $\xi$ for $\eta$, i.e., is there an \textit{asymmetric} estimate of the form 
  \begin{equation*}
     \biggl\|\sum_{i,j=1}^{d} \xi_i\eta_j (\hat{R}_{ij} - R_{ij})\biggr\|_p \lesssim N^{\frac{1}{p'}-1}\|\eta\|_{\ell^{2}}f(\xi),
  \end{equation*}
  for some function $f: \mathbb{R}^{d}\to [0,\infty)$?
  The next lemma provides a positive answer to this question.

\begin{lemma}\label{lemma:QF-1}
 Let $X_{t} \in L^{2p}$ for $p>1$ be a physically dependent process, such that the sequence
\begin{equation*}
  \kappa_{s} = \sum_{l=s}^{\infty}\sqrt{\sum_{k=l}^{\infty}\delta_{2p}^{X}(k)^{2}}
\end{equation*}
is in $\ell^{2}$.
   There is a constant $C$ depending only on $p$ and the distribution of the underlying process $X$, such that for all $\eta = (\eta_{j})_{j=0}^{d}$, and  $\xi = (\xi_{j})_{j=0}^{d}$ in $\mathbb{R}^{d+1}$,

  \begin{equation*}
    \begin{split}
      \biggl\|\sum_{i,j=0}^{d} \xi_i\eta_j (\hat{R}_{ij} - R_{ij})\biggr\|_p \leq 
      C\|\eta\|_{\ell^2}N^{\frac{1}{p'}-1}\biggl(\|\xi\|_{\ell^1} + \|\xi\|_{\ell^1}^{\frac{1}{2}}\biggl(\sum_{i=0}^{d}(i+1)|\xi_i|\biggr)^{\frac{1}{2}}\biggr),
    \end{split}
    \end{equation*}
  where $p'=\min\{2,p\}$.
\end{lemma}
A sufficient condition for $(\kappa_{s})_{s\in \mathbb{N}}\in \ell^{2}$ is $D_{2p}^{X}(\beta) < \infty$ for some $\beta>2$.
The process $X$ satisfies this condition by Assumption \ref{I2}.

\begin{proof}[Proof. (Theorem \ref{thm:irs-asymp-eff})]
We will show that 
  \begin{equation*}
    \mathcal{C} = \biggl\{\biggl|\frac{Q_{n}(\hat{k}_{n})}{L_{n}(k_{n}^{*})} - 1\biggr| \leq 8(k_{n}^{*})^{-\delta}\biggr\}
  \end{equation*}
  has a probability of at least $1-C(k_{n}^{*})^{-\gamma}$, for some $C,\gamma>0$, and $\delta>0$ from Theorem \ref{thm:main}.

  Let $\mathcal{A}$ be the event
  \begin{equation*}
    \mathcal{A} = \biggl\{\max_{1\leq k \leq K_{n}}\biggl|\frac{Q_{n}(k)}{L_{n}(k)}-1\biggr|\leq (k_{n}^{*})^{-\delta}/2 \biggr\},
  \end{equation*}
  with $\delta$ being taken from Theorem \ref{thm:main}.
  By Theorem \ref{thm:main}, $\mathbb{P}(\mathcal{A}^{c}) \lesssim(k_{n}^{*})^{-\gamma}$.
  Let $a = Q_{n}(\hat{k}_{n})/L_{n}(\hat{k}_{n})$ and $b = L_{n}(\hat{k}_{n})/L_{n}(k_{n}^{*})$. 
  We have
  \begin{equation}\label{eq:reduction-to-B}
    \begin{split}
      \biggl|\frac{Q_{n}(\hat{k}_{n})}{L_{n}(k_{n}^{*})} - 1\biggr| &= |ab -1| = |ab \pm a^{-1}b - a^{-1}a| \\
                                                                    & \leq |b| \bigl(|a-1| + |a^{-1}-1|\bigr)  + |a^{-1}| \bigl(|b-1| + |a-1|\bigr)\\
                                                                    &\leq |b-1| \bigl(|a-1| + |a^{-1}-1|\bigr)  + |a^{-1}-1| \bigl(|b-1| + |a-1|\bigr) \\
                                                                    &+ 2|a-1| + |a^{-1}-1| + |b-1|.
    \end{split}
  \end{equation}
  Recall that $|a -1| \leq \varepsilon$ implies 
  \begin{equation*}
    |a^{-1}-1| \leq \max \biggl(\frac{\varepsilon}{1+\varepsilon}, \frac{\varepsilon}{1-\varepsilon}\biggr) \leq 2 \varepsilon,
  \end{equation*}
  for every $\varepsilon\leq 1/2$. 
  On $\mathcal{A}$, we have
  \begin{equation*}
    |a-1| = \biggl|\frac{Q_{n}(\hat{k}_{n})}{L_{n}(\hat{k}_{n})} - 1\biggr|\leq \max_{1\leq k \leq K_{n}}\biggl|\frac{Q_{n}(k)}{L_{n}(k)} -1\biggr|\leq \frac{(k_{n}^{*})^{-\delta}}{2}\leq \frac{1}{2},
  \end{equation*}
  and hence
  \begin{equation*}
    |a^{-1}-1|\leq (k_{n}^{*})^{-\delta}.
  \end{equation*}
  Thus, by the last estimate and \eqref{eq:reduction-to-B}, we have $\mathcal{C}\subseteq \mathcal{A}\cap \mathcal{B}$, where
  \begin{equation*}
    \mathcal{B} = \biggl\{\biggl|\frac{L_{n}(\hat{k}_{n})}{L_{n}(k_{n}^{*})} -1\biggr| \leq (k_{n}^{*})^{-\delta} \biggr\},
  \end{equation*}
  and it is enough to show that $\mathbb{P}(\mathcal{B}^{c}) \leq C(k_{n}^{*})^{-\gamma}$, for some $C,\gamma>0$.
  Note that
  \begin{equation*}
    \frac{L_{n}(\hat{k}_{n})}{L_{n}(k_{n}^{*})}- 1 > (k_{n}^{*})^{-\delta} \iff 1 - \frac{L_{n}(k_{n}^{*})}{L_{n}(\hat{k}_{n})} > \frac{(k_{n}^{*})^{-\delta}}{1 + (k_{n}^{*})^{-\delta}}.
  \end{equation*}
  Since $S_{n}(k_{n}^{*}) - S_{n}(\hat{k}_{n}) \geq 0$ and $(k_{n}^{*})^{-\delta}/(1+(k_{n}^{*})^{-\delta})> (k_{n}^{*})^{-\delta}/2$, this implies 
  \begin{equation*}
    \frac{N\bigl(L_{n}(\hat{k}_{n}) - L_{n}(k_{n}^{*})\bigr) + S_{n}(k_{n}^{*}) - S_{n}(\hat{k}_{n})}{NL_{n}(\hat{k}_{n})}> \frac{(k_{n}^{*})^{-\delta}}{2},
  \end{equation*}
  and hence
  \begin{equation*}
    \mathcal{B}^{c}\subseteq \mathcal{D} \coloneqq \biggl\{\max_{1\leq k \leq K_{n}}\frac{N\bigl(L_{n}(k) - L_{n}(k_{n}^{*})\bigr) + S_{n}(k_{n}^{*})-S_{n}(k) }{NL_{n}(k)}> \frac{(k_{n}^{*})^{-\delta}}{2} \biggr\}.
  \end{equation*}
  In other words, we need to show that the dominant term in $S_n(k) - S_n(k_{n}^{*})$, relative to $NL_n(k)$, is $N(L_n(k) - L_n(k_{n}^{*}))$, uniformly in $1 \leq k \leq K_n$.
  Since 
    \begin{equation*}
    \begin{split}
        \sigma_k^2 - \sigma^2 &= \|a-a(k)\|_R^2, \quad \text{and} \\
        s_k^2 - \hat{\sigma}_k^2 & = \|\hat{a}(k) - a(k)\|_{\hat{R}(k)}^2,
    \end{split}
    \end{equation*}
    where $s_k^2 = N^{-1}\sum_{t=K_n+1}^n \pino{t}{k}^2$, Shibata's criterion $S_n(k) = (N+2k)\hat{\sigma}_k^2$ may be rewritten as 
    \begin{equation*}
    S_n(k) = NL_n(k) + 2k(\hat{\sigma}_k^2 - \sigma^2) + k\sigma^2 - N\|\hat{a}(k) - a(k)\|_{\hat{R}(k)}^2 + N(s_k^2 - \sigma_k^2) + N\sigma^2.
    \end{equation*}
   
    Note that the term $N\sigma^2$ cancels in $S_n(k_n^*) - S_n(k)$.
    Hence, we may write
    \begin{equation*}
     N\bigl(L_{n}(k) - L_{n}(k_{n}^{*})\bigr) + S_{n}(k_{n}^{*})-S_{n}(k) = I_{1}(k) + I_{2}(k) + I_{3}(k),
    \end{equation*}
    where
    \begin{equation*}
      \begin{split}
        I_{1}(k) &= k_{n}^{*}\sigma^{2} - N\|\hat{a}(k_{n}^{*})-a(k_{n}^{*})\|_{\hat{R}(k_{n}^{*})}^{2}
 - k\sigma^{2} + N\|\hat{a}(k)-a(k)\|_{\hat{R}(k)}^{2}, \\
        I_{2}(k) & = N(s_{k_{n}^{*}}^{2}-\sigma_{k_{n}^{*}}^{2}) -N(s_{k}^{2}-\sigma_{k}^{2}), \quad \text{and}\\
        I_{3}(k) & = 2k_{n}^{*}(\hat{\sigma}_{k_{n}^{*}}^{2}-\sigma^{2}) - 2k(\hat{\sigma}_{k}^{2}-\sigma^{2}).
      \end{split}
    \end{equation*}
    Now, $\mathcal{D} \subseteq \mathcal{D}_{1} \cup \mathcal{D}_{2} \cup \mathcal{D}_{3}$, where
    \begin{equation*}
      \mathcal{D}_{i} = \biggl\{\max_{1\leq k\leq K_{n}}\frac{I_{i}(k)}{NL_{n}(k)} > \frac{(k_{n}^{*})^{-\delta}}{6} \biggr\}.
    \end{equation*}
    We are going to show that $\mathbb{P}(\mathcal{D}_{i}) \leq C_{i}(k_{n}^{*})^{-\gamma_{i}}$ for some $C_{i},\gamma_{i}>0$ and $i=1,2,3$. Then we have
    \begin{equation*}
      \mathbb{P}(\mathcal{B}) \leq  \mathbb{P}(\mathcal{D}) \leq (C_{1}+C_{2}+C_{3}) (k_{n}^{*})^{-\min(\gamma_{i})},
    \end{equation*}
    finishing the proof.

    Starting with $\mathcal{D}_{1}$, we first observe that
    \begin{equation*}
      \begin{split}
        \max_{1\leq k\leq K_{n}}\frac{\bigl|k_{n}^{*}\sigma^{2} - N\|\hat{a}(k_{n}^{*})-a(k_{n}^{*})\|_{\hat{R}(k_{n}^{*})}^{2}\bigr|}{NL_{n}(k)} &\leq \frac{\bigl|k_{n}^{*}\sigma^{2} - N\|\hat{a}(k_{n}^{*})-a(k_{n}^{*})\|_{\hat{R}(k_{n}^{*})}^{2}\bigr|}{NL_{n}(k_{n}^{*})}\\
                                                                                                                                        & \leq \max_{1\leq k\leq K_{n}}\frac{\bigl|k\sigma^{2} - N\|\hat{a}(k)-a(k)\|_{\hat{R}(k)}^{2}\bigr|}{NL_{n}(k)},
      \end{split}
    \end{equation*}
    and hence
    \begin{equation*}
      \max_{1\leq k \leq K_{n}}\frac{I_{1}(k)}{NL_{n}(k)}\leq 2 \max_{1\leq k\leq K_{n}}\frac{|k\sigma^{2} - N\|\hat{a}(k)-a(k)\|_{\hat{R}(k)}^{2}|}{NL_{n}(k)}.
    \end{equation*}

    We note that  
    \begin{equation*}
    \begin{split}
        \|\hat{a}(k) - a(k)\|_{\hat{R}(k)}^2  &= \langle \hat{a}(k) - a(k) , \hat{R}(k)(\hat{a}(k) - a(k))\rangle \\
        & = \|\hat{a}(k) - a(k)\|_R^2 + \langle \hat{a}(k) - a(k) , (\hat{R}(k) - R(k))(\hat{a}(k) - a(k))\rangle.
    \end{split}
    \end{equation*}
    Since the Euclidean norm $\|\cdot\|_{\ell^2}$ and $\|\cdot\|_R$ are equivalent, we may estimate $|k\sigma^2 - N\|\hat{a}(k) - a(k)\|_{\hat{R}(k)}^2|$ by 
    \begin{equation*}
    \begin{split}
        |k\sigma^2 - N\|\hat{a}(k) - a(k)\|_{\hat{R}(k)}^2|  &\leq |k\sigma^2-  N\|\hat{a}(k) - a(k)\|_R^2|\\
        &+ N\Lambda_1^{-1}\|\hat{R}(K_n) - R(K_n)\|\|\hat{a}(k) - a(k)\|_R^2,
    \end{split}
    \end{equation*}

    where $\Lambda_1$ is the uniform lower bound on the eigenvalues of $R(k)$ implied by Assumption \ref{G2} (see the beginning of Section \ref{sec:prelim}). 
    Thus, $\mathcal{D}_{1}\subseteq \mathcal{D}_{1}' \cup \mathcal{D}_{1}''$, where
    \begin{equation*}
      \begin{split}
        \mathcal{D}_{1}' &= \biggl\{ \max_{1\leq k\leq K_{n}} \frac{|k\sigma^2-  N\|\hat{a}(k) - a(k)\|_R^2|}{NL_{n}(k)}> \frac{(k_{n}^{*})^{-\delta}}{24}\biggr\}, \quad \text{and} \\
        \mathcal{D}_{1}'' & = \biggl\{\max_{1\leq k\leq K_{n}} \frac{\|\hat{R}(K_n) - R(K_n)\|\|\hat{a}(k) - a(k)\|_R^2}{L_{n}(k)}> \Lambda_{1}\frac{(k_{n}^{*})^{-\delta}}{24} \biggr\}.
      \end{split}
    \end{equation*}
    By Theorem \ref{thm:main} (cf. Equation \eqref{eq:thm-main-intro}), we have $\mathbb{P}(\mathcal{D}_{1}') \leq C(k_{n}^{*})^{-\gamma}$.
    On the complement of $\mathcal{D}_{1}'$, we have
    \begin{equation*}
      \|\hat{R}(K_{n}) - R(K_{n})\| \max_{1\leq k\leq K_{n}}\frac{\|\hat{a}(k)- a(k)}{L_{n}(k)} \leq 2\|\hat{R}(K_{n}) - R(K_{n})\|,
    \end{equation*}
    and hence $\mathcal{D}_{1}'' \cap (\mathcal{D}_{1}')^{c}$ is contained in the set
    \begin{equation*}
      \mathcal{D}_{1}''' = \biggl\{\|\hat{R}(K_{n}) - R(K_{n})\| >\Lambda_{1}\frac{(k_{n}^{*})^{-\delta}}{48} \biggr\}.
    \end{equation*}
    By part \ref{mat-norm1} of Lemma \ref{lemma:mat-norm} and Markov's inequality, we have
    \begin{equation*}
      \mathbb{P}(\mathcal{D}_{1}''') \lesssim (k_{n}^{*})^{\delta p}K_{n}^{p}/N^{p/2}.
    \end{equation*}
    Note that $\delta < \kappa/8$ from the assumptions of Theorem \ref{thm:main}. For any $\delta < \kappa/4$, we can estimate the right-hand side of the last display by 
    \begin{equation*}
      (k_{n}^{*})^{\delta p}K_{n}^{p}/N^{p/2} \lesssim (k_{n}^{*})^{-\frac{\kappa p}{4}},
    \end{equation*}
    since $K_{n}^{1+\kappa/2}/\sqrt{N}$ is bounded by Assumption \ref{G3}.
    In total,
    \begin{equation*}
      \mathbb{P}(\mathcal{D}_{1}'') \leq \mathbb{P}(\mathcal{D}_{1}''\cap \mathcal{D}_{1}') + \mathbb{P}(\mathcal{D}_{1}'' \cap (\mathcal{D}_{1}')^{c}) \lesssim (k_{n}^{*})^{-\gamma_{0}},
    \end{equation*}
    where $\gamma_{0} = \min(\gamma, \kappa p/4)$.

    For $\mathcal{D}_{2}$, Lemma \ref{lemma:irs:sk} implies $\mathbb{P}(\mathcal{D}_{2}) \leq C_{2}(k_{n}^{*})^{1+\delta q/2 - q/4}$, for every $\delta < 1/2 - 2/q$.
Moving on to $\mathcal{D}_{3}$, we first note that
    \begin{equation*}
        \hat{\sigma}_k^2 - \sigma^2 = \hat{\sigma}_k^2- s_k^2 + s_k^2 - \sigma_k^2  + \sigma_k^2 - \sigma^2.
    \end{equation*}
    The last difference $\sigma_{k}^{2} - \sigma^{2} = \|a(k)-a\|_{R(k)}^{2}$ is deterministic, and non-negative. 
    Hence, we can estimate $\max_{1\leq k \leq K_{n}} I_{3}(k)/(NL_{n}(k))$ by $W_{1}+W_{2}+W_{3}$, with
    \begin{equation*}
      \begin{split}
        W_{1} &= 4\max_{1\leq k \leq K_{n}}\frac{\hat{\sigma}_{k}^{2} - s_{k}^{2}}{\frac{N}{k}L_{n}(k)}, \\
        W_{2} &= \max_{1\leq k \leq K_{n}}\frac{2k_{n}^{*}(s_{k_{n}^{*}}^{2}-\sigma_{k_{n}^{*}}^{2}) - 2k(s_{k}^{2}-\sigma_{k}^{2})}{NL_{n}(k)}, \quad \text{and}\\ 
        W_{3} & = \|a(k_{n}^{*})-a\|_{R(k_{n}^{*})}^{2}.
      \end{split}
    \end{equation*}
     By Assumption \ref{G2}, we have $\|\cdot\|_{R(k)} \leq \Lambda_{2}\|\cdot\|_{\ell^{2}}$, where $\Lambda_{2}$ is an upper bound on the eigenvalues of $R(k)$ (uniformly in $k$).
     By Lemma \ref{lemma:basic-props-AR(oo)} we have
     \begin{equation*}
       W_{3} = \|a(k_{n}^{*})-a\|_{R(k_{n}^{*})}^{2} \leq \Lambda_{2}^{2} \|a(k_{n}^{*})-a\|_{\ell^{1}}^{2} \leq M \Lambda_{2}^{2}(k_{n}^{*})^{-5}.
     \end{equation*}
     In other words, $\mathcal{D}_{3} \subseteq \mathcal{D}_{3}^{(1)} \cup \mathcal{D}_{3}^{(2)}$, with
     \begin{equation*}
       \begin{split}
         \mathcal{D}_{3}^{(i)} = \biggl\{W_{i} \geq \frac{(k_{n}^{*})^{-\delta}}{2}\biggl(\frac{1}{6}-M \Lambda_{2}^{2}(k_{n}^{*})^{-5+\delta}\biggr) \biggr\}. 
       \end{split}
     \end{equation*}
     There is a $n_{0}$, such that for every $n\geq n_{0}$,
     \begin{equation*}
       \frac{1}{6}-M \Lambda_{2}^{2}(k_{n}^{*})^{-5+\delta} \geq \frac{1}{12}.
     \end{equation*}
     Hence, for $n\geq n_{0}$ and $i=1,2$, $\mathcal{D}_{3}^{(i)}\subseteq \tilde{\mathcal{D}}_{3}^{(i)} = \{W_{i}\geq (k_{n}^{*})^{-\delta}/24\}$.
     In the case where $n\geq n_{0}$, we are going to show that $\mathbb{P}(\tilde{\mathcal{D}}_{3}^{(i)}) \leq \tilde{C}_{i}(k_{n}^{*})^{-\gamma_{i}}$.
     Then we simply increase the constant $\tilde{C}_{i}$ (say, to $C_{i}$) so that $C_{i}(k_{n}^{*})^{-\gamma_{i}}\geq 1$, for all $n=1,\dots,n_{0}-1$, which yields
     \begin{equation*}
       \mathbb{P}(\mathcal{D}_{3}^{(i)}) \leq  \mathbb{P}(\tilde{\mathcal{D}}_{3}^{(i)})\leq C_{i}(k_{n}^{*})^{-\gamma_{i}},
     \end{equation*}
     for all $n\in \mathbb{N}$, and hence 
     \begin{equation*}
       \mathbb{P}(\mathcal{D}_{3}) \leq \mathbb{P}(\mathcal{D}_{3}^{(1)}) + \mathbb{P}(\mathcal{D}_{3}^{(2)}) \leq (C_{1}+C_{2})(k_{n}^{*})^{-\min(\gamma_{1},\gamma_{2})},
     \end{equation*}
     finishing the proof. 

     Starting with $\tilde{\mathcal{D}}_{3}^{(1)}$, we note that 
     \begin{equation*}
     \begin{split}
       \max_{1\leq k \leq K_{n}}\frac{|\hat{\sigma}_{k}^{2} - s_{k}^{2}|}{\frac{N}{k}L_{n}(k)} &= \frac{K_{n}}{N}\max_{1\leq k \leq K_{n}} \frac{N\|\hat{a}(k)-a(k)\|_{\hat{R}(k)}^{2}}{NL_{n}(k)}\\
                                                                                               &\leq \sigma^{2} \frac{K_{n}^{2}}{N} + \frac{K_{n}}{N}\max_{1\leq k\leq K_{n}}\frac{|k \sigma^{2} - N\|\hat{a}(k) - a(k)\|_{\hat{R}(k)}^{2}}{NL_{n}(k)}.
     \end{split}  
     \end{equation*}
     This implies $\tilde{\mathcal{D}}_{3}^{(1)}\subseteq \mathcal{E}_{1}$, where
     \begin{equation*}
       \mathcal{E}_{1} = \biggl\{\sigma^{2}\frac{K_{n}^{2}}{N} + \frac{K_{n}}{N}A_{n}>  \frac{(k_{n}^{*})^{-\delta}}{24}\biggr\},
     \end{equation*}
     and 
     \begin{equation*}
       A_{n} = \max_{1\leq k\leq K_{n}}\frac{|k \sigma^{2} - N\|\hat{a}(k) - a(k)\|_{\hat{R}(k)}^{2}}{NL_{n}(k)}.
     \end{equation*}
     Note that $\mathcal{E}_{1}\subseteq \{A_{n}> (k_{n}^{*})^{-\delta}\}$ for large enough $n$: On $\mathcal{E}_{1} \cap \{A_{n}\leq (k_{n}^{*})^{-\delta}\}$, we have 
    \begin{equation*}
      \frac{(k_{n}^{*})^{-\delta}}{24} \leq \frac{K_{n}}{N}(k_{n}^{*})^{-\delta} + \sigma^{2} \frac{K_{n}^{2}}{N},
    \end{equation*} 
    which in turn implies
    \begin{equation}\label{eq:thm8-D3-1-reduction}
    \frac{(k_{n}^{*})^{-\delta}}{24} \leq \frac{K_{n}}{N}(k_{n}^{*})^{-\delta} + \sigma^{2} M_{0}(k_{n}^{*})^{-\kappa}
  \end{equation}
  for some constant $M_{0}$, as $K_{n}^{2+\kappa}/N$ is bounded.
  Since $\delta< \kappa/8$, the right-hand side of \eqref{eq:thm8-D3-1-reduction} is eventually smaller than the left-hand side.
  This implies $\mathcal{E}_{1}\cap \{A_{n} \leq (k_{n}^{*})^{-\delta}\} = \varnothing$ for large enough $n$, and hence $\mathcal{E}_{1}\subseteq \{A_{n}> (k_{n}^{*})^{-\delta}\}$ eventually. 
Similarly to $\mathcal{D}_{1}$, there is a constant $C$ such that $\mathbb{P}(A_{n}>(k_{n}^{*})^{-\delta}) \leq C(k_{n}^{*})^{-\gamma}$.
By adjusting constants, we get that there is a (potentially different) constant $C>0$, such that for all $n\in \mathbb{N}$
\begin{equation*}
  \mathbb{P}(\tilde{\mathcal{D}}_{3}^{(1)}) \leq \mathbb{P}(\mathcal{E}_{1}) \leq C(k_{n}^{*})^{-\gamma},
\end{equation*}
where $\gamma>0$ is from Theorem \ref{thm:main}.
     
Next, we deal with $\mathcal{D}_{3}^{(2)}$.
We note that
\begin{equation*}
  W_{2} \leq 2 \max_{1\leq k \leq K_{n}}\frac{s_{k_{n}^{*}}^{2} - \sigma_{k_{n}^{*}}^{2} - s_{k}^{2} + \sigma_{k}^{2}}{\frac{N}{k}L_{n}(k)} \leq \frac{2K_{n}}{N}\max_{1\leq k\leq K_{n}}\frac{s_{k_{n}^{*}}^{2} - \sigma_{k_{n}^{*}}^{2} - s_{k}^{2} + \sigma_{k}^{2}}{L_{n}(k)}.
\end{equation*}
By Lemma \ref{lemma:irs:sk} and a similar argument (cf. \eqref{eq:thm8-D3-1-reduction} and the surrounding text), $\mathcal{D}_{3}^{(2)}$ is eventually a subset of 
\begin{equation*}
  \biggl\{ \max_{1\leq k\leq K_{n}}\frac{s_{k_{n}^{*}}^{2} - \sigma_{k_{n}^{*}}^{2} - s_{k}^{2} + \sigma_{k}^{2}}{L_{n}(k)} > (k_{n}^{*})^{-\delta}\biggr\}.
\end{equation*}
Again, Lemma \ref{lemma:irs:sk} implies
\begin{equation*}
  \mathbb{P}(\mathcal{D}_{3}^{(2)}) \leq \mathbb{P}\biggl(\max_{1\leq k\leq K_{n}}\frac{s_{k_{n}^{*}}^{2} - \sigma_{k_{n}^{*}}^{2} - s_{k}^{2} + \sigma_{k}^{2}}{L_{n}(k)} > (k_{n}^{*})^{-\delta}\biggr) \leq C(k_{n}^{*})^{1+\delta q/2 - q/4}
\end{equation*}
for some $C>0$, finishing the proof.
\end{proof}

    \begin{proof}[Proof. (Theorem \ref{thm:irs-stable})]
      As in Theorem \ref{thm:irs-asymp-eff}, we would like to establish that the set
      \begin{equation*}
        \mathcal{A} = \biggl\{\biggl| \frac{Q_{n}(\hat{k}_{n}(\rho))}{L_{n}(k_{n}^{*})} - 1\biggr| \leq 8(k_{n}^{*})^{-\delta} \biggr\}
      \end{equation*}
      has a probability of at least $1-C(k_{n}^{*})^{-\gamma}$ for some $C,\delta,\gamma>0$. Here, $\delta \in (0,\delta_{0})$ is a positive number satisfying the restrictions of Theorems \ref{thm:main} and \ref{thm:irs-asymp-eff}. 
      The constants $C,C_{i},\gamma, \gamma_{i}>0$ are generic and may vary from appearance to appearance. By the same argument used in the proof of Theorem \ref{thm:irs-stable}, we have $\mathcal{A}\subseteq \mathcal{B}\cap \mathcal{C}$, where
      \begin{equation*}
        \begin{split}
          \mathcal{B} &= \biggl\{\biggl|\frac{Q_{n}(k)}{L_{n}(k)}-1\biggr|\leq  \frac{(k_{n}^{*})^{-\delta}}{2} \biggr\}, \quad \text{and}\\
          \mathcal{C} & = \biggl\{\biggl|\frac{L_{n}(\hat{k}_{n}(\rho))}{L_{n}(k_{n}^{*})} -1\biggr|\leq (k_{n}^{*})^{-\delta}\biggr\}.
        \end{split}
      \end{equation*}
      By Theorem \ref{thm:main}, we have $\mathbb{P}(\mathcal{B}^{c}) \leq C(k_{n}^{*})^{-\gamma}$, for some $C,\gamma>0$.
      Hence, it is enough to show that $\mathbb{P}(\mathcal{C}^{c}) \leq C(k_{n}^{*})^{-\gamma}$.
      As in the proof of Theorem \ref{thm:irs-asymp-eff}, we have
      \begin{equation*}
        \frac{L_{n}(\hat{k}_{n}(\rho))}{L_{n}(k_{n}^{*})} -1 > (k_{n}^{*})^{-\delta} \implies 1 - \frac{L_{n}(k_{n}^{*})}{L_{n}(\hat{k}_{n}(\rho))} > \frac{(k_{n}^{*})^{-\delta}}{1+(k_{n}^{*})^{-\delta}}> \frac{(k_{n}^{*})^{-\delta}}{2},
      \end{equation*}
      and since $S_{n}^{\rho}(k_{n}^{*}) - S_{n}^{\rho}(\hat{k}_{n}(\rho)) \geq 0$, we have $\mathcal{C}^{c}\subseteq \mathcal{E}\cup \mathcal{F}$, where
      \begin{equation*}
        \begin{split}
          \mathcal{E} & = \biggl\{\max_{1\leq k\leq K_{n}}\frac{N(L_{n}(k)-L_{n}(k_{n}^{*})) + S_{n}(k_{n}^{*}) - S_{n}(k)}{NL_{n}(k)}> \frac{(k_{n}^{*})^{-\delta}}{4} \biggr\}, \quad \text{and}\\
          \mathcal{F} & = \biggl\{\max_{1\leq k\leq K_{n}} \frac{S_{n}^{\rho}(k_{n}^{*}) - S_{n}^{\rho}(k) + S_{n}(k_{n}^{*}) - S_{n}(k)}{NL_{n}(k)}> \frac{(k_{n}^{*})^{-\delta}}{4} \biggr\}.
        \end{split}
      \end{equation*}
    In the proof of Theorem \ref{thm:irs-asymp-eff}, we have already seen that $\mathbb{P}(\mathcal{E}) \leq C(k_{n}^{*})^{-\delta}$ for some $C,\gamma>0$, and hence, we only have to deal with $\mathcal{F}$. Since 
    \begin{equation*}
      \begin{split}
        \sigma_{k}^{2} - \sigma^{2} &= \|a-a(k)\|_{R}^{2},\\
        s_{k}^{2} - \hat{\sigma}_{k}^{2} &= \|\hat{a}(k) -a(k)\|_{\hat{R}(k)}^{2}, \quad \text{and}\\
        L_{n}(k) &= \sigma^{2} \frac{k}{N} + \|a-a(k)\|_{R}^{2} = \sigma^{2} \frac{k}{N} + \sigma_{k}^{2} - \sigma^{2},
      \end{split}
    \end{equation*}
    we have
    \begin{equation*}
      \begin{split}
        S_{n}^{\rho}(k) & = S_{n}(k) + \rho_{n}(k)\sigma^{2}\biggl(1 - \frac{2k}{N}\biggr)+ \rho_{n}(k) L_{n}(k)\\
                        & + \rho_{n}(k) \frac{k \sigma^{2} - N\|\hat{a}(k)-a(k)\|_{\hat{R}(k)}^{2}}{N}\\
                        & + \rho_{n}(k)(s_{k}^{2}-\sigma_{k}^{2}).
      \end{split}
    \end{equation*}
    Note that this is just Equation (4.8) in \cite{shibata}, with a typo corrected.
    In total, we can now estimate
    \begin{equation*}
      \max_{1\leq k\leq K_{n}} \frac{|S_{n}^{\rho}(k_{n}^{*}) - S_{n}^{\rho}(k) + S_{n}(k_{n}^{*}) - S_{n}(k)|}{NL_{n}(k)} \leq I_{1} + I_{2}+ I_{3} + I_{4},
    \end{equation*}
    where
    \begin{equation*}
      \begin{split}
        I_{1} & = \sigma^{2}\max_{1 \leq k \leq K_{n}}\frac{|\rho_{n}(k) - \rho_{n}(k_{n}^{*})|}{NL_{n}(k)} + \frac{4}{N}, \\
        I_{2} & =2 \max_{1\leq k\leq K_{n}}\frac{|\rho_{n}(k)|}{N}, \\
        I_{3} & = 2\max_{1\leq k\leq K_{n}}\frac{|\rho_{n}(k)|}{N} \max_{1\leq k \leq K_{n}}\frac{|\sigma^{2}k - N\|\hat{a}(k)-a(k)\|_{\hat{R}(k)}^{2}|}{NL_{n}(k)}, \quad \text{and}\\
        I_{4} & = \max_{1\leq k\leq K_{n}}\frac{|\rho_{n}(k_{n}^{*})(s_{k_{n}^{*}}^{2} - \sigma^{2}_{k_{n}^{*}}) - \rho_{n}(k)(s_{k}^{2} - \sigma_{k}^{2})|}{NL_{n}(k)}.
      \end{split}
    \end{equation*}
    Hence, we can estimate $\mathbb{P}(\mathcal{F}) \leq \sum_{k=1}^{4}\mathbb{P}(\mathcal{F}_{i})$ with $\mathcal{F}_{i} = \{I_{i}>(k_{n}^{*})^{-\delta}/16\}$, and it is enough to show $\mathbb{P}(\mathcal{F}_{i}) \leq C_{i}(k_{n}^{*})^{\gamma_{i}}$ for some $C_{i},\gamma_{i}>0$.

    Starting with $\mathcal{F}_{1}$, we note there is an $n_{0}$ such that for every $n\geq n_{0}$, we have $4/N\leq (k_{n}^{*})^{-\delta}/32$, and hence $\mathcal{F}_{1}$ is contained in 
    \begin{equation*}
      \mathcal{F}_{1}' = \biggl\{ \max_{1 \leq k \leq K_{n}}\frac{|\rho_{n}(k) - \rho_{n}(k_{n}^{*})|}{NL_{n}(k)} > \frac{(k_{n}^{*})^{-\delta}}{32 \sigma^{2}}\biggr\}
    \end{equation*}
    for $n\geq n_{0}$. By assumption, there are $C,\gamma>0$, such that $\mathbb{P}(\mathcal{F}_{1}') \leq C(k_{n}^{*})^{-\gamma}$, and by adjusting constants, we get $\mathbb{P}(\mathcal{F}_{1}) \leq C'(k_{n}^{*})^{-\gamma}$ for some $C'>0$ and all $n\geq 1$.

    Our assumptions directly imply $\mathbb{P}(\mathcal{F}_{2}) \leq C(k_{n}^{*})^{-\gamma}$, and we move on to $\mathcal{F}_{3}$.
    Note that $\mathcal{F}_{3}\subseteq \mathcal{F}_{3}' \cup \mathcal{F}_{3}''$, where 
    \begin{equation*}
      \begin{split}
        \mathcal{F}_{3}' &=  \biggl\{\max_{1\leq k\leq K_{n}}\frac{|\rho_{n}(k)|}{N} > \frac{(k_{n}^{*})^{-\frac{\delta}{2}}}{\sqrt{8}}\biggr\}, \quad \text{and}\\
        \mathcal{F}_{3}'' & = \biggl\{\max_{1\leq k \leq K_{n}}\frac{\bigl|\sigma^{2}k - N\|\hat{a}(k)-a(k)\|_{\hat{R}(k)}^{2}\bigr|}{NL_{n}(k)}> \frac{(k_{n}^{*})^{-\frac{\delta}{2}}}{\sqrt{8}}\biggr\}.
      \end{split}
    \end{equation*}
    In the proof of Theorem \ref{thm:irs-asymp-eff}, we have already seen that $\mathbb{P}(\mathcal{F}_{3}'') \leq C(k_{n}^{*})^{-\gamma}$ for some $C,\gamma>0$, and the assumptions of Theorem \ref{thm:irs-stable} imply $\mathbb{P}(\mathcal{F}_{3}') \leq C(k_{n}^{*})^{-\gamma}$.

    For $\mathcal{F}_{4}$, we insert $\rho_{n}(k)(s_{k_{n}^{*}}^{2} - \sigma_{k_{n}^{*}}^{2})$ into $I_{4}$, and estimate $I_{4}$ by $I_{4}^{1} + I_{4}^{2}$, with
    \begin{equation*}
      \begin{split}
        I_{4}^{1} &= |s_{k_{n}^{*}}^{2}-\sigma_{k_{n}^{*}}^{2}| \max_{1\leq k\leq K_{n}}\frac{|\rho_{n}(k)-\rho_{n}(k_{n}^{*})|}{NL_{n}(k)},\quad \text{and} \\
        I_{4}^{2} & = \max_{1\leq k\leq K_{n}}\frac{|\rho_{n}(k)|}{N}\max_{1\leq k\leq K_{n}} \frac{|s_{k_{n}^{*}}^{2} - \sigma_{k_{n}^{*}}^{2} - s_{k}^{2} - \sigma_{k}^{2}|}{L_{n}(k)}.
      \end{split}
    \end{equation*}
    Again, $\mathcal{F}_{4} \subseteq \mathcal{F}_{4}^{1}\cup \mathcal{F}_{4}^{2}$ with $\mathcal{F}_{4}^{i} = \{I_{4}^{i}> (k_{n}^{*})^{-\delta}/32\}$, and we show $\mathbb{P}(\mathcal{F}_{4}^{i}) \leq C_{i}(k_{n}^{*})^{-\gamma_{i}}$ for some $C_{i},\gamma_{i}>0$.
    The event $\mathcal{F}_{4}^{2}$ can be dealt with much in the same way as $\mathcal{F}_{3}$, that is, $\{I_{4}^{2} > (k_{n}^{*})^{-\delta}/32\}$ is a subset of 
    \begin{equation*}
      \biggl\{\max_{1\leq k\leq K_{n}}\frac{|\rho_{n}(k)|}{N}> \frac{(k_{n}^{*})^{-\frac{\delta}{2}}}{\sqrt{32}}\biggr\} \cup \biggl\{\max_{1\leq k\leq K_{n}} \frac{|s_{k_{n}^{*}}^{2} - \sigma_{k_{n}^{*}}^{2} - s_{k}^{2} - \sigma_{k}^{2}|}{L_{n}(k)} >\frac{(k_{n}^{*})^{-\frac{\delta}{2}}}{\sqrt{32}}\biggr\}.
    \end{equation*}
    By Lemma \ref{lemma:irs:sk} and the assumptions of Theorem \ref{thm:irs-stable}, the probability of both of these sets is $\leq C(k_{n}^{*})^{-\gamma}$ for some $C,\gamma>0$, and the same is true for $\mathcal{F}_{4}^{2}$ by the union bound.

    Moving on to $\mathcal{F}_{4}^{1}$, we note that a rough estimate yields 
    \begin{equation*}
      \begin{split}
        |s_{k_{n}^{*}}^{2}-\sigma_{k_{n}^{*}}^{2}|& = \|\hat{a}(k_{n}^{*})-a(k_{n}^{*})\|_{\hat{R}(k_{n}^{*})}^{2} \\
                                                  &\leq L_{n}(k_{n}^{*}) \max_{1\leq k \leq K_{n}}\frac{|\sigma^{2}k - N\|\hat{a}(k)-a(k)\|_{R(k)}^{2}|}{NL_{n}(k)} + \sigma^{2}\frac{k_{n}^{*}}{N}\\
                                                  &\leq \sigma^{2}\biggl( \max_{1\leq k \leq K_{n}}\frac{|\sigma^{2}k - N\|\hat{a}(k)-a(k)\|_{R(k)}^{2}|}{NL_{n}(k)} + \frac{1}{k_{n}^{*}} \biggr).
      \end{split}
    \end{equation*}
    This means that $\mathcal{F}_{4}^{1}\subseteq \mathcal{G}_{1}\cup \mathcal{G}_{2}$, where 
    \begin{equation*}
      \begin{split}
        \mathcal{G}_{1} & = \biggl\{ \max_{1\leq k\leq K_{n}}\frac{\sigma^{2}k - N\|\hat{a}(k)-a(k)\|_{\hat{R}(k)}^{2}|}{NL_{n}(k)}\max_{1\leq k\leq K_{n}}\frac{|\rho_{n}(k) - \rho_{n}(k_{n}^{*})}{NL_{n}(k)} > \frac{(k_{n}^{*})^{-\delta}}{64 \sigma^{2}}\biggr\}, \quad \text{and} \\
        \mathcal{G}_{2} & = \biggl\{ \max_{1\leq k\leq K_{n}}\frac{|\rho_{n}(k) - \rho_{n}(k_{n}^{*})}{NL_{n}(k)} > \frac{(k_{n}^{*})^{-\delta}}{64 \sigma^{2}}\biggr\}.
      \end{split}
    \end{equation*}
    By a similar argument as was used for $\mathcal{F}_{3}$ or $\mathcal{F}_{4}^{2}$, we have $\mathbb{P}(\mathcal{G}_{1}) \leq C(k_{n}^{*})^{-\gamma}$, and by the assumptions of Theorem \ref{thm:irs-stable}, we have $\mathbb{P}(\mathcal{G}_{2}) \leq C(k_{n}^{*})^{-\gamma}$ for some $C,\gamma >0$, finishing the proof. 
  \end{proof}

\subsection{Proof of Theorem \ref{thm:normal:distribution}}\label{sec:proof:normal}

We follow the classical approach given in \cite{davis:brockwell:book}, where the proof is split into two parts. We first consider a \textit{regression-type} estimator $a^{*}(k)$, and show the asymptotic normality of $\sqrt{n}(a^{*}(k) - a(k))$ (Lemma \ref{lemma:thm-an-1}). 
Then we show that the difference between the Yule-Walker estimator $\hat{a}(k)$ and $a^{*}(k)$ is asymptotically negligible, i.e., $\sqrt{n}(\hat{a}(k) - a^{*}(k)) \overset{\mathbb{P}}{\longrightarrow} 0$ (Lemma \ref{lemma:thm-an-2}).
To be more precise, let $X = (X_{i-j})_{i,j=0}^{n-1,k-1}$ and $a^{*}(k) = (X^{T}X)^{-1}X^{T}X_{n+1}(k)$. 
With arguments almost identical to those used in the proof of Lemma \ref{lemma:mat-norm}, we get for $1 \leq p \leq q/2$ ($q$ being taken from Assumption \ref{I1})
\begin{equation*}
  \begin{split}
    \frac{1}{n}X^{T}X &\inLplong{p} R(k), \\
    \frac{1}{n}X^{T}X_{n+1} & \inLplong{p} \gamma(k), \quad \text{and}\\
    n(X^{T}X)^{-1} &\inPlong R(k)^{-1}.
  \end{split}
\end{equation*}

\begin{lemma}\label{lemma:thm-an-1} Given Assumptions \ref{G1}, \ref{G2}, \ref{I1}, and \ref{I2}, we have
  \begin{equation*}
    \sqrt{n}(a^{*}(k) - a(k)) \weakly \mathcal{N}(0,\Sigma(k)),
  \end{equation*}
  where $\Sigma(k) = \sum_{h\in \mathbb{Z}} R(k)^{-1} \mathbb{E}(e_{0}e_{h}X_{0}(k)X_{h}(k)^{T})R(k)^{-1}$.
\end{lemma}

\begin{proof}

Note that we have $X_{t+1}(k) = Xa(k) + Z$, where $Z = (e_{t+1},\dots,e_{t+1-k})^{T}$.
This implies
\begin{equation*}
  \sqrt{n}(a^{*}(k)-a(k)) = \sqrt{n}\bigl[(X^{T}X)^{-1}X^{T}(Xa(k)+Z) - a(k)\bigr] = n(X^{T}X)^{-1}\frac{1}{\sqrt{n}}X^{T}Z.
\end{equation*}
Since $n(X^{T}X)^{-1} \inPlong R(k)^{-1}$, we restrict our attention to $n^{-1/2}X^{T}Z$.
Setting $U_{t}= X_{t}(k)e_{t}$, we can rewrite this as
\begin{equation*}
  \frac{1}{\sqrt{n}}X^{T}Z = \frac{1}{\sqrt{n}}\sum_{t=1}^{n}U_{t}.
\end{equation*}
Note that $\mathbb{E}(U_{t})=0$. For $\lambda\in \mathbb{R}^{k}$, we set $A_{t} = \lambda^{T}U_{t}$.
By arguments similar to Inequality \eqref{eq:pd-prod-est}, we get the following estimate for $p \leq q/2$
\begin{equation*}
  \|A_{l}-A_{l}'\|_{p} \leq \sum_{j=1}^{k}|\lambda_{j}| \delta_{2p}^{X}(l-i) \|e_{0}\|_{2p} + \|X_{0}\|_{2p}\|\lambda\|_{\ell^{1}}\delta_{2p}^{I}(l).
\end{equation*}
This immediately implies 
\begin{equation*}
  \sum_{l=1}^{\infty}l^{\frac{5}{2}}\|A_{l}-A_{l}'\|_{p} < \infty,
\end{equation*}
by Assumption \ref{I2}.
Theorem 3 of \cite{wu2011asymptotic} then implies
\begin{equation*}
\frac{1}{\sqrt{n}}\sum_{t=1}^{n}\lambda^{T}U_{t} \weakly \mathcal{N}(0,\sigma_{\lambda}^{2}),
\end{equation*}
where $\sigma_{\lambda}^{2} = \sum_{h\in \mathbb{Z}} \mathbb{E}(\lambda^{T}U_{0}U_{h}^{T}\lambda)$.
An application of the Cramer-Wold device yields
\begin{equation*}
  \frac{1}{\sqrt{n}}\sum_{t=1}^{n}U_{t} \weakly N(0,\Sigma^{*}(k)),
\end{equation*}
where $\Sigma^{*}(k) = \sum_{h\in \mathbb{Z}} \mathbb{E}(e_{0}e_{h}X_{0}(k)X_{h}(k)^{T})$.
Using this, we get
\begin{equation*}
  \sqrt{n}(a^{*}(k)-a(k)) = n(X^{T}X)^{-1}\frac{1}{\sqrt{n}}X^{T}Z \weakly \mathcal{N}(0,\Sigma(k)).
\end{equation*}
\end{proof}

\begin{lemma}\label{lemma:thm-an-2}
Given Assumptions \ref{ass:main} and \ref{ass:weak:dep}, we have
\begin{equation*}
  \sqrt{n}(\hat{a}(k)-a^{*}(k)) \inPlong 0.
\end{equation*}
\end{lemma}

\begin{proof}
  We rewrite $\sqrt{n}(\hat{a}(k) - a^{*}(k))$ as 
  \begin{equation*}
    \begin{split}
      \sqrt{n}(\hat{a}(k) - a^{*}(k)) &= \sqrt{n}\bigl[\hat{R}(k)^{-1}\hat{\gamma}(k) - (X^{T}X)^{-1}X^{T}Y\bigr]\\
                                      & = \sqrt{n}[\hat{R}(k)^{-1}(\hat{\gamma}(k) - n^{-1}X^{T}X_{n+1}\bigl(k\bigr)] \\
                                      & + \sqrt{n} \bigl[\hat{R}(k)^{-1} - n(X^{T}X)^{-1}\bigr]n^{-1}X^{T}X_{n+1}(k).
    \end{split}
  \end{equation*}
  Note that the  $i$-th component of $\hat{\gamma}$ is given by $\frac{1}{N}\sum_{t=K_{n}}^{n-1}X_{t-i}X_{t+1}$, while the $i$-th component of $n^{-1}X^{T}X_{n+1}(k)$ is given by $\frac{1}{n}\sum_{t=0}^{n-1}X_{t-i}X_{t+1}$.
  Hence the difference between their $i$-th components can be written as
  \begin{equation*}
    \frac{1}{n}\sum_{t=K_{n}}^{n-1}\biggl(\frac{n}{N}-1\biggr)X_{t-i}X_{t+1} - \frac{1}{n}\sum_{t=0}^{K_{n}-1} X_{t-i}X_{t+1} = I + II.
  \end{equation*}
  We have by Assumption \ref{G3}
\begin{equation*}
  \sqrt{n}\|I\|_{p} \lesssim \sqrt{n}\biggl(\frac{n}{n-K_{n}} -1\biggr) = \sqrt{n}\frac{K_{n}}{n-K_{n}}\to 0,
\end{equation*}
as well as
\begin{equation*}
  \sqrt{n}\|II\|_{p} \lesssim \frac{K_{n}}{\sqrt{n}} \to 0,
\end{equation*}
finishing the proof of Theorem \ref{thm:normal:distribution}.
\end{proof}

\subsection{2-forms of weakly dependent processes}\label{sec:2-forms-weakly-proc}
\begin{proof}[\proofof{Lemma \ref{lemma:mat-norm}}]
First, note that for $1 < p \leq 2$
    \begin{equation*}
        \mathbb{E}(\|\hat{R}(K_n) - R(K_n)\|_F^p) \leq \sum_{i,j=1}^{K_n} \mathbb{E}(|\hat{R}_{ij} - R_{ij}|^p),
    \end{equation*}
    while Lyapunov's inequality implies 
    \begin{equation*}
        \mathbb{E}(\|\hat{R}(K_n) - R(K_n)\|_F^p) \leq K_n^{p-2}\sum_{i,j=1}^{K_n} \mathbb{E}(|\hat{R}_{ij} - R_{ij}|^p)
    \end{equation*}
    for $p>2$. An application of Remark \ref{rem:prod-pd} yields 
    \begin{equation}\label{eq:R^_ij - R_ij est}
    \begin{split}
        \|\hat{R}_{ij} - R_{ij}\|_p & = \frac{1}{N}\biggl\|\sum_{t=K_n+1}^n X_{t-i}X_{t-j} - R_{ij}\biggr\|_p \leq 2\|X_{0}\|_{2p}D_{2p}^{X}(0)B_p N^{\frac{1}{p'}-1}. 
    \end{split}
    \end{equation}
    Thus the $L^p$ norm of $\|\hat{R}(k) - R(k)\|_F$ is bounded by 
\begin{equation*}
    \bigl\|\|\hat{R}(k) - R(k)\|_F \bigr\|_p \lesssim K_n^{\frac{2}{p'}} N^{\frac{1}{p'}-1} = \bigl(K_n^2N^{1-p'} \bigr)^{\frac{1}{p'}} \to 0,
\end{equation*}
for any $p>1$.
Claim \ref{mat-norm2} and \ref{mat-norm3} follows in the same way as in Lemma 3 of \cite{berk}. However, we require explicit rates in terms of $n$ and $K_{n}$ and hence provide the proof. 
  As in \cite{berk}, we set $q(k)=\|\hat{R}(k)^{-1} - R(k)^{-1}\|$, $r(k) = \|R(k)^{-1}\|$, and $Q(k) = \|\hat{R}(k) - R(k)\|$, and get
  \begin{equation*}
    q(k) \leq (r(k)+q(k))r(k)Q(k).
  \end{equation*}
  Additionally, we set $Q'(k) = \|\hat{R}(k)-R(k)\|_{F}$, and note that $Q(k) \leq Q'(k) \leq Q'(K_{n})$.
  Since the eigenvalues of $R(k)$ are uniformly bounded aways from zero, by Assumption \ref{ass:main}, there is a constant $\Lambda>0$, such that $r(k)\leq \Lambda$ for all $k$.
  Hence, if $\Lambda Q'(K_{n}) \leq 1/2$, then also $r(k)Q(k) \leq 1/2$ for all $k = 1,\dots, K_{n}$, and a similar argument as in Lemma 3 of \cite{berk} yields
  \begin{equation*}
    q(k) \leq \frac{r(k)^{2}Q(k)}{1-r(k)Q(k)}\leq 2r(k)^{2}Q(k) \leq 2 \Lambda^{2}Q'(K_{n}),
  \end{equation*}
  for all $k = 1,\dots, K_{n}$. 
  Now, let $c>0$ be an arbitrary positive number, and set $q_{n}^{*} = \max_{1\leq k\leq K_{n}}q(k)$.
  A simple case distinction implies
  \begin{equation*}
    \begin{split}
      \mathbb{P}(q_{n}^{*}> K/(k_{n}^{*})^{c}) &= \mathbb{P}(q_{n}^{*}> K/(k_{n}^{*})^{c}, \Lambda Q'(K_{n}) > 1/2) + \mathbb{P}(q_{n}^{*}> K/(k_{n}^{*})^{c}, \Lambda Q'(K_{n}) \leq 1/2)\\
                                              & \leq \mathbb{P}(Q'(K_{n}) > 1/(2 \Lambda)) + \mathbb{P}(Q'(K_{n}) > K/(2(k_{n}^{*})^{c} \Lambda^{2})). 
    \end{split}
  \end{equation*}
 Markov's inequality, the first part of the Lemma, and the fact that $k_{n}^{*}\geq 1$ give
  \begin{equation*}
    \begin{split}
      \mathbb{P}(q_{n}^{*}>K/(k_{n}^{*})^{c}) &\leq (2 \Lambda)^{p}\bigl(1 + (\Lambda K)^{p}(k_{n}^{*})^{cp}\bigr)\mathbb{E}(\|\hat{R}(k) - R(k)\|_{F}^{p}) \\
                                             & \leq C(K) (k_{n}^{*})^{cp} (K_{n}^{2}N^{1-p'})^{\frac{p}{p'}},
    \end{split}
  \end{equation*}
  for any $p>1$, and some constant $C(K)$ depending on $K$.
\end{proof}   

\begin{proof}[\proofof{Lemma \ref{lemma:QF-1}}]
In the case $i\geq j$, a rough estimate is sufficient. An application of the triangle inequality yields
    \begin{equation*}
        \biggl\|\sum_{j=0}^{K_n} \sum_{i=j}^{K_n} \xi_i \eta_j (\hat{R}_{ij}-R_{ij})\biggr\|_p \leq \sum_{j=0}^{K_n} |\eta_j|\sum_{i=j}^{K_n}|\xi_i|  \|\hat{R}_{ij}-R_{ij}\|_p.
    \end{equation*}
    From Equation \eqref{eq:R^_ij - R_ij est} we know that $\|\hat{R}_{ij} - R_{ij}\|_p \lesssim N^{\frac{1}{p'}-1}$ for $p'=\min\{p,2\}$ (the constants do not depend on $i$ and $j$).
    An application of the Cauchy-Schwarz inequality and rearranging terms yields
    \begin{equation*}
    \begin{split}
        \biggl\|\sum_{j=0}^{K_n} \sum_{i=j}^{K_n} \xi_i \eta_j (\hat{R}_{ij}-R_{ij})\biggr\|_p & \lesssim N^{\frac{1}{p'}-1}\|\eta\|_{\ell^2} \biggl(\sum_{j=0}^{K_n} \biggl(\sum_{i=j}^{K_n} |\xi_i|\biggr)^2\biggr)^{\frac{1}{2}}\\
        & \lesssim N^{\frac{1}{p'}-1}\|\eta\|_{\ell^2}\|\xi\|_{\ell^1}^{\frac{1}{2}} \biggl(\sum_{i=0}^{K_n} i |\xi_i|\biggr)^{\frac{1}{2}}.
    \end{split}
    \end{equation*}
    A slightly more delicate analysis is necessary in the case $i< j$. Again, we start with the triangle inequality
    \begin{equation*}
        \biggl\|\sum_{i=0}^{K_n-1} \xi_i \sum_{j=i+1}^{K_n} \eta_j (\hat{R}_{ij}-R_{ij})\biggr\|_p \leq \sum_{i=0}^{K_n-1} |\xi_i| \biggl\|  \sum_{j=i+1}^{K_n} \eta_j (\hat{R}_{ij}-R_{ij})\biggr\|_p.
    \end{equation*}
    For the rest of the proof, we focus on the expression in the $L^p$ norm. 
    Since $\mathbb{E}(\hat{R}_{ij}) = R_{ij}$, we can use a martingale approximation in the sense of Lemma \ref{lemma:mdsa},

    \begin{equation*}
        \biggl\|\sum_{j=i+1}^{K_n} \eta_j (\hat{R}_{ij}-R_{ij})\biggr\|_p  = \frac{1}{N}\biggl\|\sum_{t=K_n+1}^n \sum_{j=i+1}^{K_n} \eta_j \sum_{l=0}^\infty \mathcal{P}_{t-i-l}(X_{t-i}X_{t-j})\biggr\|_p.
    \end{equation*}
    We may split the sum over $l$ into two sums: The first one deals with terms where $0 \leq l <j-i$, and the second one deals with the case $l> j-i$. 
    Starting with the first case, we note that for $l < j-i$, $X_{t-j}$ is $\sigf{t-i-l-1}$ measurable. Hence, we have 
    \begin{equation*}
        \mathcal{P}_{t-i-l}(X_{t-i}X_{t-j}) = X_{t-j}\mathcal{P}_{t-i-l}(X_{t-i}).
    \end{equation*}
    Using this and changing the order of summation, we may estimate
    \begin{multline}\label{eq:lemma-cuckcoo-l<j-i:1}
        \biggl\|\frac{1}{N} \sum_{t=K_n+1}^{K_n} \sum_{j=i+1}^{K_n} \eta_j\sum_{l=0}^{j-i-1}\mathcal{P}_{t-i-l}(X_{t-i}X_{t-j}) \biggr\|_p \\\leq \frac{1}{N}\sum_{l=0}^{K_n-i-1}\biggl\| \sum_{t=K_n+1}^{K_n} \sum_{j=i+1+l}^{K_n} \eta_j\mathcal{P}_{t-i-l}(X_{t-i})X_{t-j} \biggr\|_p. 
    \end{multline}
    Since $\mathcal{P}_{t-i-l}(X_{t-i}\sum_{j=i+1+l}^{K_n} \eta_j X_{t-j})$ is a martingale difference sequence with respect to $t$, we may apply Burkholder's inequality (Lemma \ref{lemma:burkholder})
    \begin{equation*}
      \biggl\| \sum_{t=K_n+1}^{K_n} \sum_{j=i+1+l}^{K_n} \eta_j\mathcal{P}_{t-i-l}(X_{t-i})X_{t-j} \biggr\|_p \leq \frac{B_p}{N} \biggl(\sum_{t=K_n+1}^{K_n} \biggl\|\sum_{j=i+1+l}^{K_n} \eta_j\mathcal{P}_{t-i-l}(X_{t-i})X_{t-j}\biggr\|_p^{p'}\biggr)^{\frac{1}{p'}}.
    \end{equation*}
    Since $X$ is stationary and $\mathcal{P}_{t-i-l}(X_{t-i}) \disteq \mathcal{P}_{0}(X_{l})$ (see \eqref{eq:P(Z)=E(Z-Z')} and the subsequent discussion), we have
    \begin{equation*}
      \biggl\|\sum_{j=i+1+l}^{K_n} \eta_j\mathcal{P}_{t-i-l}(X_{t-i})X_{t-j}\biggr\|_p \leq  \|\mathcal{P}_{0}(X_{l})\|_{2p}\biggl\|\sum_{j=i+1+l}^{K_n} \eta_jX_{t-j+i}\biggr\|_{2p} \leq \delta_{2p}^{X}(l)\|\eta\|_{\ell^{2}}B_{2p}D_{2p}^{X}(0),
    \end{equation*}
    by Lemma \ref{lemma:PD-standart-est}.
    which does not depend on $t$ anymore.

    Since $(\kappa_{s})_{s\geq 0} \in \ell^{2}$ we have in particular $\kappa_{s}<\infty$.
    A simple integral comparison implies that $D_{2p}^{X}(0)<\infty$. Hence, the sum over $l$ is bounded by $D_{2p}^{X}(0)<\infty$.
    Putting these estimates back into Inequality \eqref{eq:lemma-cuckcoo-l<j-i:1}, we get 
    \begin{equation*}
      \biggl\|\frac{1}{N} \sum_{t=K_n+1}^{K_n} \sum_{j=i+1}^{K_n} \eta_j\sum_{l=0}^{j-i-1}\mathcal{P}_{t-i-l}(X_{t-i}X_{t-j}) \biggr\|_p \leq \|\eta\|_{\ell^2}N^{\frac{1}{p'}- 1}B_pB_{2p}D_{2p}^X(0)^2. 
    \end{equation*}
    This finishes the case $0\leq l < j-i$.  

    Moving on to the case $l\geq j-i$, we have to deal with the expression
    \begin{equation*}
        \frac{1}{N} \biggl\|\sum_{t=K_n+1}^n \sum_{j=i+1}^{K_n} \eta_j \sum_{l=j-i}^\infty \mathcal{P}_{t-i-l}(X_{t-i}X_{t-j})\biggr\|_p,
    \end{equation*}
    which we estimate using Burkholder's inequality (Lemma \ref{lemma:burkholder}), and the fact that $\mathcal{P}_{t-i-l}(X_{t-i}X_{t-j}) \disteq \mathcal{P}_{0}(X_{l}X_{l-j+i})$, by 
    \begin{equation*}
      \leq \frac{1}{N}\sum_{j=i+1}^{K_{n}}|\eta_{j}| \sum_{l=j-i}^{\infty} \biggl\|\sum_{t=K_{n+}+1}^{n}\mathcal{P}_{t-i-l}(X_{t-i}X_{t-j})\biggr\|_{p} \lesssim N^{\frac{1}{p'}-1}\sum_{j=i+1}^{K_{n}} |\eta_{j}| \sum_{l=j-i}^{\infty}\|\mathcal{P}_{0}(X_{l}X_{l-j+i})\|_{p}.
    \end{equation*}
    If the sequence $L = (L_{s})_{s\geq1}$ with $L_{s} = \sum_{l=s}^{\infty}\|\mathcal{P}_{0}(X_{l}X_{l-s})\|_{p}$ is in $\ell^{2}$, Cauchy-Schwarz implies
    \begin{equation*}
      N^{\frac{1}{p'}-1}\sum_{j=i+1}^{K_{n}} |\eta_{j}| \sum_{l=j-i}^{\infty}\|\mathcal{P}_{0}(X_{l}X_{l-j+i})\|_{p} \leq N^{\frac{1}{p'}-1}\|\eta\|_{\ell^{2}}\|L\|_{\ell^{2}} \lesssim N^{\frac{1}{p'}-1}\|\eta\|_{\ell^{2}},
    \end{equation*}
    finishing the proof.
    We again distinguish two cases, $l=s$ and $l>s$.
    For $l=s$ we note that $H_{1}^{s}(X_{s}) = \mathbb{E}(X_{s}\mid \sigfdouble{s}{1})$ is independent of $\sigf{0}$, and hence
    \begin{equation*}
      \mathcal{P}_{0}(H_{1}^{s}(X_{s})X_{0}) = \mathbb{E}(X_{s})\mathcal{P}_{0}(X_{0}) =0,
    \end{equation*}
    Therefore we have 
    \begin{equation*}
      \mathcal{P}_{0}(X_{s}X_{0}) = \mathcal{P}_{0}(Q_{1}^{s}(X_{s})X_{0}),
    \end{equation*}
    for $Q_{1}^{s}(X_{s}) = X_{s} - H_{1}^{s}(X_{s})$.
    Using Hölder's inequality and Lemma \ref{lemma:Q}, we get 
    \begin{equation}\label{eq:lemma-Q:l>=j-i:l=s}
      \|\mathcal{P}_{0}(X_{s}X_{0})\|_{p} \leq 2\|X_{0}\|_{2p}\|Q_{1}^{s}(X_{s})\|_{2p} \lesssim \sqrt{\sum_{m=s}^{\infty}\delta_{2p}^{X}(m)^{2}}.
    \end{equation}
    For the case $l>s$, a similar argument as above yields
    \begin{equation*}
      \mathcal{P}_{0}(X_{l}X_{l-s}) = \mathcal{P}_{0}(Q_{l-s+1}^{l}(X_{l})X_{l-s}).
    \end{equation*}
    Now we write $X_{l-s} = Q_{1}^{l-s}(X_{l-s}) + H_{1}^{l-s}(X_{l-s})$, which gives
    \begin{equation*}
      \mathcal{P}_{0}(X_{l}X_{l-s}) = \mathcal{P}_{0}(Q^{l}_{l-s+1}(X_{l})Q_{1}^{l-s}(X_{l-s})) + \mathcal{P}_{0}(Q_{l-s+1}^{l}(X_{l})H_{1}^{l-s}(X_{l-s})).
    \end{equation*}
    Observe that $H_{l-s+1}^{l}(X_{l})H_{1}^{l-s}(X_{l-s})$ is independent of $\sigf{0}$, which implies 
    \begin{equation*}
      \mathcal{P}_{0}(H_{l-s+1}^{l}(X_{l})H_{1}^{l-s}(X_{l-s})) = 0.
    \end{equation*}
    Since $H_{1}^{l}(X_{l})H_{1}^{l-s}(X_{l-s})$ is also independent of $\sigf{0}$, we have
    \begin{equation*}
      \mathcal{P}_{0}(Q^{l}_{l-s+1}(X_{l})H_{1}^{l-s}(X_{l-s})) = \mathcal{P}_{0}(X_{l}H_{1}^{l-s}(X_{l-s})) = \mathcal{P}_{0}(Q_{1}^{l}(X_{l})H_{1}^{l-s}(X_{l-s})). 
    \end{equation*}
    In total, we get 
    \begin{equation*}
      \mathcal{P}_{0}(X_{l}X_{l-s}) = \mathcal{P}_{0}(Q_{l-s+1}^{l}(X_{l})Q_{1}^{l-s}(X_{l-s})) + \mathcal{P}_{0}(Q_{1}^{l}(X_{l})H_{1}^{l-s}(X_{l-s})).
    \end{equation*}
    As before, Lemma \ref{lemma:Q} implies
    \begin{equation*}
      \|\mathcal{P}_{0}(X_{l}X_{l-s})\|_{p} \lesssim \sqrt{\sum_{m=s}^{\infty}\delta_{2p}^{X}(m)^{2}} \sqrt{\sum_{h=l-s}^{\infty}\delta_{2p}^{X}(h)^{2}} + \sqrt{\sum_{k=l}^{\infty}\delta_{2p}^{X}(k)^{2}}.
    \end{equation*}
    Combining this with Estimate \eqref{eq:lemma-Q:l>=j-i:l=s} yields
    \begin{equation*}
      \begin{split}
        L_{s} &\lesssim \sqrt{\sum_{m=s}^{\infty}\delta_{2p}^{X}(m)^{2}}\biggl(1+\sum_{l=1}^{\infty}\sqrt{\sum_{h=l}^{\infty}\delta_{2p}^{X}(h)^{2}}\biggr) + \sum_{l=s+1}^{\infty}\sqrt{\sum_{k=l}^{\infty}\delta_{2p}^{X}(k)^{2}}\\
              & \lesssim \sum_{l=s}^{\infty}\sqrt{\sum_{k=l}^{\infty}\delta_{2p}^{X}(k)^{2}},
      \end{split}
    \end{equation*}
    finishing the proof.
\end{proof}

     \subsection{Proofs of Lemma \ref{lemma:Bn-Bnk-rate}, \ref{lemma:Bn-Bnk}, and \ref{lemma:irs:sk}}\label{sec:lemmas-for-main-theorem}

  \begin{proof}[\proofof{Lemma \ref{lemma:Bn-Bnk-rate}}]
Since the $\|\cdot\|_{R}$ norm and the Euclidean norm $\|\cdot\|_{\ell^{2}}$ are equivalent, we may as well look at $\|B_{n} - B_{nk}\|_{\ell^{2}}$ instead. 
Setting $\Delta_t(k) = e_t - \pino{t}{k}$, we write
\begin{equation*}
    \begin{split}
        \|B_n - B_{nk}\|_{\ell^2}^2 &= \langle B_n - B_{nk}, B_n - B_{nk}\rangle \\
        &= \frac{1}{N^2}\sum_{s,t=K_n+1}^n \Delta_t(k) \Delta_s(k) \biggl \langle \begin{pmatrix}
      X_{t-1}\\
      \vdots\\
      X_{t-k}
    \end{pmatrix}, \begin{pmatrix}
      X_{s-1}\\
      \vdots\\
      X_{s-k}
    \end{pmatrix} \biggr\rangle \\
    & = \frac{1}{N^2}\sum_{j=1}^k \biggl( \sum_{t=K_n+1}^n \Delta_t(k) X_{t-j}\biggr)^2.
    \end{split}
\end{equation*}
By the triangle inequality, we have
\begin{equation}\label{eq:Bn-Bnk:rate-triangle-ineq-general}
\begin{split}
  \bigl\| \|B_n-B_{nk}\|_{\ell^{2}}^2\bigr\|_{p} &\leq \frac{1}{N^{2}}\sum_{j=1}^k \biggl\|\biggl(\sum_{t=K_n+1}^n \Delta_t(k) X_{t-j}\biggr)^2\biggr\|_p \\
    & =\frac{1}{N^{2}}\sum_{j=1}^k \biggl\|\sum_{t=K_n+1}^n \Delta_t(k) X_{t-j}\biggr\|_{2p}^2. \\
\end{split}
\end{equation}

We would like to apply Lemma \ref{lemma:PD-standart-est} to the term in the $L^{2p}$-norm. 
To this end, we need to check whether $\mathbb{E}(\Delta_t(k)X_{t-j}) = \mathbb{E}(e_tX_{t-j}) - \mathbb{E}(\pino{t}{k}X_{t-j}) = 0$. 
By Lemma \ref{lemma:brillinger}, we can write $X_{t}$ as 
\begin{equation*}
    X_t = e_t + \sum_{m=1}^\infty b_m e_{t-m}. 
\end{equation*}
Since the innovations are uncorrelated, and $(b_{m})_{m\geq 0} \in \ell^{1}$, this yields
\begin{equation*}
    \mathbb{E}(e_tX_{t-j}) = \sum_{m=0}^{\infty} b_m \mathbb{E}(e_te_{t-j-m}) = 0,
\end{equation*}
for $j>0$.
The pseudo-innovations $\pino{t}{k}$ on the other hand are defined as the projections onto the orthogonal complement of $V_k = \langle X_{t-1}, \dots, X_{t-k}\rangle$.  
Since $\pino{t-j}{k-j} \in V_k$ we have $\mathbb{E}(\pino{t}{k}\pino{t-k}{k-j}) = 0$. Applying Lemma \ref{lemma:PD-standart-est} to the left-hand side of \eqref{eq:Bn-Bnk:rate-triangle-ineq-general}, yields
\begin{equation}\label{eq:Bn-Bnk:after-Burkholder}
    \frac{1}{N} \sum_{j=1}^k \biggl\|\sum_{t=K_n+1}^n \Delta_t(k) X_{t-j}\biggr\|_{2p}^2 \lesssim \frac{1}{N} \sum_{j=1}^k \sum_{l=0}^\infty \|\Delta_l(k)X_{l-j} - \Delta_l(k)'X_{l-j}'\|_{2p}^2. 
\end{equation}
We write $d_{m}(k) = a_{m}-a_{m}(k)$. Note that  
\begin{equation*}
  \begin{split}
    \|\Delta_{l}(k)X_{l-j} - \Delta_{l}(k)'X_{l-j}'\|_{2p} & = \biggl\|\sum_{m=1}^{\infty}d_{m}(k)(X_{l-m}X_{l-j}-X_{l-m}'X_{l-j}')\biggr\|_{2p} \\
                                                           & \leq \sum_{m=1}^{\infty}|d_{m}(k)| \|X_{l-m}X_{l-j} - X_{m-l}'X_{l-j}'\|_{2p}\\
                                                           &\leq \|X_{0}\|_{4p} \sum_{m=1}^{\infty}|d_{m}(k)|(\delta_{4p}^{X}(l-m) + \delta_{4p}^{X}(l-j)).
  \end{split}
\end{equation*}
Thus, we may estimate the right-hand side of Inequality \eqref{eq:Bn-Bnk:after-Burkholder} by 
\begin{equation*}
  \begin{split} 
  &\frac{1}{N}\sum_{m_{1},m_{2}=1}^{\infty} |d_{m_{1}}(k)| |d_{m_{2}}(k)| \sum_{j=1}^{k}\sum_{l=0}^{\infty} (\delta_{4p}^{X}(l-m_{1}) + \delta_{4p}^{X}(l-j))(\delta_{4p}^{X}(l-m_{2}) + \delta_{4p}^{X}(l-j)) \\
  &\lesssim \frac{k}{N}\|a-a(k)\|_{\ell^{1}}^{2},
  \end{split}
\end{equation*}
as $\delta_{4p}^{X}(l) = 0$ for $l<0$. 
\end{proof}

\begin{proof}[\proofof{Lemma \ref{lemma:Bn-Bnk}}]
    If we can show that
    \begin{equation}\label{eq:Bn-Bnk-rate}
        \frac{N}{k} \bigl\| \|R(k)^{-1}(B_n-B_{nk})\|_R^2\bigr\|_{p} \lesssim k^{-\frac{1}{2}},
    \end{equation}
    where $p=q/4 >2$, then Markov's inequality implies the following estimate 
    \begin{equation}\label{eq:Bn-Bnk-Markov-est}
    \begin{split}
        \mathbb{P}\bigl(A_3^2 > \varepsilon\bigr) & \leq \frac{1}{\varepsilon^p} \mathbb{E}\biggl(\max_{1\leq k \leq K_n} \|R(k)^{-1}(B_n - B_{nk})\|_R^{2p}L_n(k)^{-p}\biggr)\\
        &\leq \frac{1}{\varepsilon^p} \sum_{k=1}^{K_n} \biggl(\frac{N}{k}L_n(k)\biggr)^{-p}\biggl(\frac{N}{k}\bigl\| \|R(k)^{-}(B_n-B_{nk})\|_R^2\bigr\|_{p}\biggr)^p \\
        & \lesssim \frac{1}{\varepsilon^{p}} \sum_{k=1}^{K_n} \biggl(\frac{N}{k}L_n(k)\biggr)^{-p} k^{-\frac{p}{2}}\\
        & = \frac{1}{\varepsilon^{p}}\sum_{k=1}^{k_n^*} \underbrace{\biggl(\frac{N}{k}L_n(k)\biggr)^{-p}}_{\leq (k_n^* \frac{1}{k} \sigma^2)^{-p}} k^{-\frac{p}{2}} + \frac{1}{\varepsilon^{p}}\sum_{k=k_n^*+1}^{K_n} \underbrace{\biggl(\frac{N}{k}L_n(k)\biggr)^{-p}}_{\leq \sigma^{-2p}} k^{-\frac{p}{2}}\\ 
        & \lesssim \frac{1}{(\varepsilon k_n^*)^p}\sum_{k=1}^{k_n^*} k^{\frac{p}{2}} + \frac{1}{\varepsilon^{p}}\sum_{k=k_n^*+1}^{K_n} k^{-\frac{p}{2}} \lesssim \frac{(k_{n}^{*})^{1-p/2}}{\varepsilon^{p}}.
    \end{split}
    \end{equation}
    Since any sequence $k_n^* \in \mathrm{argmin}_{1\leq k \leq K_n} L_n(k)$ diverges, by the same arguments as in Lemma 4.1 of \cite{karagrigoriou1995}, this implies $A_3^2 \to 0$ in probability. 
    Setting $\varepsilon = c(k_{n}^{*})^{-\delta}$ yields quantitative estimate of the lemma.

    Now, to establish \eqref{eq:Bn-Bnk-rate}, we may apply Lemma \ref{lemma:Bn-Bnk-rate} and observe that 
    \begin{equation*}
      \|a - a(k)\|_{\ell^{1}}^{2} \lesssim k^{-5},
    \end{equation*}
    by Lemma \ref{lemma:basic-props-AR(oo)}. 

     \end{proof}

     \begin{proof}[\proofof{Lemma \ref{lemma:irs:sk}}]
      The proof follows along the lines of Shibata's original approach (Proposition 4.1 in \cite{shibata}). The main difference is, that Lemma 4.2 of \cite{shibata} is not available to us, as our innovations are not independent. 
Lemma \ref{lemma:QF-1} provides a replacement.  
Note that under our set of assumptions (in particular Assumption \ref{I2}), we can achieve a better rate than Shibata, which simplifies the proof slightly. 
We are going to show that 
\begin{equation*}
  \max_{1\leq k\leq K_{n}}\frac{|s_{k_n^*}^2 - \sigma_{k_n^*}^2 - s_k^2 +\sigma_k^2|}{L_{n}(k)}\leq b(k_{n}^{*})^{-\delta},
\end{equation*}
with probability at least $1-C(k_{n}^{*})^{1+\delta q/2 - q/4}$.

    The expression $s_{k_n^*}^2 - \sigma_{k_n^*}^2 - s_k^2 +\sigma_k^2$ may be rewritten as (see Proposition 4.1 in \cite{shibata})
    \begin{equation}\label{eq:s-sigma-split}
        \begin{split}
            s_{k_n^*}^2 - \sigma_{k_n^*}^2 - s_k^2 +\sigma_k^2 & = 2\langle a(k_n^*) - a(k), \hat{r}(K_n) - r(K_n)\rangle \\
            & + \langle a(k_n^*) - a(k), (\hat{R}(K_n) - R(K_n))(a(k_n^*) + a(k))\rangle,
    \end{split}
    \end{equation}
    where $\hat{r}(k) = (\hat{\gamma}(1),\dots,\hat{\gamma}(k))^T$, and $\hat{\gamma}(j) = N^{-1}\sum_{t=K_{n}+1}^{n}X_{t}X_{t-j}$.
    Here we think of every vector as $K_{n}$ dimensional, padding zeros as necessary.
    We are first going to deal with the second term, since it is more involved an requires more attention.

   Applying Markov's inequality for $p=q/2$, estimating the maximum against the corresponding sum, and setting $\eta(k)  = \eta_{n}(k)= a(k_n^*) - a(k)$ and $\xi(k) = \xi_{n}(k)= a(k_n^*) + a(k)$ in Lemma \ref{lemma:QF-1}, we get
    \begin{multline}\label{eq:markov-s_k-sigma_k}
      \mathbb{P}\left(\max_{1\leq k \leq K_{n}}\frac{|\langle a(k_n^*) - a(k), (\hat{R}(K_n) - R(K_n))(a(k_n^*) + a(k))\rangle|}{L_{n}(k)}> b(k_{n}^{*})^{-\delta} \right) \lesssim\\
            (k_{n}^{*})^{\delta p}\sum_{k=1}^{K_n} \frac{1}{N^{p/2}L_n(k)^p} \|\eta(k)\|_{\ell^2}^p\biggl(\|\xi(k)\|_{\ell^1} + \|\xi(k)\|_{\ell^1}^{\frac{1}{2}}\biggl(\sum_{i=1}^{K_n}(i+1)|a_i(k_n^*) + a_i(k)|\biggr)^{\frac{1}{2}}\biggr)^p.
    \end{multline}

    First, we are going to show that the expression
    \begin{equation*}
        \|\xi(k)\|_{\ell^1} + \|\xi(k)\|_{\ell^1}^{\frac{1}{2}}\biggl(\sum_{i=1}^{K_n}(i+1)|a_i(k_n^*) + a_i(k)|\biggr)^{\frac{1}{2}}
    \end{equation*}
    is uniformly bounded in $k$ and $n$. 
    Note that by Baxter's inequality (Lemma \ref{lemma:baxter}), we have
    \begin{equation*}
        \|\xi(k)\|_{\ell^1} \leq \|a(k_n^*) - a\|_{\ell^1} + \|a(k) - a \|_{\ell^1} + 2\|a\|_{\ell^1} \lesssim  \|a\|_{\ell^1}. 
    \end{equation*}
    In a similar vein, Lemma \ref{lemma:baxter} and Assumption \ref{I2} yield
    \begin{equation*}
    \begin{split}
        \sum_{i=1}^{K_n}(i+1)|a_i(k_n^*) + a_i(k)| &\leq \sum_{k=1}^{K_n}(i+1)|a_i(k_n^*) - a_i| + \sum_{k=1}^{K_n}(i+1)|a_i(k_n^*) - a_i| + 2\sum_{k=1}^{K_n}(i+1)|a_i|.  \\
        & \lesssim \sum_{k=1}^{\infty} i|a_i| < \infty.
    \end{split}
    \end{equation*}
    Additionally, the $\ell^2$ norm and the $\|\cdot\|_R$ norm are equivalent by Assumption \ref{G2}. Hence the right-hand side of Inequality \eqref{eq:markov-s_k-sigma_k} may be further estimated by
    \begin{equation*}
        \lesssim (k_{n}^{*})^{\delta p}\sum_{k=1}^{K_n} \frac{\|\eta(k)\|_{R}^p}{N^{p/2}L_n(k)^p} . 
    \end{equation*}
    We may insert $\pm a$ into $\|\eta(k)\|_{R}$ and use $(x+y)^p \leq 2^{1-1/p}(x^p + y^p)$ for $x,y\geq 0$ and $p\geq 1$, to further estimate this expression by
    \begin{equation}\label{eq:sk-sigmak:sum-split}
        \lesssim (k_{n}^{*})^{\delta p}\sum_{k=1}^{K_n} \frac{\|a(k_n^*) - a\|_{R}^p}{N^{p/2}L_n(k)^p} + (k_{n}^{*})^{\delta p} \sum_{k=1}^{K_n} \frac{\|a(k) - a\|_{R}^p}{N^{p/2}L_n(k)^p} =: I + II. 
    \end{equation}
    We are going to deal with each of these two terms separately. 
    Since $L_n(k) = \|a -a(k)\|_{R}^2 + \sigma^2 k/N$, we have
    \begin{equation}\label{eq:|a-a(k)| vs. L_n(k)}
        \frac{\|a-a(k)\|_R^2}{L_n(k)} = \frac{\|a-a(k)\|_R^2}{\|a -a(k)\|_{R}^2 + \sigma \frac{k}{N}} \leq 1.
    \end{equation}
    
    Starting with the first term, we split the sum along $k=k_n^*$. For $k\leq k_n^*$ we use $L_n(k) \geq L_n(k_n^*)$ and  \eqref{eq:|a-a(k)| vs. L_n(k)}, to arrive at the estimate
    \begin{equation*}
      (k_{n}^{*})^{\delta p}\sum_{k=1}^{k_n^*} \frac{\|a(k_n^*) - a\|_{R}^p}{N^{p/2}L_n(k)^p} \leq \frac{(k_n^*)^{1 +\delta p}}{N^{p/2}L_n(k_{n}^{*})^{p/2}}.
    \end{equation*}
    Since $NL_n(k) = N\|a-a(k)\|_R^2 + \sigma^2 k \geq \sigma^2 k$, this yields
\begin{equation*}
      (k_{n}^{*})^{\delta p}\sum_{k=1}^{k_n^*} \frac{\|a(k_n^*) - a\|_{R}^p}{N^{p/2}L_n(k)^p} \leq \frac{(k_n^*)^{1 +\delta p}}{N^{p/2}L_n(k_{n}^{*})^{p/2}} \lesssim (k_{n}^{*})^{1+\delta p - p/2}.
    \end{equation*}
    The exponent is negative, if $\delta < 1/2 - 1/p$.
    For $k>k_n^*$, we use $L_n(k)^p \geq L_n(k)^{p/2} L_n(k_n^*)^{p/2}$, \eqref{eq:|a-a(k)| vs. L_n(k)}, and $NL_n(k) \geq \sigma^2 k$,  to arrive at the estimate
    \begin{equation*}
      \begin{split}
        (k_{n}^{*})^{\delta p}\sum_{k=k_n^*+1}^{K_n} \frac{\|a(k_n^*) - a\|_{R}^p}{N^{p/2}L_n(k)^p}& \leq (k_{n}^{*})^{\delta p}\sum_{k=k_n^*+1}^{K_n} \frac{1}{N^{p/2}L_n(k)^{p/2}} \lesssim (k_{n}^{*})^{\delta p} \sum_{k=k_{n}^{*}+1}^{K_{n}}k^{-p/2}\\
                                                                                                   &\lesssim (k_{n}^{*})^{1+\delta p -p/2}.
      \end{split}
    \end{equation*}
    In total, this yields $I\lesssim (k_{n}^{*})^{1+\delta p - p/2}$.
    The expression $II$ in Equation \eqref{eq:sk-sigmak:sum-split} can be estimated by the exact same arguments, yielding $II\lesssim (k_{n}^{*})^{1+\delta p - p/2}$.

    Moving on to the first part of \eqref{eq:s-sigma-split}, we set $\eta = a(k_{n}^{*}) - a(k)$, $\xi_{0}=1$, and $\xi_{i}= 0$ for $i=1,\dots,K_{n}$ in Lemma \ref{lemma:QF-1}.
    This yields
    \begin{equation*}
      \begin{split}
        \mathbb{P}\biggl(\max_{1\leq k \leq K_{n}}\frac{2|\langle a(k_n^*) - a(k), \hat{r}(K_n) - r(K_n)\rangle |}{L_{n}(k)}> b(k_{n}^{*})^{-\delta}\biggr) &  \lesssim (k_{n}^{*})^{\delta p} \sum_{k=1}^{K_{n}}\frac{\|a(k_{n}^{*}) - a(k)\|_{R}^{p}}{N^{p/2}L_{n}(k)^{p}} \\
                                                                                                                                                           & \lesssim (k_{n}^{*})^{1+\delta p - p/2},
      \end{split}
    \end{equation*}
    by the same argument used to estimate \eqref{eq:sk-sigmak:sum-split} in the first part. This concludes the proof. 
\end{proof}

     \subsection{Proof of Lemma \ref{lemma:A1}}\label{sec:lemma-A1}

\begin{proof}[\proofof{Lemma \ref{lemma:A1}}]
If we can establish
\begin{equation}\label{eq:A1-rate-goal}
    \bigg\lVert \frac{N}{k}\|R(k)^{-1}B_n\|_R^2 - \sigma^2 \bigg\rVert_p \lesssim k^{-\frac{1}{2}}
\end{equation}    
for $p = q/4 >2$, then we can perform the same estimate as in \eqref{eq:Bn-Bnk-Markov-est},
    \begin{equation*}
      \begin{split}
        \mathbb{P}\biggl(\max_{1 \leq k \leq K_n} \frac{|\|R(k)^{-1}B_{n}\|_R^2 - \frac{k}{N}\sigma^2|}{L_n(k)} > \varepsilon \biggr) &\leq \frac{1}{\varepsilon^p} \sum_{k=1}^{K_n} \frac{\mathbb{E}(|\frac{N}{k}\|\hat{a}(k) - a(k)\|_R^2 - \sigma^2|^p)}{(\frac{N}{k}L_n(k))^p} \cformat{main-lemma-prep-0}{\\
        & \lesssim \frac{1}{(k_n^*)^p}\sum_{k=1}^{k_n^*} k^{\frac{p}{2}} + \sum_{k=k_n^* +1}^{K_n} k^{-\frac{p}{2}}\\
        &}{}\lesssim \frac{(k_n^*)^{1-\frac{p}{2}}}{\varepsilon^{p}}.
      \end{split}
    \end{equation*}

Here we again exploit the fact that any sequence $k_n^* \in \mathrm{argmin}_{1 \leq k \leq K_n} L_n(k)$ diverges (see Lemma 4.1 of \cite{karagrigoriou1995}). 
In particular, setting $c\varepsilon=(k_{n}^{*})^{-\delta}$, yields
\begin{equation*}
  \mathbb{P}(A_{1}\geq (k_{n}^{*})^{-\delta}) \lesssim (k_{n}^{*})^{1-p/2+\delta p}.
\end{equation*}
Thus, it is sufficient to establish \eqref{eq:A1-rate-goal} to finish the proof of Lemma \ref{lemma:A1}. 
The proof of \eqref{eq:A1-rate-goal} poses a major technical difficulty and is split into multiple steps. 
In contrast to the proof of Lemma \ref{lemma:Bn-Bnk}, we can not simply get rid of $R(k)^{-1}$ by using that its eigenvalues are uniformly bounded in $k$. 
The normalization $R(k)^{-1}$ provides is crucial for the proof. 
Hence, a more fine-grained analysis is necessary. 
Recall that the inverse covariance matrix $R(k)^{-1}$ of the stationary process $X$ can be decomposed into
  \begin{equation}\label{eq:R(k)-inverse-representation}
    R(k)^{-1} = L(k)^T D(k) L(k),
  \end{equation}
  where
  \begin{equation*}
    L(k) =
    \begin{pmatrix}
      1 & a_1(k-1) &  \cdots& a_{k-2}(k-1)& a_{k-1}(k-1)\\
      0 & 1 & a_1(k-2) &  \cdots& a_{k-2}(k-2)\\
      \vdots &\vdots &\vdots &\vdots &\vdots \\
      0 & \cdots & 0 & 1 & a_1(1)\\
      0 & \cdots & 0 & 0 & 1\\
    \end{pmatrix}
  \end{equation*}
and $D(k) = \mathrm{diag}(\sigma_{k-1}^{-2}, \dots, \sigma_{0}^{-2})$ (see \cite{kromer}, p. 98, and note that $a(k)$ are the negative Yule-Walker coefficients). 
Since the $j$-th row of $L(k)$ contains the AR-coefficients of $\pino{t-j}{k-j}$, we get
  \begin{equation*}
    L(k)B_n = \frac{1}{N}\sum_{t=K_n+1}^n L(k)\begin{pmatrix}
      X_{t-1}\\
      \vdots\\
      X_{t-k}
    \end{pmatrix}e_{t} =  \frac{1}{N} \sum_{t=K_n+1}^{n}
    \begin{pmatrix}
      \pino{t-1}{k-1}\\
      \vdots\\
      \pino{t-k}{0}
    \end{pmatrix}e_{t}.
  \end{equation*}

  Equipped with this identity and the factorization $R(k)^{-1} = L(k)^TD(k)L(k)$, we may write
  \begin{equation*}
    \begin{split}
      \frac{N}{k}\|R(k)^{-1}B_n\|_R^2 & = \cformat{main-lemma-prep-1}{\frac{N}{k} \langle B_n, R(k)^{-1}B_n\rangle \\
    & =  \frac{N}{k} \langle L(k)B_n, D(k)L(k)B_n\rangle \\
    & =\frac{1}{Nk}\sum_{s,t=K_n+1}^n e_se_t  \biggl\langle \begin{pmatrix}
      \pino{t-1}{k-1}\\
      \vdots\\
      \pino{t-k}{0}
    \end{pmatrix},D(k)\begin{pmatrix}
      \pino{s-1}{k-1}\\
      \vdots\\
      \pino{s-k}{0}
    \end{pmatrix}\biggr\rangle\\
    & =\frac{1}{Nk} \sum_{j=1}^k \sigma_{k-j}^{-2} \sum_{s,t=K_n+1}^n e_te_s\pino{t-j}{k-j}\pino{s-j}{k-j}.}{\frac{1}{Nk} \sum_{j=1}^k \sigma_{k-j}^{-2} \sum_{s,t=K_n+1}^n e_te_s\pino{t-j}{k-j}\pino{s-j}{k-j}.}
    \end{split}
  \end{equation*}

  Following \cite{karagrigoriou2001}, we split this sum into $U_1 + 2 U_2$ with
  \begin{equation*}
    \begin{split}
      U_1 &= \frac{1}{Nk}\sum_{t=K_n+1}^{n}\sum_{j=1}^k \sigma_{k-j}^{-2} e_{t}^2 \pino{t-j}{k-j}^2, \text{ and}\\
      U_2 &= \frac{1}{Nk}\sum_{t=K_n+2}^{n}\sum_{s=K_n+1}^{t-1}\sum_{j=1}^k \sigma_{k-j}^{-2} e_{t} \pino{t-j}{k-j}e_{s}\pino{s-j}{k-j}.
    \end{split}
  \end{equation*}
  The term $U_1$ contains the \textit{mass} (\textit{i.e.} $U_1 \to \sigma^2$ in $L^p$), while $U_2$ should be thought of as an error term, i.e., $U_{2}\to 0$ in $L^{p}$.
  For the remainder of the proof, we write $\II{t}{k}{j} = e_{t}\pino{t-j}{k-j}$ and $\I{t}{s}{k}{j} = \II{t}{k}{j}\II{s}{k}{j} = e_{t}\pino{t-j}{k-j}e_{s}\pino{s-j}{k-j}$.
  The proof is split into the following steps:

  \begin{enumerate}[label=Step \Roman*,ref=(Step \Roman*),align=left]
\item \label{Step I} (Locate the mass). We show that $\|U_1 - \sigma^2\|_p \lesssim k^{-1/2}$.
The rest of this proof is concerned with establishing sufficiently fast decay of $U_2$, \textit{i.e.}, $\|U_2\|_p \lesssim k^{-1/2}$. 
\item\label{Step II} (Center the remainder). Center $U_2$. To this end we show that $|\mathbb{E}(U_2)| \lesssim k^{-1}$. In the remaining steps, we show that $\|U_2 -\mathbb{E}(U_2)\|_p \lesssim k^{-1/2}$.
  Note that this was not necessary in previous proofs, as $\mathbb{E}(U_{2})=0$, if the innovations are independent.
\item \label{Step III} (Estimate nearby terms). We split the remainder $\|U_2 - \mathbb{E}(U_2)\|_p$ into two terms; one in which we collect all terms with $t-s \leq k+1$ and one where $t-s >k+1$.  
To deal with the terms where $t-s \leq k+1$ we show
      \begin{equation*}
        \begin{split}
          &\frac{1}{Nk}\biggl\|\sum_{t=K_n+2}^{K_n+k+2} \sum_{s=K_n+1}^{t-1} \sum_{j=1}^k \sigma_{k-j}^{-2} (\I{t}{s}{k}{j}-\mathbb{E}(\I{t}{s}{k}{j}))\biggr\|_p \lesssim k^{-\frac{1}{2}}, \quad \text{and}\\
          &\frac{1}{Nk}\biggl\|\sum_{t=K_n+k+3}^{n} \sum_{s=t-1-k}^{t-1} \sum_{j=1}^k \sigma_{k-j}^{-2} (\I{t}{s}{k}{j}-\mathbb{E}(\I{t}{s}{k}{j}))\biggr\|_p \lesssim k^{-\frac{1}{2}}.
        \end{split}
      \end{equation*}
      These estimates are established in Lemma \ref{lemma:stepIII} and Lemma \ref{lemma:StepIII-2} respectively.
\item \label{Step IV} (Estimate terms far apart). Since the remainder is centered by \ref{Step II}, we may perform a martingale difference approximation. More precisely, we apply Lemma \ref{lemma:mdsa} to 
        \begin{equation*}
          Z_t(k) = \sum_{j=1}^k \sigma_{k-j}^{-2}\sum_{s=K_n+1}^{t-k-2} \I{t}{s}{k}{j}.
        \end{equation*}
        This yields 
        \begin{equation*}
            \frac{1}{Nk} \bigg\| \sum_{t=K_n+k+3}^n  Z_t(k) - \mathbb{E}(Z_t(k))\biggr\|_p = \frac{1}{Nk} \bigg\| \sum_{l=0}^\infty \sum_{t=K_n+k+3}^n  \mathcal{P}_{t-l}(Z_t(k))\biggr\|_p.
        \end{equation*}
        Using the linearity of $\mathcal{P}_{t-l}$, and the fact that $\II{s}{k}{j} = e_s \pino{s-j}{k-j}$ is $\mathcal{F}_{t-l-1}$-measurable for $l=0,\dots, k$, we may estimate the sum by the following two terms
      \begin{equation*}
        \begin{split}
         &\frac{1}{Nk} \bigg\| \sum_{l=0}^\infty \sum_{t=K_n+k+3}^n  \mathcal{P}_{t-l}(Z_t(k))\biggr\|_p \\
         & \leq \frac{1}{Nk}\sum_{l=0}^k\bigg\| \sum_{t=K_n+k+3}^n \sum_{s=K_n+1}^{t-k-2} \sum_{j=1}^k \sigma_{k-j}^{-2} \mathcal{P}_{t-l}(\II{t}{k}{j})\II{s}{k}{j}\biggr\|_p \\
          & + \frac{1}{Nk} \bigg\|\sum_{l=k+1}^{\infty} \sum_{t=K_n+k+3}^n \mathcal{P}_{t-l}(Z_t(k))\biggr\|_p. \\
        \end{split}
      \end{equation*}
      Since both of these estimates are somewhat involved, we will split them into two lemmas. In Lemma \ref{lemma:stepIV-1} we are going to show
      \begin{equation*}
        \frac{1}{Nk}\sum_{l=0}^k\bigg\| \sum_{t=K_n+k+3}^n \sum_{s=K_n+1}^{t-k-2} \sum_{j=1}^k \sigma_{k-j}^{-2} \mathcal{P}_{t-l}(\II{t}{k}{j})\II{s}{k}{j}\biggr\|_p \lesssim k^{-\frac{1}{2}},
      \end{equation*}
      while the estimate
      \begin{equation*}
        \frac{1}{Nk} \bigg\|\sum_{l=k+1}^{\infty} \sum_{t=K_n+k+3}^n \mathcal{P}_{t-l}(Z_t(k))\biggr\|_p \lesssim k^{-\frac{1}{2}}
      \end{equation*}
      will be established in Lemma \ref{lemma:stepIV-2}, finishing the proof.
\end{enumerate}

\end{proof}

    \begin{lemma}[Step I]
      Under Assumptions \ref{ass:main} and \ref{ass:weak:dep}, we have 
      \begin{equation*}
        \|U_1 - \sigma^2\|_p \leq C_{p}k^{-1/2},
      \end{equation*}
      where $p=q/4$, and $C_{p}$ is a constant depending only on $q$, and the distribution of the process $X$.
    \end{lemma}
    \begin{proof}
      Inserting $\pm e_t$ into $U_1 - \sigma^2$, a simple application of the triangle inequality yields
  \begin{equation*}
    \begin{split}
      \|U_1 - \sigma^2\|_p &= \biggl\|\frac{1}{kN}\sum_{t=K_n+1}^{n} \sum_{j=1}^k\biggl( e_{t}^2 \frac{\pino{t-j}{k-j}^2}{\sigma_{k-j}^2} \pm e_{t}^2 - \sigma^2\biggr)\biggr\|_p\\
      \cformat{main-lemma-stepI-1}{&= \biggl\|\frac{1}{kN}\sum_{t=K_n+1}^{n}  e_{t}^2 \sum_{j=1}^k \biggl(\frac{\pino{t-j}{k-j}^2}{\sigma_{k-j}^2} -1 \biggr) + \frac{1}{N}\sum_{t=K_n+1}^{n}\biggl(e_{t}^2 - \sigma^2\biggr)\biggr\|_p\\}{}
      &\leq \frac{1}{kN}\biggl\|\sum_{t=K_n+1}^{n}  e_{t}^2 \sum_{j=1}^k \biggl(\frac{\pino{t-j}{k-j}^2}{\sigma_{k-j}^2} -1 \biggr)\biggr\|_p + \frac{1}{N}\biggl\|\sum_{t=K_n+1}^{n}\biggl(e_{t}^2 - \sigma^2\biggr)\biggr\|_p.\\
    \end{split}
  \end{equation*}
  The second term is $O(N^{-1/2})$ by Lemma \ref{lemma:PD-standart-est}. 
  Since the process $(e_t)$ is stationary, we may estimate the first term by the triangle and Hölder's inequality, 
  \begin{equation*}
    \begin{split}
      \frac{1}{kN}\biggl\|\sum_{t=K_n+1}^{n}  e_{t}^2 \sum_{j=1}^k \biggl(\frac{\pino{t-j}{k-j}^2}{\sigma_{k-j}^2} -1 \biggr)\biggr\|_p &\leq \frac{1}{k}\|e_0\|_{2p}\biggl\| \sum_{j=1}^k \biggl(\frac{\pino{-j}{k-j}^2}{\sigma_{k-j}^2} -1 \biggr)\biggr\|_{2p}.
    \end{split}
  \end{equation*}
    Note that we are not in a position to apply Lemma \ref{lemma:PD-standart-est} directly to the right-hand side of this estimate, since we do not evaluate a single process $Z = (Z_t)_{t\in\zz}$ at different time points $Z_j$, but rather we evaluate a different process $(\pino{t}{k-j})_{t\in \zz}$ for each $j$. However, we still have $\mathbb{E}(\pino{-j}{k-j}^2) = \sigma_{k-j}^2$, and thus can perform a martingale difference approximation as in Lemma \ref{lemma:mdsa},
    \begin{equation*}
        \begin{split}
            \biggl\| \sum_{j=1}^k \biggl(\frac{\pino{-j}{k-j}^2}{\sigma_{k-j}^2} -1 \biggr)\biggr\|_{2p} & \leq \sum_{l=0}^\infty \biggl\| \sum_{j=1}^k \sigma_{k-j}^{-2}\mathcal{P}_{-j-l}(\pino{-j}{k-j}^2) \biggr\|_{2p}.
        \end{split}
    \end{equation*}
    After reversing the order of summation, $\mathcal{P}_{-k+j-l}(\pino{-k+j}{j})$ is a martingale difference sequence with respect to $\sigf{-k-l+j}$ for $j=0,\dots, k-1$. An application of Burkholder's inequality to the right-hand side of this estimate yields
    \begin{equation}\label{eq:A1-after-burkholder}
        \begin{split}
          \sum_{l=0}^\infty \biggl\| \sum_{j=1}^k \sigma_{k-j}^{-2}\mathcal{P}_{-j-l}(\pino{-j}{k-j}^2) \biggr\|_{2p} \cformat{main-lemma-stepI-3}{& \leq B_{2p} \sum_{l=0}^\infty \sqrt{ \sum_{j=1}^k \sigma_{k-j}^{-4}\|\mathcal{P}_{-j-l}(\pino{-j}{k-j}^2) \|_{2p}^2}\\}{}
            & \leq B_{2p} \sum_{l=0}^\infty \sqrt{ \sum_{j=1}^k \sigma_{k-j}^{-4}\|\mathcal{P}_{0}(\pino{l}{k-j}^2) \|_{2p}^2}.\\
        \end{split}
    \end{equation}
    Recall that $\mathcal{P}_{0}(\pino{l}{k-j}^2) = \mathbb{E}(\pino{l}{k-j}^2 - (\pino{l}{k-j}')^2\mid \mathcal{F}_0)$ (see Equation \eqref{eq:P(Z)=E(Z-Z')}). 
    \cformat{main-lemma-stepI-2}{Since the conditional expectation is an $L^{2p}$-contraction we may estimate
    \begin{equation*}
        \begin{split}
            \|\mathcal{P}_{0}(\pino{l}{k-j}^2)\|_{2p} & \leq \|\pino{l}{k-j}^2 - (\pino{l}{k-j}')^2\|_{2p} \\
            & = \|\pino{l}{k-j}^2 \pm \pino{l}{k-j}\pino{l}{k-j}' - (\pino{l}{k-j}')^2\|_{2p} \\
            & \leq \|\pino{l}{k-j}(\pino{l}{k-j}-\pino{l}{k-j}')\|_{2p} + \|\pino{l}{k-j}'(\pino{l}{k-j}-\pino{l}{k-j}')\|_{2p}.
        \end{split}
    \end{equation*}
    An application of Hölder's inequality yields
    \begin{equation*}
        \begin{split}
            \|\mathcal{P}_{0}(\pino{l}{k-j}^2)\|_{2p} & \leq \|\pino{l}{k-j}\|_{4p}\|(\pino{l}{k-j}-\pino{l}{k-j}')\|_{4p} + \|\pino{l}{k-j}'\|_{4p}\|(\pino{l}{k-j}-\pino{l}{k-j}')\|_{4p}
        \end{split}
    \end{equation*}
    Since $\pino{l}{k-j} \disteq \pino{l}{k-j}'$ and $\|\pino{0}{k-j}\|_{4p}$ is uniformly bounded in $k$ and $j$ by Lemma \ref{lemma:basic-props-AR(oo)}, this may be estimated by
    \begin{equation*}
        \|\mathcal{P}_{0}(\pino{l}{k-j}^2)\|_{2p} \leq 2\|\pino{0}{k-j}\|_{4p} \delta_{4p}^{I_{k-j}}(l) \lesssim\delta_{4p}^{I_{k-j}}(l).
    \end{equation*}}{%
    Using Lemma \ref{lemma:basic-props-AR(oo)} and a similar argument as in \eqref{eq:pd-prod-est}, we get
\begin{equation*}
        \|\mathcal{P}_{0}(\pino{l}{k-j}^2)\|_{2p} \leq 2\|\pino{0}{k-j}\|_{4p} \delta_{4p}^{I_{k-j}}(l) \lesssim\delta_{4p}^{I_{k-j}}(l).
    \end{equation*}}

    Again, using Lemma \ref{lemma:basic-props-AR(oo)}, we estimate the terms $\sigma_{k-j}^{-4}$ in \eqref{eq:A1-after-burkholder} by $\sigma^{-4}$.
    So far, we have proved the following estimate
      \begin{equation*}
        \|U_1 - \sigma^2\|_p \lesssim \frac{1}{\sqrt{k}} \sum_{l= 0}^\infty \sqrt{\frac{1}{k} \sum_{j=0}^{k-1} \delta_{4p}^{I_j}(l)^2} + \frac{1}{\sqrt{N}}.
  \end{equation*}

  The second term is $\lesssim k^{-1}$, since $K_n^2/n \to 0$. Hence, it is sufficient to show that 
    \begin{equation*}
        \sum_{l= 0}^\infty \sqrt{\frac{1}{k} \sum_{j=0}^{k-1} \delta_{4p}^{I_j}(l)^2} \leq C,
  \end{equation*}
  for some constant $C$ independent of $k$, in order to finish the proof. 
  Inserting $\pm\delta_{4p}^I(l)$ into $\bigl(\sum_{0 \leq j \leq k-1} (\delta_{4p}^{I_j}(l)\pm\delta_{4p}^I(l))^2\bigr)^{1/2}$, and using the triangle inequality for the $\ell^2$ norm, we may estimate this expression by 
  \begin{equation*}
    \sum_{l\geq 0} \sqrt{\frac{1}{k} \sum_{j=0}^{k-1} \delta_{4p}^{I_j}(l)^2} \leq \sum_{l\geq 0} \sqrt{\frac{1}{k} \sum_{j=0}^{k-1} (\delta_{4p}^{I_j}(l) - \delta_{4p}^I(l))^2} + D^I_{4p}(0).
  \end{equation*}

  An application of the reverse triangle inequality for the $L^{4p}$ norm yields
  \begin{equation*}
    \begin{split}
      |\delta_{4p}^{I_j}(l) - \delta_{4p}^I(l)| \cformat{main-lemma-stepI-4}{&= |\|\pino{l}{j} - \pino{l}{j}'\|_{4p} - \|e_{l} - e_{l}'\|_{4p}| \\
      &\leq \|(\pino{l}{j} - e_l) - (\pino{l}{j}' - e_l')\|_{4p} \\
      & = \biggl\|\sum_{m= 0}^\infty (a_m(j) - a_m)(X_{l-m} - X_{l-m}')\biggr\|_{4p} \\}{}
      & \leq \sum_{m= 0}^l |a_m(j) - a_m|\delta_{4p}^X(l-m),
    \end{split}
  \end{equation*}
  where we have used the fact that $\delta_{4p}^X(t) = 0$ for $t < 0$. Using this, and defining the vector $v_m=(|a_m(j)-a_m|)_{0\leq j\leq k-1}$, we arrive at the estimate
  \begin{equation*}
      \begin{split}
        \sum_{l= 0}^\infty \sqrt{\frac{1}{k} \sum_{j=0}^{k-1} (\delta_{4p}^{I_j}(l) - \delta_{4p}^I(l))^2}&\leq \frac{1}{\sqrt{k}}\sum_{l=0}^\infty\sqrt{\sum_{j=0}^{k-1} \biggl(\sum_{m= 0}^l |a_m(j) - a_m|\delta_{4p}^X(l-m)\biggr)^2}\\
        & = \frac{1}{\sqrt{k}} \sum_{l=0}^\infty \biggl\|\sum_{m=0}^l \delta_{4p}^X(l-m) v_m\biggr\|_{\ell^2}.
      \end{split}
  \end{equation*}
  After an application of the triangle inequality for the $\ell^2$ norm, we spot the Cauchy product of two series
  \begin{equation*}
      \begin{split}
        \frac{1}{\sqrt{k}} \sum_{l=0}^\infty \biggl\|\sum_{m=0}^l \delta_{4p}^X(l-m) v_m\biggr\|_{\ell^2} & \leq \frac{1}{\sqrt{k}} \sum_{l=0}^\infty \sum_{m=0}^l \delta_{4p}^X(l-m) \|v_m\|_{\ell^2} \\
        &= \frac{1}{\sqrt{k}} \sum_{l=0}^\infty \delta_{4p}^X(l) \sum_{m=0}^\infty \|v_m\|_{\ell^2}.
      \end{split}
  \end{equation*}
    Using $\|x\|_{\ell^2} \leq \|x\|_{\ell^1}$, we get 
    \begin{equation*}
        \begin{split}
            \sum_{m=0}^\infty \|v_m\|_{\ell^2} & \leq \sum_{m=0}^\infty \|v_m\|_{\ell^1} = \sum_{j=0}^{k-1} \sum_{m=0}^\infty  |a_m(j) - a_m| = \sum_{j=0}^{k-1} \|a(j)-a\|_{\ell^1}. 
        \end{split}
    \end{equation*}
  Now by Lemma \ref{lemma:basic-props-AR(oo)}, $\|a(j) - a\|_{\ell^1} \lesssim j^{-5/2}$. Thus $\sum_{m\geq 0} \|v_m\|_{\ell^2}$ is convergent, and 
  \begin{equation*}
        \sum_{l= 0}^\infty \sqrt{\frac{1}{k} \sum_{j=0}^{k-1} \delta_{4p}^{I_j}(l)^2} \leq C
  \end{equation*}
  for some constant $C$ independent of $k$, finishing the proof of the first step. 
  \end{proof}

  \begin{lemma}[Step II]\label{lemma:stepII}
    Under Assumptions \ref{ass:main} and \ref{ass:weak:dep}, we have $|\mathbb{E}(U_2)| \leq C_{q}k^{-1/2}$, for some constant $C_{q}$ depending only on $q$, and the distribution of the process $X$.
  \end{lemma}
  \begin{proof}
    
  Using the stationarity of the processes $e_t$ and $\pino{t}{k}$ we get 
 \begin{equation}\label{eq:stepII:est-begin}
    \begin{split}
      |\mathbb{E}(U_2)| &= \frac{1}{k} \biggl|\sum_{h=1}^{N-1} \frac{N-h}{N} \sum_{j=1}^k \frac{1}{\sigma_{k-j}^2}\mathbb{E}(\I{h}{0}{k}{j})\biggr|\\
                        & \leq \frac{1}{k} \sum_{h=1}^{N-1}  \sum_{j=1}^k\sigma_{k-j}^{-2}|\mathbb{E}(\I{h}{0}{k}{j})| = II_{1} + II_{2} + II_{3},
    \end{split}
  \end{equation}
  where
  \begin{equation*}
    \begin{split}
      II_{1} & = \frac{1}{k}\sum_{j=2}^k \sigma_{k-j}^{-2}\sum_{h=1}^{j-1}  |\mathbb{E}(\I{h}{0}{k}{j})|, \\
    II_{2} & = \frac{1}{k} \sum_{j=1}^k\sigma_{k-j}^{-2}\sum_{h=j+1}^{N-1}  |\mathbb{E}(\I{h}{0}{k}{j})|, \quad \text{and} \\
    II_{3} & = \frac{1}{k} \sum_{j=1}^k  \sigma_{k-j}^{-2} |\mathbb{E}(\I{h}{0}{k}{j})|.
    \end{split}
  \end{equation*}
  Our goal is to show that $II_1, II_2$ and $II_3$, and hence $|\mathbb{E}(U_2)|$, are of order $O(k^{-1})$.
  The strategy will be the same for each of the three terms $II_1$, $II_2$, and $II_3$. 
    We may sort the time points in each product $e_he_0\pino{h-j}{k-j}\pino{-j}{k-j}$, $e_h\pino{h-j}{k-j}e_0\pino{-j}{k-j}$ and $e_je_0\pino{0}{k-j}\pino{-j}{k-j}$ in decreasing order.
  Ultimately, we wish to translate the difference between adjacent time points into a controllable measure of the dependence between adjacent random variables.  
  This will be achieved using Lemma \ref{lemma:Q}.
The term $II_1$ is structurally the most involved and therefore will be dealt with in full detail.
The terms $II_{2}$ and $II_{3}$ can be dealt with by similar, shorter arguments.
  For $II_1$, we employ the following general strategy:
  \begin{enumerate}[label=($\text{II}_1$-\alph*): ,ref=($\text{II}_1$-\alph*),align=left]
      \item \label{II1-a} First we are going to show that
      \begin{equation*}
      \begin{split}
          \mathbb{E}(e_he_0\pino{h-j}{k-j}\pino{-j}{k-j}) & =\mathbb{E}(Q^h_{h-j+1}(e_he_0)Q_{-j+1}^{h-j}(\pino{h-j}{k-j})\pino{-j}{k-j})\\
      & + \mathbb{E}(Q_1^h(e_h)Q_{h-j+1}^0(e_0)H_{-j+1}^{h-j}(\pino{h-j}{k-j})\pino{-j}{k-j}) \\
      & + \mathbb{E}(Q_{h-j+1}^h(e_h)H_{h-j+1}^0(e_0)H_{-j+1}^{h-j}(\pino{h-j}{k-j})\pino{-j}{k-j}),
      \end{split}
  \end{equation*}
  where $H_{s}^{t}(Z) = \mathbb{E}(Z\mid \sigfdouble{t}{s})$ and $Q_{s}^{t}(Z) = Z - H_{s}^{t}(Z)$.
  Applying Hölder's inequality to each of these summands individually gives 
  \begin{equation*}
  \begin{split}
      |\mathbb{E}(e_he_0\pino{h-j}{k-j}\pino{-j}{k-j})| & \lesssim \|Q^h_{h-j+1}(e_he_0)\|_2 \|Q_{-j+1}^{h-j}(\pino{h-j}{k-j})\|_4 \\
      &+ \|Q_1^h(e_h)\|_4\|Q_{h-j+1}^0(e_0)\|_4 \\
      &+ \|Q_{h-j+1}^h(e_h)\|_4.
  \end{split}
  \end{equation*}
  To see that the constants in the last estimate depend on neither $j$ nor $k$, one can use the fact that $\|\pino{0}{k-j}\|_4$ is uniformly bounded in $k$ and $j$ by Lemma \ref{lemma:basic-props-AR(oo)}, and $\|H^{t}_s(\pino{t}{k-j})\|_4 \leq \|\pino{0}{k-j}\|_4$ (and similarly $\|H^t_s(e_t)\|_4 \leq \|e_t\|_4$).
  \item \label{II1-b} An application of Lemma \ref{lemma:Q} and the subsequent remark yields the following rates
  \begin{equation}\label{eq:stepII:Q-rates}
      \begin{split}
          \|Q^h_{h-j+1}(e_he_0)\|_2 &\lesssim j^{-\frac{5}{2}} + (j-h)^{-\frac{5}{2}},\\
          \|Q_{-j+1}^{h-j}(\pino{h-j}{k-j})\|_4 &\lesssim h^{-\frac{5}{2}}, \\
          \|Q_1^h(e_h)\|_4 &\lesssim h^{-\frac{5}{2}},\\
          \|Q_{h-j+1}^0(e_0)\|_4 &\lesssim (j-h)^{-\frac{5}{2}},\\
          \|Q_{h-j+1}^h(e_h)\|_4 &\lesssim j^{-\frac{5}{2}}.
      \end{split}
  \end{equation}
  \item \label{II1-c} With these rates at hand, we can show that $II_1 \lesssim k^{-1}$. Since $\sigma_{k-j}^2 \geq \sigma^2$ by Lemma \ref{lemma:basic-props-AR(oo)}, $II_1$ may be estimated by $II_1 \lesssim II_1' + II_1'' + II_1'''$ where
  \begin{equation*}
      \begin{split}
          II_1' & = \frac{1}{k}\sum_{j=2}^k\sum_{h=1}^{j-1} (j^{-\frac{5}{2}} + (j-h)^{-\frac{5}{2}})h^{-\frac{5}{2}},\\
          II_1'' &=  \frac{1}{k}\sum_{j=2}^k\sum_{h=1}^{j-1}h^{-\frac{5}{2}}(j-h)^{-\frac{5}{2}}, \quad \text{and}\\
          II_1''' &= \frac{1}{k}\sum_{j=2}^k(j-1)j^{-\frac{5}{2}},
      \end{split}
  \end{equation*}
  all of which are of order $O(k^{-1}).$
  
  \end{enumerate}
  Let us start with \ref{II1-a}.
  Since $\pino{h-j}{k-j}\pino{-j}{k-j}$ is independent of $\mathcal{F}^h_{h-j+1}$ and $\mathbb{E}(e_0e_h) = 0$ for $h\geq 1$, we have $\mathbb{E}(H^h_{h-j+1}(e_0e_h)\pino{h-j}{k-j}\pino{-j}{k-j}) = \mathbb{E}(e_0e_h)\mathbb{E}(\pino{h-j}{k-j}\pino{-j}{k-j}) = 0$. Thus, we may write 
  \begin{equation*}
      \mathbb{E}(e_he_0\pino{h-j}{k-j}\pino{-j}{k-j}) = \mathbb{E}(Q^h_{h-j+1}(e_he_0)\pino{h-j}{k-j}\pino{-j}{k-j}).
  \end{equation*}
  Since $\pino{h-j}{k-j} = Q^{h-j}_{-j+1}(\pino{h-j}{k-j}) + H^{h-j}_{-j+1}(\pino{h-j}{k-j})$, we can rewrite the last display as
  \begin{equation*}
      \begin{split}
          \mathbb{E}(e_he_0\pino{h-j}{k-j}\pino{-j}{k-j}) &= \mathbb{E}(Q^h_{h-j+1}(e_he_0)\pino{h-j}{k-j}\pino{-j}{k-j}) \\
          &= \mathbb{E}(Q^h_{h-j+1}(e_he_0)Q^{h-j}_{-j+1}(\pino{h-j}{k-j})\pino{-j}{k-j}) \\
          &+\mathbb{E}(Q^h_{h-j+1}(e_he_0)H^{h-j}_{-j+1}(\pino{h-j}{k-j})\pino{-j}{k-j})
      \end{split}
  \end{equation*}
  Note that $e_0e_h$ is independent of $H^{h-j}_{-j+1}(\pino{h-j}{k-j})$, as $h<j$. 
  Hence we get
  \begin{equation*}
      \begin{split}
          \mathbb{E}(Q^h_{h-j+1}(e_he_0)H^{h-j}_{-j+1}(\pino{h-j}{k-j})\pino{-j}{k-j})& = \mathbb{E}(e_he_0H^{h-j}_{-j+1}(\pino{h-j}{k-j})\pino{-j}{k-j}) \\
          & + \mathbb{E}(H^h_{h-j+1}(e_he_0)H^{h-j}_{-j+1}(\pino{h-j}{k-j})\pino{-j}{k-j}) \\
          & = \mathbb{E}(e_he_0H^{h-j}_{-j+1}(\pino{h-j}{k-j})\pino{-j}{k-j}),
      \end{split}
  \end{equation*}
  where we have used the fact that $H^h_{h-j+1}(e_he_0)H^{h-j}_{-j+1}(\pino{h-j}{k-j})$ is independent of $\pino{-j}{k-j}$, and $\mathbb{E}(\pino{-j}{k-j}) = 0$. 
  So far we have arrived at the identity 
  \begin{equation*}
      \begin{split}
          \mathbb{E}(e_he_0\pino{h-j}{k-j}\pino{-j}{k-j}) & = \mathbb{E}(Q^h_{h-j+1}(e_he_0)Q^{h-j}_{-j+1}(\pino{h-j}{k-j})\pino{-j}{k-j}) \\
          &+\mathbb{E}(e_he_0H^{h-j}_{-j+1}(\pino{h-j}{k-j})\pino{-j}{k-j}).
      \end{split}
  \end{equation*}
  We now repeat this process for the second summand. 
  Since $H^h_1(e_h)$ is independent of $\mathcal{F}_{0}$ and $\mathbb{E}(e_h) = 0$, we get 
  \begin{equation*}
      \begin{split}
          \mathbb{E}(e_he_0H^{h-j}_{-j+1}(\pino{h-j}{k-j})\pino{-j}{k-j}) & = \mathbb{E}(Q^h_1(e_h)e_0H^{h-j}_{-j+1}(\pino{h-j}{k-j})\pino{-j}{k-j}).
      \end{split}
  \end{equation*}
  Writing $e_0 = Q^0_{h-j+1}(e_0) +H^0_{h-j+1}(e_0)$, and using the fact that $\mathbb{E}(e_h)=0$, and $H^h_1(e_h)$ is independent of $H^0_{h-j+1}(e_0)H^{h-j}_{-j+1}(\pino{h-j}{k-j})\pino{-j}{k-j}$, the right-hand side of the last display is equal to
  \begin{equation*}
      \begin{split}
          \mathbb{E}(Q^h_1(e_h)e_0H^{h-j}_{-j+1}(\pino{h-j}{k-j})\pino{-j}{k-j}) & = \mathbb{E}(Q^h_1(e_h)Q^0_{h-j+1}(e_0)H^{h-j}_{-j+1}(\pino{h-j}{k-j})\pino{-j}{k-j})\\
          & +\mathbb{E}(e_h H^0_{h-j+1}(e_0)H^{h-j}_{-j+1}(\pino{h-j}{k-j})\pino{-j}{k-j}).
      \end{split}
  \end{equation*}
  The second summand may be rewritten as 
  \begin{equation*}
      \begin{split}
          \mathbb{E}(e_h H^0_{h-j+1}(e_0)H^{h-j}_{-j+1}(\pino{h-j}{k-j})\pino{-j}{k-j}) = \mathbb{E}(Q^h_{h-j+1}(e_h) H^0_{h-j+1}(e_0)H^{h-j}_{-j+1}(\pino{h-j}{k-j})\pino{-j}{k-j}),
      \end{split}
  \end{equation*}
  since $H^h_{h-j+1}(e_h) H^0_{h-j+1}(e_0)H^{h-j}_{-j+1}(\pino{h-j}{k-j})$ is independent of $\pino{-j}{k-j}$, and $\mathbb{E}(\pino{-j}{k-j}) = 0$. In total, we get 
  \begin{equation*}
      \begin{split}
          \mathbb{E}(e_he_0\pino{h-j}{k-j}\pino{-j}{k-j}) & =\mathbb{E}(Q^h_{h-j+1}(e_he_0)Q_{-j+1}^{h-j}(\pino{h-j}{k-j})\pino{-j}{k-j})\\
      & + \mathbb{E}(Q_1^h(e_h)Q_{h-j+1}^0(e_0)H_{-j+1}^{h-j}(\pino{h-j}{k-j})\pino{-j}{k-j}) \\
      & + \mathbb{E}(Q_{h-j+1}^h(e_h)H_{h-j+1}^0(e_0)H_{-j+1}^{h-j}(\pino{h-j}{k-j})\pino{-j}{k-j}).
      \end{split}
  \end{equation*}
  This proves \ref{II1-a} and thus implies $II_1 \lesssim k^{-1}$. 
  \cformat{main-lemma-stepII-1}{Now we move on to $II_2$ and $II_3$. 
We will repeat the steps \ref{II1-a} through \ref{II1-c} for each of these two terms. 
However, $II_2$ and $II_3$ are structurally easier that $II_1$. 
For $II_2$, we will work through the following steps
\begin{enumerate}[label=($\text{II}_2$-\alph*):,ref=($\text{II}_2$-\alph*),align=left]
\item \label{II2-a}We show that 
\begin{equation*}
    \begin{split}
      \mathbb{E}(e_h\pino{h-j}{k-j}e_0\pino{-j}{k-j}) & = \mathbb{E}(Q^h_{h-j+1}(e_h)Q^{h-j}_1(\pino{h-j}{k-j})e_0\pino{-j}{k-j}) \\
      & + \mathbb{E}(Q^h_{1}(e_h)H^{h-j}_1(\pino{h-j}{k-j})e_0\pino{-j}{k-j}).
    \end{split}
  \end{equation*}
  Again, after applying Hölder's inequality to each of these terms individually and using the facts that $\|\pino{0}{k-j}\|_4$ is uniformly bounded in $k$ and $j$ by Lemma \ref{lemma:basic-props-AR(oo)} and $\|H^{h-j}_1(\pino{h-j}{k-j})\|_4 \leq \|\pino{h-j}{k-j}\|_4$, we have
  \begin{equation*}
      |\mathbb{E}(e_h\pino{h-j}{k-j}e_0\pino{-j}{k-j})| \lesssim \|Q^h_{h-j+1}(e_h)\|_4\|Q^{h-j}_1(\pino{h-j}{k-j})\|_4 + \|Q^h_{1}(e_h)\|_4.
  \end{equation*}
  \item By Lemma \ref{lemma:Q} and the subsequent remark, we get the following rates
  \begin{equation}\label{eq:stepII:Q-rates-2}
      \begin{split}
          \|Q^h_{h-j+1}(e_h)\|_4 &\lesssim j^{-\frac{5}{2}},\\
          \|Q^{h-j}_1(\pino{h-j}{k-j})\|_4 & \lesssim (h-j)^{-\frac{5}{2}}\quad \text{and}\\
          \|Q^h_{1}(e_h)\|_4 & \lesssim h^{-\frac{5}{2}}.
      \end{split}
  \end{equation}
  \item \label{II2-c}In light of \eqref{eq:stepII:Q-rates-2} and since $\sigma_{k-j}^{-2} \leq \sigma^{-2}$ by Lemma \ref{lemma:basic-props-AR(oo)}, we may estimate $II_2$ by $II_2 \lesssim II_2' + II_2''$, where
  \begin{equation*}
      \begin{split}
          II_2' &= \frac{1}{k}\sum_{j=1}^k \sum_{h=j+1}^{N-1} j^{-\frac{5}{2}}(h-j)^{-\frac{5}{2}}\quad \text{and}\\
          II_2'' & = \frac{1}{k}\sum_{j=1}^k \sum_{h=j+1}^{N-1} h^{-\frac{5}{2}},
      \end{split}
  \end{equation*}
  both of which are of order $O(k^{-1})$.
\end{enumerate}
For \ref{II2-a}, we start by writing $\pino{h-j}{k-j} = Q^{h-j}_1(\pino{h-j}{k-j}) + H^{h-j}_1(\pino{h-j}{k-j})$, which implies 
\begin{equation*}
    \begin{split}
        \mathbb{E}(e_h\pino{h-j}{k-j}e_0\pino{-j}{k-j}) & = \mathbb{E}(e_h Q^{h-j}_1(\pino{h-j}{k-j})e_0\pino{-j}{k-j}) \\
        & + \mathbb{E}(e_h H^{h-j}_1(\pino{h-j}{k-j})e_0\pino{-j}{k-j}).
    \end{split}
\end{equation*}
Since $H^{h}_{h-j+1}(e_h)$ is independent of $Q^{h-j}_1(\pino{h-j}{k-j})e_0\pino{-j}{k-j}$ and $\mathbb{E}(e_h)=0$, we may replace the first term with $\mathbb{E}(Q^h_{h-j+1}(e_h)Q^{h-j}_1(\pino{h-j}{k-j})e_0\pino{-j}{k-j})$. The second term may be replaced by $\mathbb{E}(Q^h_1(e_h) H^{h-j}_1(\pino{h-j}{k-j})e_0\pino{-j}{k-j})$ as $H^h_1(e_h)H^{h-j}_1(\pino{h-j}{k-j})$ is independent of $e_0\pino{-j}{k-j}$ and $\mathbb{E}(e_0\pino{-j}{k-j}) =0$. This proves the claim in \ref{II2-a}.

Next, we deal with $II_3$. We again employ the same strategy. Since $\mathbb{E}(e_j) = 0$ and $H^{j}_1(e_j)$ is independent of $e_0\pino{0}{k-j}\pino{-j}{k-j}$, we may write 
  \begin{equation*}
    |\mathbb{E}(e_je_0\pino{0}{k-j}\pino{-j}{k-j})| = |\mathbb{E}(Q^j_1(e_j)e_0\pino{0}{k-j}\pino{-j}{k-j})| \lesssim \|Q^j_1(e_j)\|_4,
  \end{equation*}
where we have used Hölder's inequality and the fact that $\|\pino{0}{k-j}\|_4$ is uniformly bounded in $k$ and $j$ (see Lemma \ref{lemma:basic-props-AR(oo)}). By Lemma \ref{lemma:Q} we have $\|Q^j_1(e_j)\|_4 \lesssim j^{-5/2}$. Since $\sigma_{k-j}^{-2} \leq \sigma^{-2}$, we get 
\begin{equation*}
    II_3 \lesssim \frac{1}{k} \sum_{j=1}^k j^{-\frac{5}{2}} \lesssim k^{-1},
\end{equation*}
which completes the proof.}{The terms $II_{2}$ and $II_{3}$ are shown to be of order $k^{-1}$ by similar, albeit much shorter arguments.}

\end{proof}

  \begin{lemma}[Step III, Part 1]\label{lemma:stepIII}
    Under Assumptions \ref{ass:main} and \ref{ass:weak:dep}, we have
    \begin{equation}
      \label{eq:step3-1}
      \frac{1}{Nk}\biggl\|\sum_{t=K_n+2}^{K_n+k+2} \sum_{s=K_n+1}^{t-1} \sum_{j=1}^k \sigma_{k-j}^{-2} (\I{t}{s}{k}{j}-\mathbb{E}(\I{t}{s}{k}{j}))\biggr\|_p \leq \frac{C_{p}}{\sqrt{k}},
    \end{equation}
    where $p=q/4$, and $C_p$ are constants depending only on $p$, and the process $X$. 
  \end{lemma}

  \begin{proof}
  Starting with \eqref{eq:step3-1}, we note that a \textit{very rough} estimate is enough here. 
  First, by Lemma \ref{lemma:stepII}, we may ignore that the summands on the left-hand side of \eqref{eq:step3-1} are centered (note that having centered summands will be crucial for \eqref{eq:step3-2} below). 
  Applying the triangle inequality twice to the non-centered version of \eqref{eq:step3-1}, and using the fact that $\sigma_{k-j}^{-2} \leq \sigma^{-2}$ (see Lemma \ref{lemma:basic-props-AR(oo)}) yields the estimate
  \begin{equation*}
      \begin{split}
         \frac{1}{\sigma^{2}Nk}\sum_{j=1}^k \sum_{t=K_n+2}^{K_n+k+2} \biggl\|\sum_{s=K_n+1}^{t-1} \I{t}{s}{k}{j}\biggr\|_p.
      \end{split}
  \end{equation*}
  Note that $\|e_t\pino{t-j}{k-j}\|_{2p} \leq \|e_t\|_{4p}\|\pino{t-j}{k-j}\|_{4p}$ is unformly bounded in $t,j$ and $k$ by the stationarity of the process $(X_t)_{t\in\zz}$ and Lemma \ref{lemma:basic-props-AR(oo)}. Hence, an application of Hölder's inequality yields
 \begin{equation*}
   \begin{split}
   \frac{1}{\sigma^{2}Nk}\sum_{j=1}^k \sum_{t=K_n+2}^{K_n+k+2} \biggl\|\sum_{s=K_n+1}^{t-1} \I{t}{s}{k}{j}\biggr\|_p
          &\lesssim \frac{1}{Nk}\sum_{j=1}^k \sum_{t=K_n+2}^{K_n+k+2} \|\II{t}{k}{j}\|_{2p}\biggl\|\sum_{s=K_n+1}^{t-1} \II{s}{k}{j}\biggr\|_{2p} \\
          &\lesssim\frac{1}{Nk}\sum_{j=1}^k \sum_{t=K_n+2}^{K_n+k+2} \biggl\|\sum_{s=K_n+1}^{t-1} \II{s}{k}{j}\biggr\|_{2p}. 
   \end{split}
 \end{equation*} 
  Since $\mathbb{E}(e_s\pino{s-j}{k-j}) = 0$, Lemma \ref{lemma:PD-standart-est} yields 
  \begin{equation}\label{eq:stepIII:part1:example-for-part2}
      \begin{split}
        \biggl\|\sum_{s=K_n+1}^{t-1} \II{s}{k}{j}\biggr\|_{2p} & \leq 2B_{2p} \sqrt{t-K_n-2}(\|e_{0}\|_{q}D_{q}^{I_{k-j}}(0) + \|\pino{0}{k-j}\|_{q}D_{q}^{I}(0)). 
      \end{split}
  \end{equation}
  By Assumption \ref{I2}, $D_{4p}^I(0) < \infty$, by Lemma \ref{lemma:dep-rates} $D_{4p}^{I_{k-j}}(0)$ is bounded uniformly in $k$ and $j$, while Lemma \ref{lemma:basic-props-AR(oo)} implies that $\|\pino{0}{k-j}\|_{q}$ is uniformly bounded in $k$ and $j$. In total, we have
  \begin{equation*}
      \begin{split}
          \frac{1}{Nk}\sum_{j=1}^k \sum_{t=K_n+2}^{K_n+k+2} \biggl\|\sum_{s=K_n+1}^{t-1} \II{s}{k}{j}\biggr\|_{2p} &  \frac{1}{Nk}\lesssim \sum_{j=1}^k \sum_{t=K_n+2}^{K_n+k+2} \sqrt{t-K_n-2} \\
          &\lesssim \frac{1}{Nk}\sum_{j=1}^k \sum_{t=0}^{k} \sqrt{t} \lesssim \frac{k^{\frac{3}{2}}}{N},
      \end{split}
  \end{equation*}
  where we have used $\sum_{0\leq t \leq k} \sqrt{t} \lesssim k^{3/2}$. 
Since $K_n^2/n \to 0$ by Assumption \ref{G3} and $k \leq K_n$, we have $k^{3/2}/N \leq k^{-1/2} K_n^2/N \lesssim k^{-1/2}$, proving \eqref{eq:step3-1}.
\end{proof}

\begin{lemma}[Step III, Part 2]\label{lemma:StepIII-2}
  Given Assumptions \ref{ass:main} and \ref{ass:weak:dep}, we have 
 \begin{equation}
      \label{eq:step3-2}
      \frac{1}{Nk}\biggl\|\sum_{t=K_n+k+3}^{n} \sum_{s=t-1-k}^{t-1} \sum_{j=1}^k \sigma_{k-j}^{-2} (\I{t}{s}{k}{j}-\mathbb{E}(\I{t}{s}{k}{j}))\biggr\|_p \leq \frac{C_{p}}{\sqrt{k}},
  \end{equation} 
  where $C_{p}$ is a constant depending only on $p=q/4$ and the distribution of the process $X$.
\end{lemma}

\begin{proof}
  The estimate \eqref{eq:step3-2} requires a little bit more care. 
  First, we apply the triangle inequality to the sum over $j$, and use $\sigma_{k-j}^{-2} \leq \sigma^{-2}$ (Lemma \ref{lemma:basic-props-AR(oo)}). This gives the estimate
  \begin{equation*}
      \frac{1}{\sigma^2Nk} \sum_{j=1}^k \biggl\|\sum_{t=K_n+k+3}^{n} \sum_{s=t-1-k}^{t-1}  \I{t}{s}{k}{j}-\mathbb{E}(\I{t}{s}{k}{j})\biggr\|_p,
  \end{equation*}
  for the left-hand side in \eqref{eq:step3-2}. Since the summands are centered, we can use a martingale difference approximation in the sense of Lemma \ref{lemma:mdsa}, yielding
  \begin{equation}\label{eq:stepIII:quadraticform-ref}
    \biggl\|\sum_{t=K_n+k+3}^{n} \sum_{s=t-1-k}^{t-1} \I{t}{s}{k}{j} -\mathbb{E}(\I{t}{s}{k}{j})\biggr\|_p \cformat{main-lemma-stepIII2-1}{ \\
    = \biggl\|\sum_{l\in\zz}\sum_{t=K_n+k+3}^{n} \sum_{s=t-1-k}^{t-1}  \mathcal{P}_l(\I{t}{s}{k}{j})\biggr\|_p \\}{}
      = \biggl\|\sum_{l\in\zz}\mathcal{P}_l\biggl(\sum_{t=K_n+k+3}^{n} \sum_{s=t-1-k}^{t-1}  \I{t}{s}{k}{j}\biggr)\biggr\|_p.
  \end{equation}
  Now since $p= q/4 > 2$, and $\mathcal{P}_l(Z)$ is a martingale difference sequence for arbitrary $Z \in L^1$, we may apply Burkholder's inequality (Lemma \ref{lemma:burkholder}), which leads to the bound
  \begin{equation}\label{eq:stepIII:stepIII-2-Burk}
    \biggl\|\sum_{l\in\zz}\mathcal{P}_l\biggl(\sum_{t=K_n+k+3}^{n} \sum_{s=t-1-k}^{t-1}  \I{t}{s}{k}{j}\biggr)\biggr\|_p
      \lesssim \sqrt{ \sum_{l\in\zz}\biggl\|\mathcal{P}_l\biggl(\sum_{t=K_n+k+3}^{n} \sum_{s=t-1-k}^{t-1}  \I{t}{s}{k}{j}\biggr)\biggr\|_p^2}.
  \end{equation}
  By \eqref{eq:P(Z)=E(Z-Z')}, and the subsequent discussion, we have
  \begin{equation*}
      \biggl\|\mathcal{P}_l\biggl(\sum_{t=K_n+k+3}^{n} \sum_{s=t-1-k}^{t-1}  \I{t}{s}{k}{j}\biggr)\biggr\|_p
      \leq \biggl\|\sum_{t=K_n+k+3}^{n} \sum_{s=t-1-k}^{t-1}  \I{t}{s}{k}{j} - (\I{t}{s}{k}{j})^{(l)}\biggr\|_p.
  \end{equation*}
  Inserting $\pm (\II{t}{k}{j})^{(l)}\II{s}{k}{j} = \pm e_l^{(l)}\pino{t-j}{k-j}^{(l)} e_s\pino{s-j}{k-j}$, we may estimate this term further by $III_1 + III_2$, where 
  \begin{equation*}
      \begin{split}
          III_1 &= \biggl\|\sum_{t=K_n+k+3}^{n} \sum_{s=t-1-k}^{t-1}  (\II{t}{k}{j}  - (\II{t}{k}{j})^{(l)}) \II{s}{k}{j}\biggr\|_p, \\
          III_2 & = \biggl\|\sum_{t=K_n+k+3}^{n} \sum_{s=t-1-k}^{t-1}  (\II{t}{k}{j})^{(l)}(\II{s}{k}{j}  - (\II{s}{k}{j})^{(l)}) \biggr\|_p.
          \end{split}
  \end{equation*}
  Note that $III_1$ and $III_2$ depend on $l$, $k$, and $j$. 
  We will deal with $III_1$ and $III_2$ individually. 
  Starting with $III_1$, an application of Hölder's inequality yields
  \begin{equation*}
      III_1 \leq \sum_{t=K_n+k+3}^{n} \|\II{t}{k}{j}  - (\II{t}{k}{j})^{(l)}\|_{2p}\biggl\|\sum_{s=t-1-k}^{t-1}   \II{s}{k}{j}\biggr\|_{2p}.
  \end{equation*}
  Since $\mathbb{E}(\II{s}{k}{j}) = 0$, Lemma \ref{lemma:PD-standart-est} yields  \begin{equation*}
    \biggl\|\sum_{s=t-1-k}^{t-1}   \II{s}{k}{j}\biggr\|_{2p} \lesssim \sqrt{k}.
  \end{equation*}
  Using an argument similar to Estimate \eqref{eq:pd-prod-est}, we get
  \begin{equation}\label{eq:stepIII:e_te_(t-j,k-j) - e_t'e_(t-j,k-j)'}
      \begin{split}
          \|\II{t}{k}{j}  - (\II{t}{k}{j})^{(l)}\|_{2p} & \lesssim \delta_{4p}^{I_{k-j}}(t-j-l) + \delta_{4p}^I(t-l),
      \end{split}
  \end{equation}
  since $\|\pino{0}{k}\|_{4p}$ is uniformly bounded in $k$ by Lemma \ref{lemma:basic-props-AR(oo)}.
  In total, we get 
\begin{equation}\label{eq:stepIII:III_1-rate}
  III_1 \lesssim \sqrt{k}\sum_{t=K_n+k+3}^n \biggl(\delta_{4p}^{I_{k-j}}(t-j-l) + \delta_{4p}^I(t-l)\biggr).
\end{equation}
For $III_2$, we may change the order of summation, and apply Hölder's inequality
\begin{equation*}
    \begin{split}
        III_2 & = \biggl\|\sum_{s=K_n+k+2}^{n} \sum_{t=(s+1)\vee(K_n+k+3)}^{(s+k+1)\wedge n}  (\II{t}{k}{j})^{(l)}(\II{s}{k}{j}  - (\II{s}{k}{j})^{(l)}) \biggr\|_p \\
        & \leq \sum_{s=K_n+k+2}^{n} \|\II{s}{k}{j}  - (\II{s}{k}{j})^{(l)}\|_{2p} \biggl\|\sum_{t=(s+1)\vee(K_n+k+3)}^{(s+k+1)\wedge n} (\II{t}{k}{j})^{(l)}\biggr\|_{2p},
    \end{split}
\end{equation*}
where $x\wedge y = \min\{x,y\}$ and $x\vee y = \max\{x,y\}$. 
Note that the sum over $t$ contains at most $k$ terms. Since $e_t^{(l)}\pino{t-j}{k-j}^{(l)} \disteq e_t\pino{t-j}{k-j}$, we have $\mathbb{E}(e_t^{(l)}\pino{t-j}{k-j}^{(l)}) =\mathbb{E}(e_t\pino{t-j}{k-j})= 0$ and hence an application of Lemma \ref{lemma:PD-standart-est} yields 
\begin{equation*}
  \biggl\|\sum_{t=(s+1)\vee(K_n+k+3)}^{(s+k+1)\wedge n}  (\II{t}{k}{j})^{(l)}\biggr\|_{2p} \lesssim \sqrt{k}. 
\end{equation*}
  The term $\|\II{s}{k}{j}  - (\II{s}{k}{j})^{(l)}\|_{2p}$ can be estimated in exactly the same way as in \eqref{eq:stepIII:e_te_(t-j,k-j) - e_t'e_(t-j,k-j)'}, i.e.,
  \begin{equation*}
      \|\II{s}{k}{j}  - (\II{s}{k}{j})^{(l)}\|_{2p} \lesssim\delta_{4p}^{I_{k-j}}(s-j-l) + \delta_{4p}^I(s-l).
  \end{equation*}
  Summing up, we have 
  \begin{equation*}
    III_2 \lesssim \sqrt{k}\sum_{s=K_n+k+2}^{n-1} \biggl(\delta_{4p}^{I_{k-j}}(s-j-l) + \delta_{4p}^I(s-l)\biggr).
  \end{equation*}
  In order to finish the proof, we observe that \eqref{eq:stepIII:stepIII-2-Burk} can be estimated by
  \begin{equation*}
  \begin{split}
    \biggl\|\sum_{l\in\zz}\mathcal{P}_l\biggl(\sum_{t=K_n+k+3}^{n} \sum_{s=t-1-k}^{t-1}  \I{t}{s}{k}{j}\biggr)\biggr\|_p \cformat{main-lemma-stepIII2-2}{&\lesssim \sqrt{ \sum_{l\in\zz} III_1^2 + III_2^2}\\}{}
                                                                                                                                                    &\lesssim \sqrt{\sum_{l\in\zz} III_1^2} + \sqrt{ \sum_{l\in\zz}  III_2^2}.
  \end{split}
  \end{equation*}
  Starting with the first term, we plug in the estimate \eqref{eq:stepIII:III_1-rate}, and use $(a+b)^2 \leq 2 (a^2 + b^2)$ again, to obtain
  \begin{equation}\label{eq:stepIII:sum-III_1^2}
  \begin{split}
      \sum_{l\in\zz} III_1^2 &\lesssim k \sum_{l\in\zz} \biggl(\sum_{t=K_n+k+3}^n \bigl(\delta_{4p}^{I_{k-j}}(t-j-l) + \delta_{4p}^I(t-l)\bigl)\biggr)^2 \\
      &\lesssim k \sum_{l\in \zz} \biggl(\sum_{t=K_n+k+3}^n \delta_{4p}^{I_{k-j}}(t-j-l)\biggr)^2+ \sum_{l\in \zz}\biggl(\sum_{t=K_n+k+3}^n \delta_{4p}^I(t-l)\biggr)^2.
  \end{split}
  \end{equation}
  Now we may sum up one factor in each of the two squares, which yields 
  \begin{equation*}
      \begin{split}
          \biggl(\sum_{t=K_n+k+3}^n \delta_{4p}^{I_{k-j}}(t-j-l)\biggr)^2 &\leq D_{4p}^{I_{k-j}}(0) \sum_{t=K_n+k+3}^n \delta_{4p}^{I_{k-j}}(t-j-l), \quad \text{and} \\
          \biggl(\sum_{t=K_n+k+3}^n \delta_{4p}^I(t-l)\biggr)^2 & \leq D_{4p}^I(0)\sum_{t=K_n+k+3}^n \delta_{4p}^I(t-l).
      \end{split}
  \end{equation*}

  Putting this back into \eqref{eq:stepIII:sum-III_1^2}, and using the fact that $D_{4p}^{I_{k-j}}(0)$ is uniformly bounded in $k$ and $j$ by Lemma \ref{lemma:dep-rates} and Assumption \ref{I2}, yields
  \begin{equation*}
      \begin{split}
          \sum_{l\in\zz} III_1^2 &\lesssim k \sum_{t=K_n+k+3}^n \sum_{l\in \zz} \delta_{4p}^I(t-l) + k\sum_{t=K_n+k+3}^n \sum_{l\in \zz} \delta_{4p}^{I_{k-j}}(t-j-l) \lesssim Nk,
      \end{split}
  \end{equation*}
  and hence, 
  \begin{equation*}
    \biggl( \sum_{l\in\zz} III_1^2 \biggr)^{\frac{1}{2}} \lesssim \sqrt{kN}. 
  \end{equation*}
  
  Using the same argument, we can show
  \begin{equation*}
    \biggl( \sum_{l\in\zz} III_2^2 \biggr)^{\frac{1}{2}} \lesssim\sqrt{kN}. 
  \end{equation*}
  In total, we have 
  \begin{multline*}
    \frac{1}{Nk}\biggl\|\sum_{t=K_n+k+3}^{n} \sum_{s=t-1-k}^{t-1} \sum_{j=1}^k \sigma_{k-j}^{-2} (\I{t}{s}{k}{j}-\mathbb{E}(\I{t}{s}{k}{j}))\biggr\|_p\\
      \lesssim \frac{1}{Nk}\sum_{j=1}^k \biggl( \sum_{l\in\zz} III_1^2 \biggr)^{\frac{1}{2}} +\frac{1}{kN} \biggl( \sum_{l\in\zz}  + III_2^2\biggr)^{\frac{1}{2}}
      \lesssim k^{\frac{1}{2}}N^{-\frac{1}{2}}. 
  \end{multline*}
  As $K_n^2/n \to 0$ by Assumption \ref{G3}, the last expression is $\lesssim k^{-1/2}$. 
  \end{proof}
  
  \begin{lemma}[Step IV, Part 1]\label{lemma:stepIV-1}
    Given Assumptions \ref{ass:main} and \ref{ass:weak:dep}, we have
    \begin{equation}\label{eq:stepIV}
      \frac{1}{Nk}\sum_{l=0}^k\bigg\| \sum_{t=K_n+k+3}^n \sum_{s=K_n+1}^{t-k-2} \sum_{j=1}^k \sigma_{k-j}^{-2} \mathcal{P}_{t-l}(\II{t}{k}{j})\II{s}{k}{j}\biggr\|_p \leq \frac{C_{p}}{\sqrt{k}},
    \end{equation}
    for $p=q/4$ and some constant $C_{p}$ depending only on $p$, and the process $X$.
  \end{lemma}

  \begin{proof}
    Note that  $\pino{t-j}{k-j}$ is $\mathcal{F}_{t-l-1}$-measurable whenever $j \geq l+1$. In this case, we have
    \begin{equation*}
    \begin{split}
        \mathcal{P}_{t-l}(\II{t}{k}{j}) &= \mathbb{E}(\II{t}{k}{j} \mid \mathcal{F}_{t-l}) - \mathbb{E}(\II{t}{k}{j} \mid \mathcal{F}_{t-l-1})= \pino{t-j}{k-j}\mathcal{P}_{t-l}(e_t).
    \end{split}
    \end{equation*}
    Exploiting this property, we may estimate the sum in \eqref{eq:stepIV} by two terms, $IV_1 + IV_2$, one in which $j \geq l+1$, and one in which $j \leq l$, i.e.,
    \begin{equation*}
      \begin{split}
        IV_1& = \frac{1}{Nk}\sum_{l=0}^k\bigg\| \sum_{t=K_n+k+3}^n \mathcal{P}_{t-l}(e_t)\sum_{s=K_n+1}^{t-k-2} \sum_{j=l+1}^k \sigma_{k-j}^{-2} \pino{t-j}{k-j}\II{s}{k}{j}\biggr\|_p, \quad \text{and} \\
      IV_2 & = \frac{1}{Nk}\sum_{l=1}^k\bigg\| \sum_{t=K_n+k+3}^n \sum_{s=K_n+1}^{t-k-2} \sum_{j=1}^l \sigma_{k-j}^{-2} \mathcal{P}_{t-l}(\II{t}{k}{j})\II{s}{k}{j}\biggr\|_p.
      \end{split}
    \end{equation*}
      Since $\pino{t-j}{k-j}\II{s}{k}{j}$ is $\mathcal{F}_{t-l-1}$-measurable in $IV_1$, we have 
      \begin{equation*}
        \mathcal{P}_{t-l}(e_t)\pino{t-j}{k-j}\II{s}{k}{j} = \mathcal{P}_{t-l}(\I{t}{s}{k}{j}),
      \end{equation*}
      and thus the sequence 
      \begin{equation*}
          \mathcal{P}_{t-l}(e_t)\sum_{s=K_n+1}^{t-k-2} \sum_{j=l+1}^k \sigma_{k-j}^{-2} \pino{t-j}{k-j}\II{s}{k}{j}
          =\mathcal{P}_{t-l}\biggl(\sum_{s=K_n+1}^{t-k-2} \sum_{j=l+1}^k \sigma_{k-j}^{-2} \I{t}{s}{k}{j}\biggr)
      \end{equation*}
      is a martingale difference sequence with respect to $t$. An application of Burkholder's inequality (see Lemma \ref{lemma:burkholder}) yields
       \begin{equation*}
      \begin{split}
        IV_1 & \leq \frac{1}{Nk}\sum_{l=0}^k\sqrt{\sum_{t=K_n+k+3}^n \bigg\|\mathcal{P}_{t-l}(e_t)  \sum_{s=K_n+1}^{t-k-2} \sum_{j=l+1}^k \sigma_{k-j}^{-2} \pino{t-j}{k-j}\II{s}{k}{j}\biggr\|_p^2}.
      \end{split}
    \end{equation*}
    Applying Hölder's inequality and using the fact that $\|\mathcal{P}_{t-l}(e_t)\|_{4p} \leq \delta_{4p}^I(l)$ (see \eqref{eq:P(Z)=E(Z-Z')} and the subsequent discussion), we may further estimate this term by
    \begin{equation*}
      \begin{split}
        IV_1 & \leq \frac{1}{Nk}\sum_{l=0}^k\sqrt{\sum_{t=K_n+k+3}^n \|\mathcal{P}_{t-l}(e_t)\|_{4p}^2\bigg\|  \sum_{s=K_n+1}^{t-k-2} \sum_{j=l+1}^k \sigma_{k-j}^{-2} \pino{t-j}{k-j}\II{s}{k}{j}\biggr\|_{\frac{4p}{3}}^2} \\
        &\leq \frac{1}{Nk}\sum_{l=0}^k\delta_{4p}^I(l)\sqrt{\sum_{t=K_n+k+3}^n\bigg\|  \sum_{s=K_n+1}^{t-k-2} \sum_{j=l+1}^k \sigma_{k-j}^{-2} \pino{t-j}{k-j}\II{s}{k}{j}\biggr\|_{\frac{4p}{3}}^2},
      \end{split}
    \end{equation*}
    and thus it is sufficient to show, that
    \begin{equation}\label{eq:stepIV:rate-goal-IV_1}
      \bigg\|\sum_{s=K_n+1}^{t-k-2} \sum_{j=l+1}^k \sigma_{k-j}^{-2} \pino{t-j}{k-j}\II{s}{k}{j}\biggr\|_{\frac{4p}{3}} \lesssim \sqrt{kN},
    \end{equation}
uniformly in $t$ and $l$.
    In the subsequent argument, it is more convenient if the summands are ascending in both $s$ and $j$, rather than just $s$. This leads to
    \begin{equation*}
      \begin{split}
        \bigg\|\sum_{s=K_n+1}^{t-k-2} \sum_{j=l+1}^k \sigma_{k-j}^{-2} \pino{t-j}{k-j}\II{s}{k}{j}\biggr\|_{\frac{4p}{3}} & = \bigg\|\sum_{j=0}^{k-l-1} \sigma_{j}^{-2}\pino{t-k+j}{j} \sum_{s=K_n+1}^{t-k-2}  \II{s}{k}{k-j}\biggr\|_{\frac{4p}{3}}.
      \end{split}
      \end{equation*}
      Writing $\pino{t-k+j}{j} = Q^{t-k+j}_{t-k}(\pino{t-k+j}{j}) + H^{t-k+j}_{t-k}(\pino{t-k+j}{j})$, with $H^t_s(Z) = \mathbb{E}(Z\mid\mathcal{F}^t_s)$, $Q^t_s(Z) = Z -H^t_s(Z)$ and $\mathcal{F}_s^t = \sigma(\varepsilon_t, \dots, \varepsilon_s)$ for $t\geq s$, we estimate this term by $IV_1' + IV_1''$, where
      \begin{equation*}
        \begin{split}
          IV_1' &= \bigg\|\sum_{j=0}^{k-l-1} \sigma_{j}^{-2}Q^{t-k+j}_{t-k}(\pino{t-k+j}{j}) \sum_{s=K_n+1}^{t-k-2}  \II{s}{k}{k-j}\biggr\|_{\frac{4p}{3}}, \quad \text{and} \\
          IV_1'' &=  \bigg\|\sum_{j=0}^{k-l-1} \sigma_{j}^{-2}H^{t-k+j}_{t-k}(\pino{t-k+j}{j}) \sum_{s=K_n+1}^{t-k-2}  \II{s}{k}{k-j}\biggr\|_{\frac{4p}{3}}.
        \end{split}
      \end{equation*}

      Let us start with $IV_1'$. 
      Since $\sigma_j^{-2}\leq \sigma^{-2}$ (see Lemma \ref{lemma:basic-props-AR(oo)}), an application of the triangle and Hölder's inequality gives
      \begin{equation*}
          \begin{split}
            IV_1' \leq \frac{1}{\sigma^2} \sum_{j=0}^{k-l-1} \|Q^{t-k+j}_{t-k}(\pino{t-k+j}{j})\|_{4p}\bigg\|\sum_{s=K_n+1}^{t-k-2}  \II{s}{k}{k-j}\biggr\|_{2p}.
          \end{split}
      \end{equation*}
      Since $\mathbb{E}(\II{s}{k}{k-j}) = 0$, we may apply Lemma \ref{lemma:PD-standart-est} to the second factor, yielding 
      \begin{equation}\label{eq:stepIV:sum_s <~ N^(1/2)}
        \bigg\|\sum_{s=K_n+1}^{t-k-2}  e_s\pino{s-k+j}{j}\biggr\|_{2p} \lesssim \sqrt{N}. 
      \end{equation}
      By Lemma \ref{lemma:Q}, we have $\|Q^{t-k+j}_{t-k}(\pino{t-k+j}{j})\| \lesssim (j+1)^{-5/2}$. The constants in this estimate do not depend on $j$, as $D^{I_j}_{4p}(\frac{5}{2})$ is uniformly bounded in $j$ by Lemma \ref{lemma:dep-rates}, and 
      $\|\pino{0}{j}\|_{4p}$ is uniformly bounded in $j$ by Lemma \ref{lemma:basic-props-AR(oo)}.
      
      In total, $IV_1'$ can be estimated by 
      \begin{equation*}
        IV_1' \lesssim \sum_{j=0}^{k-l-1} \|Q^{t-k+j}_{t-k}(\pino{t-k+j}{j})\|_{4p}\bigg\|\sum_{s=K_n+1}^{t-k-2}  \II{s}{k}{k-j}\biggr\|_{2p} \lesssim \sqrt{N} \sum_{j=0}^{k-l-1} (j+1)^{-\frac{5}{2}} \lesssim \sqrt{N}.
      \end{equation*}

      Next, we have to deal with $IV_1''$. 
      Since $t-k > s$ in $IV_1''$, $H^{t-k+j}_{t-k}(\pino{t-k+j}{j})$ is independent of $\II{s}{k}{k-j} = e_s \pino{s-k+j}{j}$, and hence we have 
      \begin{equation*}
        \mathbb{E}\bigl(H^{t-k+j}_{t-k}(\pino{t-k+j}{j})\II{s}{k}{k-j}\bigr) = \mathbb{E}\bigl(H^{t-k+j}_{t-k}(\pino{t-k+j}{j})\bigr)\mathbb{E}(\II{s}{k}{k-j}) = 0.
      \end{equation*} 
      Thus, we may perform a martingale difference approximation in the sense of Lemma \ref{lemma:mdsa},
      \begin{equation*}
        IV_1'' = \bigg\|\sum_{j=0}^{k-l-1} \sigma_{j}^{-2}\sum_{s=K_n+1}^{t-k-2}\sum_{m=0}^\infty\mathcal{P}_{t-k+j-m}\biggl(H^{t-k+j}_{t-k}(\pino{t-k+j}{j})\II{s}{k}{k-j}\biggr)\biggr\|_{\frac{4p}{3}}.
      \end{equation*}
      If $m>j$, $H^{t-k+j}_{t-k}(\pino{t-k+j}{j})$ is independent of $\sigma(e_s\pino{s-k+j}{j},\mathcal{F}^{t-k+j-m}_{-\infty}) \subseteq \sigf{s}$. In this case, we have 
      \begin{equation*}
        \mathcal{P}_{t-k+j-m}\bigl(H^{t-k+j}_{t-k}(\pino{t-k+j}{j})\II{s}{k}{k-j}\bigr) = \mathbb{E}(\pino{t-k+j}{j}) \mathcal{P}_{t-k+j-m}(\II{s}{k}{k-j}) = 0.
      \end{equation*}
      However, if $m\leq j$, we get
      \begin{equation*}
        \mathcal{P}_{t-k+j-m}\bigl(H^{t-k+j}_{t-k}(\pino{t-k+j}{j})\II{s}{k}{k-j}\bigr) = \mathcal{P}_{t-k+j-m}\bigl(H^{t-k+j}_{t-k}(\pino{t-k+j}{j})\bigr)\II{s}{k}{k-j},
      \end{equation*}
      as in this case $t-k+j-m \geq t-k > s$, and thus $e_s\pino{s-k+j}{j}$ is $\mathcal{F}_{t-k+j-m-1}$-measurable. 
      Combining the last two identities and summing up, we arrive at
      \begin{equation*}
          IV_1'' = \bigg\|\sum_{m=0}^{k-l-1}\sum_{j=m}^{k-l-1} \sigma_{j}^{-2}\sum_{s=K_n+1}^{t-k-2}\mathcal{P}_{t-k+j-m}\bigl(H^{t-k+j}_{t-k}(\pino{t-k+j}{j})\bigr)e_s\pino{s-k+j}{j}\biggr\|_{\frac{4p}{3}}.
      \end{equation*}
      Performing a triangle inequality in the sum over $m$, Burkholder's inequality (see Lemma \ref{lemma:burkholder}) with respect to the sum over $j$, using $\sigma_j^{-2} \leq \sigma^{-2}$ (see Lemma \ref{lemma:basic-props-AR(oo)}) and applying Hölder's inequality, we get 
      \begin{equation}\label{eq:stepIV:IV_1''-basic-est}
        IV_1'' \lesssim \sum_{m=0}^{k-l-1} \biggl(\sum_{j=m}^{k-l-1}\biggl\|\mathcal{P}_{t-k+j-m}\bigl(H^{t-k+j}_{t-k}(\pino{t-k+j}{j})\bigr)\biggr\|_{4p}^2\biggl\|\sum_{s=K_n+1}^{t-k-2}\II{s}{k}{k-j}\biggr\|_{2p}^2\biggr)^{\frac{1}{2}}.
      \end{equation}
      By Lemma \ref{lemma:PD-standart-est},
    \begin{equation*}
      \bigg\|\sum_{s=K_n+1}^{t-k-2}  \II{s}{k}{k-j}\biggr\|_{2p} \lesssim \sqrt{N}.
      \end{equation*}
      Hence, we just have to focus on $\mathcal{P}_{t-k+j-m}\bigl(H^{t-k+j}_{t-k}(\pino{t-k+j}{j})\bigr)$. 
      Using the tower property, we obtain
      \cformat{main-lemma-stepIV1-1}{
      \begin{multline*}
              \mathcal{P}_{t-k+j-m}\bigl(H^{t-k+j}_{t-k}(\pino{t-k+j}{j})\bigr) \\
              =\mathbb{E}\bigl(H^{t-k+j}_{t-k}(\pino{t-k+j}{j})\mid \mathcal{F}_{t-k+j-m}\bigr) - \mathbb{E}\bigl(H^{t-k+j}_{t-k}(\pino{t-k+j}{j})\mid \mathcal{F}_{t-k+j-m-1}\bigr) \\
              =\mathbb{E}\bigl(\mathbb{E}\bigl(\pino{t-k+j}{j}\mid \mathcal{F}^{t-k+j}_{t-k}\bigr)\mid \mathcal{F}_{t-k+j-m}\bigr) - \mathbb{E}\bigl(\mathbb{E}\bigl(\pino{t-k+j}{j}\mid \mathcal{F}^{t-k+j}_{t-k}\bigl)\mid \mathcal{F}_{t-k+j-m-1}\bigr)\\
              = \mathbb{E}(\pino{t-k+j}{j}\mid \mathcal{F}^{t-k+j-m}_{t-k}) - \mathbb{E}(\pino{t-k+j}{j}\mid \mathcal{F}^{t-k+j-m-1}_{t-k}). 
          \end{multline*}}{
          \begin{equation*}
            \mathcal{P}_{t-k+j-m}\bigl(H^{t-k+j}_{t-k}(\pino{t-k+j}{j})\bigr) = \mathbb{E}(\pino{t-k+j}{j}\mid \mathcal{F}^{t-k+j-m}_{t-k}) - \mathbb{E}(\pino{t-k+j}{j}\mid \mathcal{F}^{t-k+j-m-1}_{t-k}). 
        \end{equation*}}
      Since we have (cf. Equation \eqref{eq:P(Z)=E(Z-Z')} and the subsequent discussion)
    \begin{equation*}
        \mathbb{E}\bigl(\pino{t-k+j}{j}\mid\sigfdouble{t-k+j-m-1}{t-k}\bigr) = \mathbb{E}\bigl(\pino{t-k+j}{j}^{(t-k+j-m)}\mid\sigfdouble{t-k+j-m-1}{t-k}\bigr) = \mathbb{E}\bigl(\pino{t-k+j}{j}^{(t-k+j-m)}\mid\sigfdouble{t-k+j-m}{t-k}\bigr),
    \end{equation*}
    we may write
    \begin{equation*}
        \mathcal{P}_{t-k+j-m}\bigl(H^{t-k+j}_{t-k}(\pino{t-k+j}{j})\bigl) = \mathbb{E}\bigl(\pino{t-k+j}{j}-\pino{t-k+j}{j}^{(t-k+j-m)}\mid\sigfdouble{t-k+j-m}{t-k}\bigl).
    \end{equation*}
    Using the fact that the conditional expectation is a $L^{4p}$ contraction again, we get
    \begin{equation*}
    \begin{split}
      \bigl\|\mathcal{P}_{t-k+j-m}\bigl(H^{t-k+j}_{t-k}(\pino{t-k+j}{j})\bigr)\bigl\|_{4p} \cformat{main-lemma-stepIV1-2}{&= \bigl\|\mathbb{E}\bigl(\pino{t-k+j}{j}-\pino{t-k+j}{j}^{(t-k+j-m)}\mid\sigfdouble{t-k+j-m}{t-k}\bigr)\bigr\|_{4p} \\}{}
        &\leq\bigl\|\pino{t-k+j}{j}-\pino{t-k+j}{j}^{(t-k+j-m)}\bigr\|_{4p} =\delta_{4p}^{I_j}(m).
    \end{split}
    \end{equation*}
    Plugging this estimate and \eqref{eq:stepIV:sum_s <~ N^(1/2)} back into \eqref{eq:stepIV:IV_1''-basic-est}, we get 
    \begin{equation*}
      IV_1'' \lesssim \sqrt{Nk} \sum_{m=0}^{k-l-1} \delta_{4p}^{I_j}(m) \lesssim \sqrt{kN},
    \end{equation*}
    as $D_{4p}^{I_j}(0)$ is uniformly bounded in $j$ by Lemma \ref{lemma:dep-rates}. 
    Together with $IV_{1}'\lesssim \sqrt{N}$, this proves \eqref{eq:stepIV:rate-goal-IV_1}, finishing the term $IV_{1}$.
      
      Now we move on to $IV_2$. 
      An application of the triangle inequality yields ($\sigma_j^{-2} \leq \sigma^{-2}$, by Lemma \ref{lemma:basic-props-AR(oo)})
      \begin{equation*}
        IV_2 \lesssim \frac{1}{Nk}\sum_{l=1}^k\sum_{j=1}^l\bigg\| \sum_{t=K_n+k+3}^n \mathcal{P}_{t-l}(\II{t}{k}{j})\sum_{s=K_n+1}^{t-k-2} \II{s}{k}{j}\biggr\|_p.
      \end{equation*} 
      Since $\mathcal{P}_{t-l}(\II{t}{k}{j})\II{s}{k}{j} = \mathcal{P}_{t-l}(\I{t}{s}{k}{j})$ as $\II{s}{k}{j}$ is $\sigf{t-l-1}$-measurable, the sequence
      \begin{equation*}
        \sum_{s=K_n+1}^{t-k-2} \mathcal{P}_{t-l}(\II{t}{k}{j})\II{s}{k}{j}
        = \mathcal{P}_{t-l}\biggl(\sum_{s=K_n+1}^{t-k-2} \I{t}{s}{k}{j}\biggr)
      \end{equation*}
      is a martingale difference sequence. An application of Lemma \ref{lemma:burkholder}, Hölder's inequality, and Lemma \ref{lemma:PD-standart-est} yields
      \begin{equation}\label{eq:stepIV:IV_2-simple-rate}
        \begin{split}
          IV_{2}&\lesssim \frac{1}{Nk}\sum_{l=1}^k\sum_{j=1}^l\biggl( \sum_{t=K_n+k+3}^n \|\mathcal{P}_{t-l}(\II{t}{k}{j})\|_{2p}^2\biggl\|\sum_{s=K_n+1}^{t-k-2} \II{s}{k}{j}\biggr\|_{2p}^2\biggr)^{\frac{1}{2}}\\
                & \lesssim \frac{1}{k}\sum_{l=1}^k \sum_{j=1}^l \|\mathcal{P}_{0}(\II{l}{k}{j})\|_{2p} = \frac{1}{k}\sum_{j=1}^k\sum_{l=j}^k \|\mathcal{P}_{0}(\II{l}{k}{j})\|_{2p}
        \end{split}
      \end{equation}
      We have to treat the cases $l=j$ and $l>j$ separately. For $l=j$, we are going to show that 
      \begin{equation}\label{eq:stepIV:IV_2:l=j-rate}
        \|\mathcal{P}_{0}(\II{l}{k}{j})\|_{2p} = \|\mathcal{P}_{0}(e_l\pino{l-j}{k-j})\|_{2p} \lesssim j^{-\frac{5}{2}},
      \end{equation}
      while $l>j$ will lead to
      \begin{equation}\label{eq:stepIV:IV_2:l>j-rate}
        \|\mathcal{P}_{0}(\II{l}{k}{j})\|_{2p} = \|\mathcal{P}_{0}(e_l\pino{l-j}{k-j})\|_{2p} \lesssim j^{-\frac{5}{2}}(l-j)^{-\frac{5}{2}} + l^{-\frac{5}{2}}
      \end{equation}
      (note that the expression \eqref{eq:stepIV:IV_2:l>j-rate} is not defined for $l=j$). Plugging these rates into \eqref{eq:stepIV:IV_2-simple-rate}, we immediately get $IV_2 \lesssim k^{-1}$, as $D^I_{4p}(0) < \infty$ by Assumption \ref{I2}.

%

      Since the case $l=j$ can be dealt with by a simpler version of the argument used for the case $l>j$, we focus on the latter.
      We may apply an argument similar to the one in \ref{II1-a} of Lemma \ref{lemma:stepII}, where $\mathcal{P}_{0}(\cdot)$ takes the role of $\mathbb{E}(\cdot)$. 
      Since $j< l$, $H^l_{l-j+1}(e_l) = \mathbb{E}(e_l\mid \sigfdouble{l}{l-j+1})$ is independent of $\sigf{0}$. 
      Thus, we have 
      \begin{equation*}
      \begin{split}
        \mathcal{P}_{0}\bigl(H^l_{l-j+1}(e_l)\pino{l-j}{k-j}\bigr) \cformat{main-lemma-stepIV1-3}{&= \mathbb{E}\bigl(H^l_{l-j+1}(e_l)\pino{l-j}{k-j} \mid \sigf{0}\bigr) - \mathbb{E}\bigl(H^l_{l-j+1}(e_l)\pino{l-j}{k-j} \mid \sigf{-1}\bigr) \\
                                                                                                  &=\mathbb{E}\bigl(H^l_{l-j+1}(e_l)\bigr)\mathbb{E}\bigl(\pino{l-j}{k-j} \mid \sigf{0}\bigr) - \mathbb{E}\bigl(H^l_{l-j+1}(e_l)\bigr)\mathbb{E}\bigl(\pino{l-j}{k-j} \mid \sigf{-1}\bigr)\\}{}
          &= 0,
      \end{split}
      \end{equation*}
as $\mathbb{E}\bigl(H^l_{l-j+1}(e_l)\bigr) = \mathbb{E}(e_l) = 0$. 
      In other words, 
      \begin{equation*}
        \mathcal{P}_{0}(e_{l}\pino{l-j}{k-j}) = \mathcal{P}_{0}(Q_{l-j+1}^{l}(e_{l})\pino{l-j}{k-j}).
      \end{equation*}
      Writing $\pino{l-j}{k-j} = Q^l_{l-j+1}(\pino{l-j}{k-j}) + H^l_{l-j+1}(\pino{l-j}{k-j})$, this gives 
      \begin{equation*}
      \begin{split}
          \mathcal{P}_{0}\bigl(e_l\pino{l-j}{k-j}\bigr) &= \mathcal{P}_{0}\bigl(Q_{l-j+1}^l(e_l)\pino{l-j}{k-j}\bigr) \\
          &= \mathcal{P}_{0}\bigl(Q_{l-j+1}^l(e_l)Q_{1}^{l-j}(\pino{l-j}{k-j})\bigr) + \mathcal{P}_{0}\bigl(Q_{l-j+1}^l(e_l)H_{1}^{l-j}(\pino{l-j}{k-j})\bigr).
      \end{split}
      \end{equation*}
     As $H^{l-j}_{1}(\pino{l-j}{k-j})$ is independent of the $\sigma$-algebra $\sigma\bigl(H^{l}_{l-j+1}(e_l), \sigf{0}\bigr)$ (and hence also of $\sigma\bigl(H^{l}_{l-j+1}(e_l), \sigf{-1}\bigr)$), this yields 
      \begin{equation*}
      \begin{split}
          \mathcal{P}_{0}\bigl(H_{l-j+1}^l(e_l)H_{1}^{l-j}(\pino{l-j}{k-j})\bigr) & = \mathbb{E}\bigl(H_{1}^{l-j}(\pino{l-j}{k-j})\bigr) \mathcal{P}_{0}\bigl(H_{l-j+1}^l(e_l)\bigr) = 0,
      \end{split}
      \end{equation*}
      as $\mathbb{E}\bigl(H_{1}^{l-j}(\pino{l-j}{k-j})\bigr) = \mathbb{E}(\pino{l-j}{k-j}) = 0$. 
      Exploiting this fact, we get
      \begin{equation*}
          \mathcal{P}_{0}\bigl(Q_{l-j+1}^l(e_l)H_{1}^{l-j}(\pino{l-j}{k-j})\bigr) = \mathcal{P}_{0}\bigl(e_l H_{1}^{l-j}(\pino{l-j}{k-j})\bigr). 
      \end{equation*}
      Since $H^l_{1}(e_l)H_{1}^{l-j}(\pino{l-j}{k-j})$ is independent of $\sigf{0}$, we have 
      \begin{multline*}
          \mathcal{P}_{0}\bigl(H^l_{1}(e_l)H_{1}^{l-j}(\pino{l-j}{k-j})\bigr)
          = \mathbb{E}\bigl(H^l_{1}(e_l)H_{1}^{l-j}(\pino{l-j}{k-j})\bigr) - \mathbb{E}\bigl(H^l_{1}(e_l)H_{1}^{l-j}(\pino{l-j}{k-j})\bigr) = 0,
      \end{multline*}
      which in turn implies
      \begin{equation*}
          \mathcal{P}_{0}\bigl(e_l H_{1}^{l-j}(\pino{l-j}{k-j})\bigr) = \mathcal{P}_{0}\bigl(Q^l_{1}(e_l) H_{1}^{l-j}(\pino{l-j}{k-j})\bigr). 
      \end{equation*}
      Putting everything together, we get
      \begin{equation*}
          \begin{split}
              \mathcal{P}_{0}(e_l\pino{l-j}{k-j}) & = \mathcal{P}_{0}\bigl(Q_{l-j+1}^l(e_l)Q_{1}^{l-j}(\pino{l-j}{k-j})\bigr) \\
              &+ \mathcal{P}_{0}\bigl(Q_{1}^l(e_l)H_{1}^{l-j}(\pino{l-j}{k-j})\bigr).
          \end{split}
      \end{equation*}
      Thus, using the fact that the conditional expectation is a $L^{2p}$ contraction, and applying Hölder's inequality, we arrive at
      \begin{equation}\label{eq:stepIV:P(e_te_(t-j,k-j)-est}
          \begin{split}
            \|\mathcal{P}_{0}(e_l\pino{l-j}{k-j})\|_{4p} \cformat{main-lemma-stepIV1-4}{& \leq 2 \|Q_{l-j+1}^l(e_l)\|_{4p}\|Q_{1}^{l-j}(\pino{l-j}{k-j})\|_{4p} \\
                                                                                        &+ 2\|Q_{1}^l(e_l)\|_{4p}\|H_{1}^{l-j}(\pino{l-j}{k-j})\|_{4p}\\}{}
              &\lesssim \|Q_{l-j+1}^l(e_l)\|_{4p}\|Q_{1}^{l-j}(\pino{l-j}{k-j})\|_{4p} + \|Q_{1}^l(e_l)\|_{4p}.
          \end{split}
      \end{equation}
      An application of Lemma \ref{lemma:Q} yields
      \begin{equation*}
      \begin{split}
          \|Q_{l-j+1}^l(e_l)\|_{4p} & \lesssim j^{-\frac{5}{2}}, \\
          \|Q_{1}^{l-j}(\pino{l-j}{k-j})\|_{4p}&\lesssim (l-j)^{-\frac{5}{2}}, \quad \text{and} \\
          \|Q_{1}^l(e_l)\|_{4p} & \lesssim l^{-\frac{5}{2}}. 
      \end{split}
      \end{equation*}
      Plugging these rates into \eqref{eq:stepIV:P(e_te_(t-j,k-j)-est}, we get 
      \begin{equation}\label{eq:stepIV:rates-P(e_le_(l-j,k-j))}
          \|\mathcal{P}_{0}(e_l\pino{l-j}{k-j})\|_{4p} \lesssim j^{-\frac{5}{2}}(l-j)^{-\frac{5}{2}} + l^{-\frac{5}{2}},
      \end{equation}
      which is sufficient to establish $IV_2 \lesssim k^{-1}$ by \eqref{eq:stepIV:IV_2-simple-rate}, finishing the proof.

  \end{proof}

  \begin{lemma}[Step IV, Part 2] \label{lemma:stepIV-2}
    Given Assumptions \ref{ass:main} and \ref{ass:weak:dep}, we have
    \begin{equation}\label{eq:stepIV-2}
      \frac{1}{Nk} \bigg\|\sum_{l=k+1}^{\infty} \sum_{t=K_n+k+3}^n \sum_{s=K_n+1}^{t-k-2} \sum_{j=1}^k \sigma_{k-j}^{-2} \mathcal{P}_{t-l}(\I{t}{s}{k}{j})\biggr\|_p \leq \frac{C_{p}}{\sqrt{k}},
    \end{equation}
    for $p=q/4$ and some constant $C_{p}$ depending only on $p$, and the process $X$. 
  \end{lemma}

  \begin{proof}
    We note that for $l \leq t-s-1$, the term $\II{s}{k}{j} = e_{s}\pino{s-j}{k-j}$ is $\sigf{t-l-1}$-measurable, and hence $\mathcal{P}_{t-l}(\I{t}{s}{k}{j}) = \mathcal{P}_{t-l}(\II{t}{k}{j})\II{s}{k}{j}$. 
    We may split the sum in \eqref{eq:stepIV-2} along $l=t-s$, and estimate the left-hand side of \eqref{eq:stepIV-2} by $IV_3 + IV_4$, where
    \begin{equation*}
      \begin{split}
        IV_3 &= \frac{1}{Nk} \bigg\| \sum_{t=K_n+k+3}^n \sum_{s=K_n+1}^{t-k-2} \sum_{j=1}^k\sum_{l=k+1}^{t-s-1} \sigma_{k-j}^{-2} \mathcal{P}_{t-l}(\II{t}{k}{j})\II{s}{k}{j}\biggr\|_p, \quad \text{and} \\
        IV_4 &= \frac{1}{Nk} \bigg\| \sum_{t=K_n+k+3}^n \sum_{s=K_n+1}^{t-k-2} \sum_{j=1}^k\sum_{l=t-s}^{\infty} \sigma_{k-j}^{-2} \mathcal{P}_{t-l}(\I{t}{s}{k}{j})\biggr\|_p.
      \end{split}
    \end{equation*}
    Starting with $IV_3$, we may change the order of summation, and apply the triangle inequalities to the sum over $l$ and $j$. Since $\sigma_{k-j}^{-2} \leq \sigma^{-2}$ (see Lemma \ref{lemma:basic-props-AR(oo)}), this yields
    \begin{equation}\label{eq:stepIV-2:IV_3-first-est}
      IV_3 \lesssim \frac{1}{Nk}\sum_{l=k+1}^{n-K_n-2} \sum_{j=1}^k \bigg\|\sum_{t=K_n+l+2}^n \sum_{s=K_n+1}^{t-l-1} \mathcal{P}_{t-l}(\II{t}{k}{j})\II{s}{k}{j}\biggr\|_p.
    \end{equation}
    Our goal for this expression is to show, that 
    \begin{equation}\label{eq:stepIV-2:IV_3-rate-goal}
      \bigg\|\sum_{t=K_n+l+2}^n \sum_{s=K_n+1}^{t-l-1} \mathcal{P}_{t-l}(\II{t}{k}{j})\II{s}{k}{j}\biggr\|_p \lesssim N\biggl(j^{-\frac{5}{2}}(l-j)^{-\frac{5}{2}} + l^{-\frac{5}{2}}\biggr). 
    \end{equation}
    Plugging this into \eqref{eq:stepIV-2:IV_3-first-est}, we get 
    \begin{equation*}
        IV_3 \lesssim \frac{1}{k}\sum_{l=k+1}^\infty \sum_{j=1}^k \biggl(j^{-\frac{5}{2}}(l-j)^{-\frac{5}{2}} + l^{-\frac{5}{2}}\biggr) = \frac{1}{k}\sum_{j=1}^k j^{-\frac{5}{2}} \sum_{l=k+1-j}^\infty l^{-\frac{5}{2}} + \sum_{l=k+1}^\infty l^{-\frac{5}{2}} \lesssim k^{-1}.
    \end{equation*}
    To establish Estimate \eqref{eq:stepIV-2:IV_3-rate-goal}, we note that by the introductory remarks of this proof, we have $\mathcal{P}_{t-l}(\II{t}{k}{j})\II{s}{k}{j}=\mathcal{P}_{t-l}(\I{t}{s}{k}{j})$, and hence the term
    \begin{equation}
      \sum_{s=K_n+1}^{t-l-1}\mathcal{P}_{t-l}(\II{t}{k}{j})\II{s}{k}{j} = \mathcal{P}_{t-l}\biggl(\sum_{s=K_n+1}^{t-l-1}\I{t}{s}{k}{j}\biggr)
    \end{equation}
    is a martingale difference sequence with respect to $t$. 
    \cformat{main-lemma-stepIV2-1}{Applying Burkholder's inequality (Lemma \ref{lemma:burkholder}) to \eqref{eq:stepIV-2:IV_3-rate-goal}, and using the fact that $\|\mathcal{P}_{t-l}(\II{t}{k}{j})\|_{2p} = \|\mathcal{P}_{0}(\II{l}{k}{j})\|_{2p}$, we get
      \begin{equation*}
      \begin{split}
        \bigg\|\sum_{t=K_n+l+2}^n \sum_{s=K_n+1}^{t-l-1} \mathcal{P}_{t-l}(\II{t}{k}{j})\II{s}{k}{j}\biggr\|_p &\lesssim \biggl(\sum_{t=K_n+l+2}^n\biggl\|\sum_{s=K_n+1}^{t-l-1} \mathcal{P}_{t-l}(\II{t}{k}{j})\II{s}{k}{j}\biggr\|_p^2\biggr)^{\frac{1}{2}} \\
                                                                                                               & \lesssim \biggl(\sum_{t=K_n+l+2}^n \|\mathcal{P}_{0}(\II{l}{k}{j})\|_{2p}^2 \biggl\|\sum_{s=K_n+1}^{t-l-1}\II{s}{k}{j}\biggr\|_p^2\biggr)^{\frac{1}{2}}.
      \end{split}
    \end{equation*}
    Since $\mathbb{E}(e_s\pino{s-j}{k-j})=0$ we may apply Lemma \ref{lemma:PD-standart-est} to the effect of (cf. Estimate \eqref{eq:stepIII:part1:example-for-part2})
    \begin{equation*}
      \biggl\|\sum_{s=K_n+1}^{t-l-1}e_s \pino{s-j}{k-j}\biggr\|_p \lesssim \sqrt{N}, 
    \end{equation*}
    as there are at most $N$ terms in this sum. Since $\|\mathcal{P}_{0}(e_l \pino{l-j}{k-j})\|_{2p}$ does not depend on $t$ this yields the estimate
    \begin{equation*}
        \bigg\|\sum_{t=K_n+l+2}^n \sum_{s=K_n+1}^{t-l-1} \mathcal{P}_{t-l}(e_t \pino{t-j}{k-j})e_s \pino{s-j}{k-j}\biggr\|_p \lesssim N\|\mathcal{P}_{0}(e_l \pino{l-j}{k-j})\|_{2p}.
    \end{equation*}}{
    Applying Burkholder's inequality (Lemma \ref{lemma:burkholder}) to \eqref{eq:stepIV-2:IV_3-rate-goal}, Lemma \ref{lemma:PD-standart-est} and  $\mathcal{P}_{t-l}(\II{t}{k}{j}) \disteq \mathcal{P}_{0}(\II{l}{k}{j})$ yield
    \begin{equation*}
      \begin{split}
        \bigg\|\sum_{t=K_n+l+2}^n \sum_{s=K_n+1}^{t-l-1} \mathcal{P}_{t-l}(\II{t}{k}{j})\II{s}{k}{j}\biggr\|_p 
                                                                                                               & \lesssim N\|\mathcal{P}_{0}(e_l \pino{l-j}{k-j})\|_{2p}.
      \end{split}
    \end{equation*}}
    Using Estimate \eqref{eq:stepIV:rates-P(e_le_(l-j,k-j))} for $\|\mathcal{P}_{0}(e_l \pino{l-j}{k-j})\|_{2p}$ again, immediately yields \eqref{eq:stepIV-2:IV_3-rate-goal}.

    Moving on with $IV_4$, we note that for $l=t-s,\dots, t-s+j-1$, the term $\pino{s-j}{k-j}$ is $\sigf{t-l-1}$ measurable (and hence $\sigf{t-l}$ measurable). Hence, we have 
    \begin{equation*}
      \mathcal{P}_{t-l}(\I{t}{s}{k}{j}) = \mathcal{P}_{t-l}(\II{t}{k}{j}e_s )\pino{s-j}{k-j}.
    \end{equation*}
    To exploit this idea, we may estimate $IV_4$ by $IV_4' + IV_4''$, where

    \begin{equation*}
      \begin{split}
        IV_4' & = \frac{1}{Nk} \bigg\| \sum_{t=K_n+k+3}^n \sum_{s=K_n+1}^{t-k-2} \sum_{j=1}^k\sum_{l=t-s}^{t-s+j-1} \sigma_{k-j}^{-2} \mathcal{P}_{t-l}(\I{t}{s}{k}{j})\biggr\|_p, \quad \text{and} \\
        IV_4'' & = \frac{1}{Nk} \bigg\| \sum_{t=K_n+k+3}^n \sum_{s=K_n+1}^{t-k-2} \sum_{j=1}^k\sum_{l=t-s+j}^{\infty} \sigma_{k-j}^{-2} \mathcal{P}_{t-l}(\I{t}{s}{k}{j})\biggr\|_p.
      \end{split}
    \end{equation*}
    Starting with $IV_{4}'$, we note that the distance $h=t-s$, between $t$ and $s$, takes the values $h=k+1,\dots,N-1$. Each of the possible values of $h$ appears no more than $N$ times. 
Using the triangle inequality, Lemma \ref{lemma:basic-props-AR(oo)}, and $\mathcal{P}_{t-l}(e_{t}\pino{t-j}{k-j}e_{t-h}) \disteq \mathcal{P}_{0}(e_{l}\pino{l-j}{k-j}e_{l-h})$, we get
    \begin{equation}\label{eq:stepIV-2:IV_4'-est}
        IV_4' \lesssim \frac{1}{k}\sum_{j=1}^k \sum_{h=k+1}^{N-1} \sum_{l=h}^{h+j-1} \|\mathcal{P}_{0}(e_l \pino{l-j}{k-j}e_{l-h})\|_{\frac{4p}{3}}.
    \end{equation}
    \cformat{main-lemma-stepIV2-2}{We are going to use a similar argument as in \eqref{eq:stepIV:IV_2:l=j-rate} and \eqref{eq:stepIV:IV_2:l>j-rate}.
    Again, we treat the cases $l=h$ and $l>h$ separately. 
    However, the argument that we use to derive the necessary estimates for both cases is very similar.
    We note that $H^l_{l-h+1}(e_l\pino{l-j}{k-j})$ is independent of the $\sigma$-algebra $\sigma(e_{l-h},\sigf{0}) = \sigf{0}$ (and hence also of $\sigf{-1}$) as $l \geq h$.
    This allows us to replace $e_l\pino{l-j}{k-j}$ with $Q^l_{l-h+1}(e_l\pino{l-j}{k-j})$, ie.,
    \begin{equation}\label{eq:stepIV-2:IV_4':P_0-first-step}
        \mathcal{P}_{0}(e_l \pino{l-j}{k-j}e_{l-h}) = \mathcal{P}_{0}\bigl(Q^l_{l-h+1}(e_l\pino{l-j}{k-j})e_{l-h}\bigr).
    \end{equation}
    For the case $l=h$, this is already good enough. 
    Using the fact that the conditional expectation is a $L^{4p/3}$ contraction and Hölder's inequality, we may estimate 
    \begin{equation*}
        \begin{split}
            \bigl\|\mathcal{P}_{0}\bigl(e_l \pino{l-j}{k-j}e_{0}\bigr)\bigr\|_{\frac{4p}{3}} &= \bigl\|\mathcal{P}_{0}\bigl(Q^l_1(e_l\pino{l-j}{k-j})e_{0}\bigr)\bigr\|_{\frac{4p}{3}} \\
            & \leq \|\mathbb{E}\bigl(Q^l_1(e_l\pino{l-j}{k-j})e_{0} \mid \sigf{0}\bigr)\|_{\frac{4p}{3}} + \bigl\|\mathbb{E}\bigl(Q^l_1(e_l\pino{l-j}{k-j})e_{0} \mid \sigf{-1}\bigr)\bigr\|_{\frac{4p}{3}} \\
            &\leq 2\bigl\|Q^l_1\bigl(e_l\pino{l-j}{k-j}\bigr)e_{0}\bigr\|_{\frac{4p}{3}} \\
            & \leq 2\|e_0\|_{4p}\bigl\|Q^l_1(e_l\pino{l-j}{k-j})\bigr\|_{2p} \lesssim \bigl\|Q^l_1(e_l\pino{l-j}{k-j})\bigr\|_{2p}. 
        \end{split}
    \end{equation*}
    By Lemma \ref{lemma:Q}, we have  $\bigl\|Q^l_1(e_l\pino{l-j}{k-j})\bigr\|_{2p} \lesssim (h-j)^{-5/2} + h^{-5/2}$. 
    In light of this and \eqref{eq:stepIV-2:IV_4'-est}, we may estimate
    \begin{equation*}
        \begin{split}
            IV_4' \lesssim \frac{1}{k} \sum_{j=1}^k \sum_{h=k+1}^{N-1} \biggl((h-j)^{-\frac{5}{2}} + h^{-\frac{5}{2}}\biggr) \lesssim k^{-\frac{1}{2}}.
        \end{split}
    \end{equation*}
    Here we have again used, that $\sum_{m\geq h} m^\alpha \lesssim h^{\alpha +1}$ for $\alpha < -1$ and $\sum_{1\leq m \leq h} m^\beta \lesssim h^{\beta + 1}$ for $\beta >-1$.

    Next, we move on to the case $l>h$. 
    Recall that Equation \eqref{eq:stepIV-2:IV_4':P_0-first-step} is still valid in this case. 
    In view of this, we may write $e_{l-h} = Q^{l-h}_1(e_{l-h}) + H^{l-h}_1(e_{l-h})$, which gives 
    \begin{equation}
        \begin{split}
            \mathcal{P}_{0}\bigl(e_l \pino{l-j}{k-j}e_{l-h}\bigr) &= \mathcal{P}_{0}\bigl(Q^l_{l-h+1}(e_l\pino{l-j}{k-j})e_{l-h}\bigr)\\
            & = \mathcal{P}_{0}\bigl(Q^l_{l-h+1}(e_l\pino{l-j}{k-j})Q^{l-h}_1(e_{l-h})\bigr)\\
            & +\mathcal{P}_{0}\bigl(Q^l_{l-h+1}(e_l\pino{l-j}{k-j})H^{l-h}_1(e_{l-h})\bigr).
        \end{split}
    \end{equation}
    Now note that $H^{l-h}_{1}(e_{l-h})$ is independent of $\sigma\bigl(H^{l}_{l-h+1}(e_l\pino{l-j}{k-j}),\sigf{0}\bigr)$ and hence also of $\sigma\bigl(H^{l}_{l-h+1}(e_l\pino{l-j}{k-j}),\sigf{-1}\bigr)$. This implies 
    \begin{equation*}
        \mathcal{P}_{0}\bigl(H^l_{l-h+1}(e_l\pino{l-j}{k-j})H^{l-h}_1(e_{l-h})\bigr) = \mathbb{E}\bigl(H^{l-h}_1(e_{l-h})\bigr)\mathcal{P}_{0}\bigl(Q^l_{l-h+1}(e_l\pino{l-j}{k-j})\bigr) = 0,
    \end{equation*}
    as $\mathbb{E}\bigl(H^{l-h}_1(e_{l-h})\bigr) = \mathbb{E}(e_{l-h}) = 0$. Thus we get 
    \begin{equation}\label{eq:stepIV-2:IV_4':P-id-1}
        \mathcal{P}_{0}(e_l \pino{l-j}{k-j}e_{l-h}) = \mathcal{P}_{0}\bigl(Q^l_{l-h+1}(e_l\pino{l-j}{k-j})Q^{l-h}_1(e_{l-h})\bigr) +\mathcal{P}_{0}\bigl(e_l\pino{l-j}{k-j}H^{l-h}_1(e_{l-h})\bigr).
    \end{equation}
    We focus our attention on the second term and observe that $H^l_{l-j+1}(e_l)$ is independent of $\sigma\bigl(\pino{l-h}{k-j}H^{l-h}_1(e_{l-h}),\sigf{0}\bigr)$, and hence of $\sigma\bigl(\pino{l-h}{k-j}H^{l-h}_1(e_{l-h}),\sigf{-1}\bigr)$. This allows us to replace $e_l$ by $Q^l_{l-j+1}(e_l)$ in the last expression, i.e.,
    \begin{equation}\label{eq:stepIV-2:IV_4':P-id-2}
        \mathcal{P}_{0}\bigl(e_l\pino{l-j}{k-j}H^{l-h}_1(e_{l-h})\bigr) = \mathcal{P}_{0}\bigl(Q^l_{l-j+1}(e_l)\pino{l-j}{k-j}H^{l-h}_1(e_{l-h})\bigr).
    \end{equation}
    Using again that $\pino{l-j}{k-j} = Q^{l-j}_{l-h+1}(\pino{l-j}{k-j}) + H^{l-j}_{l-h+1}(\pino{l-j}{k-j})$ and get 
    \begin{equation}\label{eq:stepIV-2:IV_4':P-id-3}
        \begin{split}
            \mathcal{P}_{0}\bigl(Q^l_{l-j+1}(e_l)\pino{l-j}{k-j}H^{l-h}_1(e_{l-h})\bigr) & = \mathcal{P}_{0}\bigl(Q^l_{l-j+1}(e_l)Q^{l-j}_{l-h+1}(\pino{l-j}{k-j})H^{l-h}_1(e_{l-h})\bigr) \\
            & + \mathcal{P}_{0}\bigl(Q^l_{l-j+1}(e_l)H^{l-j}_{l-h+1}(\pino{l-j}{k-j})H^{l-h}_1(e_{l-h})\bigr).
        \end{split}
    \end{equation}
Since again $H^l_{l-j+1}(e_l)$ is independent of $\sigma(H^{l-j}_{l-h+1}(\pino{l-j}{k-j})H^{l-h}_1(e_{l-h}),\sigf{0})$ and hence of $\sigma\bigl(H^{l-j}_{l-h+1}(\pino{l-j}{k-j})H^{l-h}_1(e_{l-h}),\sigf{-1}\bigr)$, thus we have 
    \begin{multline*}
        \mathcal{P}_{0}\bigl(H^l_{l-j+1}(e_l)H^{l-j}_{l-h+1}(\pino{l-j}{k-j})H^{l-h}_1(e_{l-h})\bigr)\\
        = \mathbb{E}\bigl(H^l_{l-j+1}(e_l)\bigr)\mathcal{P}_{0}\bigl(H^{l-j}_{l-h+1}(\pino{l-j}{k-j})H^{l-h}_1(e_{l-h})\bigr) = 0,
    \end{multline*}
    as $\mathbb{E}\bigl(H^l_{l-j+1}(e_l)\bigr) = 0$. Hence
    \begin{equation}\label{eq:stepIV-2:IV_4':P-id-4}
    \begin{split}
        \mathcal{P}_{0}\bigl(Q^l_{l-j+1}(e_l)H^{l-j}_{l-h+1}(\pino{l-j}{k-j})H^{l-h}_1(e_{l-h})\bigr) & = \mathcal{P}_{0}\bigl(e_lH^{l-j}_{l-h+1}(\pino{l-j}{k-j})H^{l-h}_1(e_{l-h})\bigr).
    \end{split}
    \end{equation}
    Similarly, since $H^l_1(e_l)H^{l-j}_{l-h+1}(\pino{l-j}{k-j})H^{l-h}_1(e_{l-h})$ is independent of $\sigf{0}$ and $\sigf{-1}$ we have 
    \begin{equation*}
    \begin{split}
        \mathcal{P}_{0}\bigl(H^l_1(e_l)H^{l-j}_{l-h+1}(\pino{l-j}{k-j})H^{l-h}_1(e_{l-h})\bigr) &= \mathbb{E}\bigl(H^l_1(e_l)H^{l-j}_{l-h+1}(\pino{l-j}{k-j})H^{l-h}_1(e_{l-h})\bigr)\\
        &- \mathbb{E}\bigl(H^l_1(e_l)H^{l-j}_{l-h+1}(\pino{l-j}{k-j})H^{l-h}_1(e_{l-h})\bigr) = 0.
    \end{split}
    \end{equation*}
    This means that we can replace $e_l$ by $Q^l_1(e_l)$ in $\mathcal{P}_{0}(e_lH^{l-j}_{l-h+1}(\pino{l-j}{k-j})H^{l-h}_1(e_{l-h}))$, leading to
    \begin{equation}\label{eq:stepIV-2:IV_4':P-id-5}
        \mathcal{P}_{0}\bigl(e_lH^{l-j}_{l-h+1}(\pino{l-j}{k-j})H^{l-h}_1(e_{l-h})\bigr) = \mathcal{P}_{0}\bigl(Q^l_1(e_l)H^{l-j}_{l-h+1}(\pino{l-j}{k-j})H^{l-h}_1(e_{l-h})\bigr).
    \end{equation}
    Combining \eqref{eq:stepIV-2:IV_4':P-id-1}, \eqref{eq:stepIV-2:IV_4':P-id-2}, \eqref{eq:stepIV-2:IV_4':P-id-3}, \eqref{eq:stepIV-2:IV_4':P-id-4} and \eqref{eq:stepIV-2:IV_4':P-id-5} we get 
    \begin{equation*}
        \begin{split}
            \mathcal{P}_{0}\bigl(e_l \pino{l-j}{k-j}e_{l-h}\bigr) & = \mathcal{P}_{0}\bigl(Q^l_{l-h+1}(e_l\pino{l-j}{k-j})Q^{l-h}_1(e_{l-h})\bigr) \\
            & + \mathcal{P}_{0}\bigl(Q^l_{l-j+1}(e_l)Q^{l-j}_{l-h+1}(\pino{l-j}{k-j})H^{l-h}_1(e_{l-h})\bigr) \\
            & +\mathcal{P}_{0}\bigl(Q^l_1(e_l)H^{l-j}_{l-h+1}(\pino{l-j}{k-j})H^{l-h}_1(e_{l-h})\bigr).
        \end{split}
    \end{equation*}

    Since the conditional expectation is a $L^p$ contraction, we may apply the triangle and Hölder's inequality to the effect of 
    \begin{equation}\label{eq:stepIV-2:IV_4':l>h:P-rate-est}
        \begin{split}
            \|\mathcal{P}_{0}(e_l \pino{l-j}{k-j}e_{l-h})\|_{\frac{4p}{3}} & \lesssim \|Q^l_{l-h+1}(e_l\pino{l-j}{k-j})\|_{2p}\|Q^{l-h}_1(e_{l-h})\|_{4p} \\
            & + \|Q^l_{l-j+1}(e_l)\|_{4p}\|Q^{l-j}_{l-h+1}(\pino{l-j}{k-j})\|_{4p}\|H^{l-h}_1(e_{l-h})\|_{4p} \\
            & +\|Q^l_1(e_l)\|_{4p}\|H^{l-j}_{l-h+1}(\pino{l-j}{k-j})\|_{4p}\|H^{l-h}_1(e_{l-h})\|_{4p} \\
            & \lesssim \|Q^l_{l-h+1}(e_l\pino{l-j}{k-j})\|_{2p}\|Q^{l-h}_1(e_{l-h})\|_{4p} \\
            & + \|Q^l_{l-j+1}(e_l)\|_{4p}\|Q^{l-j}_{l-h+1}(\pino{l-j}{k-j})\|_{4p} \\
            & +\|Q^l_1(e_l)\|_{4p}.
        \end{split}
    \end{equation}
    In the last inequality, we have again used the fact that the conditional expectation is a contraction in addition to the stationarity of the underlying processes and the fact that $\|\pino{0}{k-j}\|_{4p}$ is uniformly bounded in $k$ and $j$ by Lemma \ref{lemma:basic-props-AR(oo)}. 
    An application of Lemma \ref{lemma:Q} yields the following rates
    \begin{equation*}
        \begin{split}
            \|Q^l_{l-h+1}(e_l\pino{l-j}{k-j})\|_{2p} & \lesssim h^{-\frac{5}{2}} + (h-j)^{-\frac{5}{2}}, \\
            \|Q^{l-h}_{1}(e_{l-h})\|_{4p} &\lesssim (l-h)^{-\frac{5}{2}}, \\
            \|Q^l_{l-j+1}(e_l)\|_{4p} & \lesssim j^{-\frac{5}{2}} \\
            \|Q^{l-j}_{l-h+1}(\pino{l-j}{k-j})\|_{4p} & \lesssim (h-j)^{-\frac{5}{2}} \quad \text{and} \\
            \|Q^l_1(e_l)\|_{4p} &\lesssim l^{-\frac{5}{2}}.
        \end{split}
    \end{equation*}
    Plugging these rates into \eqref{eq:stepIV-2:IV_4':l>h:P-rate-est} and \eqref{eq:stepIV-2:IV_4'-est} (note that we have already dealt with the terms where $l=h$), we may estimate $IV_4'$ by $IV_4' \lesssim IV_4'(a) + IV_4'(b) +IV_4'(c)$, where
    \begin{equation*}
        \begin{split}
            IV_4'(a) &= \frac{1}{k} \sum_{j=1}^k \sum_{h=k+1}^{N-1} \sum_{l=h+1}^{h+j-1} \biggl(h^{-\frac{5}{2}} + (h-j)^{-\frac{5}{2}}\biggr)(l-h)^{-\frac{5}{2}}, \\
            IV_4'(b) &= \frac{1}{k} \sum_{j=1}^k \sum_{h=k+1}^{N-1} \sum_{l=h+1}^{h+j-1} j^{-\frac{5}{2}}(h-j)^{-\frac{5}{2}} \quad \text{and} \\
            IV_4'(c) & =\frac{1}{k} \sum_{j=1}^k \sum_{h=k+1}^{N-1} \sum_{l=h+1}^{h+j-1} l^{-\frac{5}{2}},
        \end{split}
    \end{equation*}
    all of which are of order $O(k^{-1/2})$.}{%
    The term $\|\mathcal{P}_{0}(e_{l}\pino{l-j}{k-j}e_{l-h})\|_{\frac{4p}{3}}$ can be treated by the same arguments that were used to establish \eqref{eq:stepIV:IV_2:l=j-rate} and \eqref{eq:stepIV:IV_2:l>j-rate}.
    Again, we want to treat the cases $l=h$ and $l>h$ separately.
    We just give the necessary identities and rates, that immediately follow from Lemma \ref{lemma:Q}.

    For the case $l=h$, we show 
    \begin{equation}\label{eq:stepIV-2:IV_4':P_0-first-step}
        \mathcal{P}_{0}(e_h \pino{h-j}{k-j}e_{0}) = \mathcal{P}_{0}\bigl(Q^h_{1}(e_h\pino{h-j}{k-j})e_{0}\bigr),
    \end{equation}
    which immediately yields
    \begin{equation*}
      \|\mathcal{P}_{0}(e_{h}\pino{h-j}{k-j}e_{0})\|_{\frac{4p}{3}} \lesssim \bigl\|Q^h_1(e_h\pino{h-j}{k-j})\bigr\|_{2p} \lesssim (h-j)^{-5/2} + h^{-5/2},
    \end{equation*}
    by Lemma \ref{lemma:basic-props-AR(oo)} and Lemma \ref{lemma:Q}.
    This gives
    \begin{equation*}
        \begin{split}
 \frac{1}{k} \sum_{j=1}^k \sum_{h=k+1}^{N-1} \biggl((h-j)^{-\frac{5}{2}} + h^{-\frac{5}{2}}\biggr) \lesssim k^{-\frac{1}{2}}.
        \end{split}
    \end{equation*}
    Here we have used, that $\sum_{m\geq h} m^\alpha \lesssim h^{\alpha +1}$ for $\alpha < -1$, and $\sum_{1\leq m \leq h} m^\beta \lesssim h^{\beta + 1}$ for $\beta >-1$.

The case $l>h$ is somewhat more involved. 
We first establish
\begin{equation*}
        \begin{split}
            \mathcal{P}_{0}\bigl(e_l \pino{l-j}{k-j}e_{l-h}\bigr) & = \mathcal{P}_{0}\bigl(Q^l_{l-h+1}(e_l\pino{l-j}{k-j})Q^{l-h}_1(e_{l-h})\bigr) \\
            & + \mathcal{P}_{0}\bigl(Q^l_{l-j+1}(e_l)Q^{l-j}_{l-h+1}(\pino{l-j}{k-j})H^{l-h}_1(e_{l-h})\bigr) \\
            & +\mathcal{P}_{0}\bigl(Q^l_1(e_l)H^{l-j}_{l-h+1}(\pino{l-j}{k-j})H^{l-h}_1(e_{l-h})\bigr).
        \end{split}
    \end{equation*}
    This immediately gives
\begin{equation}\label{eq:stepIV-2:IV_4':l>h:P-rate-est}
        \begin{split}
            \|\mathcal{P}_{0}(e_l \pino{l-j}{k-j}e_{l-h})\|_{\frac{4p}{3}}
            & \lesssim \|Q^l_{l-h+1}(e_l\pino{l-j}{k-j})\|_{2p}\|Q^{l-h}_1(e_{l-h})\|_{4p} \\
            & + \|Q^l_{l-j+1}(e_l)\|_{4p}\|Q^{l-j}_{l-h+1}(\pino{l-j}{k-j})\|_{4p} \\
            & +\|Q^l_1(e_l)\|_{4p},
        \end{split}
    \end{equation}
    where the constants are uniformly bounded by Lemma \ref{lemma:basic-props-AR(oo)}.
    Using Lemma \ref{lemma:Q}, we get the following rates
    \begin{equation*}
        \begin{split}
            \|Q^l_{l-h+1}(e_l\pino{l-j}{k-j})\|_{2p} & \lesssim h^{-\frac{5}{2}} + (h-j)^{-\frac{5}{2}}, \\
            \|Q^{l-h}_{1}(e_{l-h})\|_{4p} &\lesssim (l-h)^{-\frac{5}{2}}, \\
            \|Q^l_{l-j+1}(e_l)\|_{4p} & \lesssim j^{-\frac{5}{2}}, \\
            \|Q^{l-j}_{l-h+1}(\pino{l-j}{k-j})\|_{4p} & \lesssim (h-j)^{-\frac{5}{2}}, \quad \text{and} \\
            \|Q^l_1(e_l)\|_{4p} &\lesssim l^{-\frac{5}{2}}.
        \end{split}
    \end{equation*}
    Plugging these rates into \eqref{eq:stepIV-2:IV_4'-est}, we can estimate $IV_{4}'$ by $IV_{4}'(a) + IV_{4}'(b) + IV_{4}'(c)$, where
\begin{equation*}
        \begin{split}
            IV_4'(a) &= \frac{1}{k} \sum_{j=1}^k \sum_{h=k+1}^{N-1} \sum_{l=h+1}^{h+j-1} \biggl(h^{-\frac{5}{2}} + (h-j)^{-\frac{5}{2}}\biggr)(l-h)^{-\frac{5}{2}}, \\
            IV_4'(b) &= \frac{1}{k} \sum_{j=1}^k \sum_{h=k+1}^{N-1} \sum_{l=h+1}^{h+j-1} j^{-\frac{5}{2}}(h-j)^{-\frac{5}{2}}, \quad \text{and} \\
            IV_4'(c) & =\frac{1}{k} \sum_{j=1}^k \sum_{h=k+1}^{N-1} \sum_{l=h+1}^{h+j-1} l^{-\frac{5}{2}},
        \end{split}
    \end{equation*}
    all of which are of order $O(k^{-1/2})$.

  }

  Now we move on to $IV_4''$. Using $\mathcal{P}_{t-l}(\I{t}{s}{k}{j}) \disteq \mathcal{P}_{0}(\I{l}{l-h}{k}{j})$, we may estimate
    \begin{equation}\label{eq:stepIV-2:IV_4''-basic-est}
      IV_4'' \lesssim \frac{1}{k}\sum_{j=1}^k \sum_{h=k+1}^{N-1} \sum_{l=h+j}^\infty \|\mathcal{P}_0(\I{l}{l-h}{k}{j})\|_p. 
    \end{equation}
    \cformat{main-lemma-stepIV2-3}{%
      We again distinguish between $l=h+j$ and $l>h+j$, and again, each case may be treated  by the same arguments used to establish \eqref{eq:stepIV:IV_2:l=j-rate} and 
    \eqref{eq:stepIV:IV_2:l>j-rate}.
    For the case $l=h+j$ we note that 
    \begin{equation*}
        \mathcal{P}_0\bigl(e_{h+j}\pino{h}{k-j}e_j\pino{0}{k-j}\bigr) =\mathcal{P}_0\bigl(Q^{h+j}_{j+1}(e_{h+j}\pino{h}{k-j})e_j\pino{0}{k-j}\bigr),  
    \end{equation*}
    as $H^{h+j}_{j+1}(e_{h+j}\pino{h}{k-j})$ is independent of $\sigma\bigl(e_j\pino{0}{k-j}, \sigf{0}\bigr)$, and hence
    \begin{equation*}
        \mathcal{P}_0\bigl(H^{h+j}_{j+1}(e_{h+j}\pino{h}{k-j})e_j\pino{0}{k-j}\bigr) = \mathbb{E}\bigl(H^{h+j}_{j+1}(e_{h+j}\pino{h}{k-j})\bigr) \mathcal{P}_0(e_j\pino{0}{k-j}) = 0,
    \end{equation*}
    since $\mathbb{E}\bigl(H^{h+j}_{j+1}(e_{h+j}\pino{h}{k-j})\bigr) = \mathbb{E}(e_{h+j}\pino{h}{k-j}) = 0$ by Lemma \ref{lemma:basic-props-AR(oo)}. Using Hölder's inequality and the fact, that the conditional expectation is a $L^p$ contraction, we get 
    \begin{equation*}
        \bigl\|\mathcal{P}_0\bigl(Q^{h+j}_{j+1}(e_{h+j}\pino{h}{k-j})e_j\pino{0}{k-j}\bigr)\bigr\|_p \leq 2 \bigl\|Q^{h+j}_{j+1}(e_{h+j}\pino{h}{k-j})\bigr\|_{2p}\|e_0\|_{4p}\|\pino{0}{k-j}\|_{4p}. 
    \end{equation*}
    Since $\|\pino{0}{k-j}\|_{4p}$ is uniformly bounded in $k$ and $j$ by Lemma \ref{lemma:basic-props-AR(oo)}, we get
    \begin{equation*}
        \bigl\|\mathcal{P}_0\bigl(Q^{h+j}_{j+1}(e_{h+j}\pino{h}{k-j})e_j\pino{0}{k-j}\bigr)\bigr\|_p \lesssim \bigl\|Q^{h+j}_{j+1}(e_{h+j}\pino{h}{k-j})\bigr\|_{2p}. 
    \end{equation*}
    By Lemma \ref{lemma:Q} and the subsequent remark, we arrive at the rate
    \begin{equation*}
        \bigl\|\mathcal{P}_0\bigl(Q^{h+j}_{j+1}(e_{h+j}\pino{h}{k-j})e_j\pino{0}{k-j}\bigr)\bigr\|_p \lesssim h^{-\frac{5}{2}} + (h-j)^{-\frac{5}{2}}. 
    \end{equation*}
    Plugging these estimates into \eqref{eq:stepIV-2:IV_4''-basic-est}, the right-hand side may be estimated by (we are only looking at $l=h+j$)
    \begin{equation*}
        \frac{1}{k}\sum_{j=1}^k \sum_{h=k+1}^\infty \biggl(h^{-\frac{5}{2}} + (h-j)^{-\frac{5}{2}}\biggr) = \sum_{h=k+1}^\infty h^{-\frac{5}{2}} + \frac{1}{k} 
 \sum_{j=1}^k \sum_{h=k+1-j}^\infty h^{-\frac{5}{2}}.
    \end{equation*}
    The first term is of order $O(k^{-1/2})$. For the second term, we may use $\sum_{h\geq m} h^{\alpha} \lesssim m^{\alpha + 1}$ for $\alpha < -1$ and $\sum_{1\leq h \leq m} h^\beta \lesssim m^{\beta+1}$ for $\beta >-1$. 
    This yields
    \begin{equation*}
        \frac{1}{k} 
 \sum_{j=1}^k \sum_{h=k+1-j}^\infty h^{-\frac{5}{2}} = \frac{1}{k} 
 \sum_{j=0}^{k-1} \sum_{h=j+1}^\infty h^{-\frac{5}{2}} \lesssim \frac{1}{k} \sum_{j=1}^k j^{-\frac{1}{2}} \lesssim k^{-\frac{1}{2}. }
    \end{equation*}

   Next, we move on to the case $l>h+j$. Looking at $\|\mathcal{P}_0(e_l\pino{l-j}{k-j}e_{l-h}\pino{l-h-j}{k-j})\|_p$ we note that $H^l_{l-h+1}(e_l\pino{l-j}{k-j})$ is independent of $\sigma(e_{l-h}\pino{l-h-j}{k-j}, \sigf{0})$ and independent of $\sigma(e_{l-h}\pino{l-h-j}{k-j}, \sigf{-1})$. As above, this implies
    \begin{multline*}
        \mathcal{P}_0\bigl(H^l_{l-h+1}(e_l\pino{l-j}{k-j})e_{l-h}\pino{l-h-j}{k-j}\bigr)\\
        = \mathbb{E}\bigl(H^l_{l-h+1}(e_l\pino{l-j}{k-j})\bigr)\mathcal{P}_0(e_l\pino{l-j}{k-j}e_{l-h}\pino{l-h-j}{k-j}) = 0,
    \end{multline*}
    as $\mathbb{E}\bigl(H^l_{l-h+1}(e_l\pino{l-j}{k-j})\bigr) = \mathbb{E}(e_l\pino{l-j}{k-j}) = 0$, and we obtain 
    \begin{equation*}
        \mathcal{P}_0(e_l\pino{l-j}{k-j}e_{l-h}\pino{l-h-j}{k-j}) = \mathcal{P}_0\bigl(Q^l_{l-h+1}(e_l\pino{l-j}{k-j})e_{l-h}\pino{l-h-j}{k-j}\bigr).
    \end{equation*}
    By expanding $e_{l-h}\pino{l-h-j}{k-j}$ into $e_{l-h}\pino{l-h-j}{k-j} = Q^{l-h}_1(e_{l-h}\pino{l-h-j}{k-j}) + H^{l-h}_1(e_{l-h}\pino{l-h-j}{k-j})$ we get
    \begin{equation*}
    \begin{split}
        \mathcal{P}_0\bigl(Q^l_{l-h+1}(e_l\pino{l-j}{k-j})e_{l-h}\pino{l-h-j}{k-j}\bigr) & = \mathcal{P}_0\bigl(Q^l_{l-h+1}(e_l\pino{l-j}{k-j})Q^{l-h}_1(e_{l-h}\pino{l-h-j}{k-j})\bigr) \\
        & + \mathcal{P}_0\bigl(Q^l_{l-h+1}(e_l\pino{l-j}{k-j})H^{l-h}_1(e_{l-h}\pino{l-h-j}{k-j})\bigr).
    \end{split}
    \end{equation*}
    Now since $l>h+j$, $H^l_{l-h+1}(e_l\pino{l-j}{k-j})H^{l-h}_1(e_{l-h}\pino{l-h-j}{k-j})$ is independent of $\sigf{0}$ and $\sigf{-1}$). Thus we have 
    \begin{equation*}
        \begin{split}
            \mathcal{P}_0\bigl(H^l_{l-h+1}(e_l\pino{l-j}{k-j})H^{l-h}_1(e_{l-h}\pino{l-h-j}{k-j})\bigr) & = \mathbb{E}\bigl(H^l_{l-h+1}(e_l\pino{l-j}{k-j})H^{l-h}_1(e_{l-h}\pino{l-h-j}{k-j})\bigr) \\
            & - \mathbb{E}\bigl(H^l_{l-h+1}(e_l\pino{l-j}{k-j})H^{l-h}_1(e_{l-h}\pino{l-h-j}{k-j})\bigr) = 0. 
        \end{split}
    \end{equation*}
    Hence we have 
    \begin{equation*}
        \mathcal{P}_0\bigl(Q^l_{l-h+1}(e_l\pino{l-j}{k-j})H^{l-h}_1(e_{l-h}\pino{l-h-j}{k-j})\bigr) = \mathcal{P}_0\bigl(e_l\pino{l-j}{k-j}H^{l-h}_1(e_{l-h}\pino{l-h-j}{k-j})\bigr).
    \end{equation*}
    Since $H^{l}_{l-j+1}(e_l)$ is independent of $\sigma(\pino{l-j}{k-j}H^{l-h}_1(e_{l-h}\pino{l-h-j}{k-j}), \sigf{0})$, we get 
    \begin{multline*}
        \mathcal{P}_0\bigl(H^l_{l-j+1}(e_l)\pino{l-j}{k-j}H^{l-h}_1(e_{l-h}\pino{l-h-j}{k-j})\bigr)\\
        = \mathbb{E}\bigl(H^l_{l-j+1}(e_l)\bigr)\mathcal{P}_0\bigl(\pino{l-j}{k-j}H^{l-h}_1(e_{l-h}\pino{l-h-j}{k-j})\bigr) = 0,
    \end{multline*}
    as $\mathbb{E}\bigl(H^l_{l-j+1}(e_l)\bigr) = \mathbb{E}(e_l) = 0$. Hence we may replace $e_l$ by $Q^l_{l-j+1}(e_l)$ in the last expression, that is,
    \begin{equation*}
        \mathcal{P}_0\bigl(e_l\pino{l-j}{k-j}H^{l-h}_1(e_{l-h}\pino{l-h-j}{k-j})\bigr) = \mathcal{P}_0\bigl(Q^l_{l-j+1}(e_l)\pino{l-j}{k-j}H^{l-h}_1(e_{l-h}\pino{l-h-j}{k-j})\bigr).
    \end{equation*}
    Now, writing $\pino{l-j}{k-j} = Q^{l-j}_{1}(\pino{l-j}{k-j}) + H^{l-j}_{1}(\pino{l-j}{k-j})$ we get 
    \begin{equation*}
        \begin{split}
            \mathcal{P}_0\bigl(Q^l_{l-j+1}(e_l)\pino{l-j}{k-j}H^{l-h}_1(e_{l-h}\pino{l-h-j}{k-j})\bigr) & = \mathcal{P}_0\bigl(Q^l_{l-j+1}(e_l)Q^{l-j}_{1}(\pino{l-j}{k-j})H^{l-h}_1(e_{l-h}\pino{l-h-j}{k-j})\bigr) \\
            & + \mathcal{P}_0\bigl(Q^l_{l-j+1}(e_l)H^{l-j}_{1}(\pino{l-j}{k-j})H^{l-h}_1(e_{l-h}\pino{l-h-j}{k-j})\bigr).
        \end{split}
    \end{equation*}
    Since $H^l_{l-j+1}(e_l)H^{l-j}_{1}(\pino{l-j}{k-j})H^{l-h}_1(e_{l-h}\pino{l-h-j}{k-j})$ is independent of $\sigf{0}$ and $\sigf{-1}$, we have 
    \begin{multline*}
        \mathcal{P}_0\bigl(H^l_{l-j+1}(e_l)H^{l-j}_{1}(\pino{l-j}{k-j})H^{l-h}_1(e_{l-h}\pino{l-h-j}{k-j})\bigr)\\ 
        = \mathbb{E}\bigl(H^l_{l-j+1}(e_l)H^{l-j}_{1}(\pino{l-j}{k-j})H^{l-h}_1(e_{l-h}\pino{l-h-j}{k-j})\bigr) \\
        - \mathbb{E}\bigl(H^l_{l-j+1}(e_l)H^{l-j}_{1}(\pino{l-j}{k-j})H^{l-h}_1(e_{l-h}\pino{l-h-j}{k-j})\bigr) = 0.
    \end{multline*}
    Hence we have 
    \begin{multline*}
    \mathcal{P}_0\bigl(Q^l_{l-j+1}(e_l)H^{l-j}_{1}(\pino{l-j}{k-j})H^{l-h}_1(e_{l-h}\pino{l-h-j}{k-j})\bigr)\\
    = \mathcal{P}_0\bigl(e_lH^{l-j}_{1}(\pino{l-j}{k-j})H^{l-h}_1(e_{l-h}\pino{l-h-j}{k-j})\bigr).
    \end{multline*}
    We may further replace $e_l$ with $Q^l_1(e_l)$ in this expression, as $H^l_1(e_l)H^{l-j}_{1}(\pino{l-j}{k-j})H^{l-h}_1(e_{l-h}\pino{l-h-j}{k-j})$ is independent of $\sigf{0}$, which implies 
    \begin{multline*}
        \mathcal{P}_0\bigl(H^l_1(e_l)H^{l-j}_{1}(\pino{l-j}{k-j})H^{l-h}_1(e_{l-h}\pino{l-h-j}{k-j})\bigr) \\
         = \mathbb{E}\bigl(H^l_1(e_l)H^{l-j}_{1}(\pino{l-j}{k-j})H^{l-h}_1(e_{l-h}\pino{l-h-j}{k-j})\bigr) \\
         - \mathbb{E}\bigl(H^l_1(e_l)H^{l-j}_{1}(\pino{l-j}{k-j})H^{l-h}_1(e_{l-h}\pino{l-h-j}{k-j})\bigr) = 0.
    \end{multline*}
    In total we have 
    \begin{equation*}
      \begin{split}
 \mathcal{P}_0\bigl(e_l\pino{l-j}{k-j}e_{l-h}\pino{l-h-j}{k-j}\bigr) & =\mathcal{P}_0\bigl(Q^l_{l-h+1}(e_l\pino{l-j}{k-j})Q^{l-h}_1(e_{l-h}\pino{l-h-j}{k-j})\bigr) \\
 & + \mathcal{P}_0\bigl(Q^l_{l-j+1}(e_l)Q^{l-j}_{1}(\pino{l-j}{k-j})H^{l-h}_1(e_{l-h}\pino{l-h-j}{k-j})\bigr) \\
 & + \mathcal{P}_0\bigl(Q^l_{1}(e_l)H^{l-j}_{1}(\pino{l-j}{k-j})H^{l-h}_1(e_{l-h}\pino{l-h-j}{k-j})\bigr).
      \end{split}
    \end{equation*}
    Now, using Hölder's inequality, the fact that the conditional expectation is a $L^{p}$ contraction, we obtain
    \begin{equation}\label{eq:IV_4'':l>h+j}
        \begin{split}
            \|\mathcal{P}_0(e_l\pino{l-j}{k-j}e_{l-h}\pino{l-h-j}{k-j})\|_p & \lesssim\|Q^l_{l-h+1}(e_l\pino{l-j}{k-j})\|_{2p}\|Q^{l-h}_1(e_{l-h}\pino{l-h-j}{k-j})\|_{2p} \\
 & + \|Q^l_{l-j+1}(e_l)\|_{4p}\|Q^{l-j}_{1}(\pino{l-j}{k-j})\|_{4p}\\
 & + \|Q^l_{1}(e_l)\|_{4p}.
        \end{split}
    \end{equation}
    Here we have also used that 
    \begin{equation*}
        \begin{split}
            \bigl\|H^{l-j}_{l-h+1}(\pino{l-j}{k-j})\bigr\|_{4p} &\leq \|\pino{0}{k-j}\|_{4p} \quad \text{and} \\
           \bigl\|H^{l-h}_1(e_{l-h}\pino{l-h-j}{k-j})\bigr\|_{2p} &\leq \|e_{0}\pino{0}{k-j}\|_{2p} \leq \|e_{0}\|_{4p} \|\pino{0}{k-j}\|_{4p}
        \end{split}
    \end{equation*}
    are uniformly bounded in $k$ and $j$ by Lemma \ref{lemma:basic-props-AR(oo)} and the stationarity of $(X_t)_{t\in\zz}$. By Lemma \ref{lemma:Q} and the subsequent discussion, we get the following rates
    \begin{equation*}
        \begin{split}
            \|Q^l_{l-h+1}(e_l\pino{l-j}{k-j})\|_{2p} & \lesssim h^{-\frac{5}{2}} + (h-j)^{-\frac{5}{2}}, \\
            \|Q^{l-h}_1(e_{l-h}\pino{l-h-j}{k-j})\|_{2p} & \lesssim (l-h)^{-\frac{5}{2}} + (l-h-j)^{-\frac{5}{2}},\\
            \|Q^l_{l-j+1}(e_l)\|_{4p} & \lesssim j^{-\frac{5}{2}},\\
            \|Q^{l-j}_{1}(\pino{l-j}{k-j})\|_{4p} & \lesssim (l-j)^{-\frac{5}{2}} \quad \text{and}\\
            \|Q^l_{1}(e_l)\|_{4p} & \lesssim l^{-\frac{5}{2}}.
        \end{split}
    \end{equation*}
    
    Plugging this and \eqref{eq:IV_4'':l>h+j} into \eqref{eq:stepIV-2:IV_4''-basic-est}, we may estimate $IV_4''$ by $IV_4'' \lesssim IV_4''(a) + IV_4''(b) + IV_4''(c)$, where
    \begin{equation*}
    \begin{split}
        IV_4''(a) &= \frac{1}{k}\sum_{j=1}^k \sum_{h=k+1}^{N-1}  \biggl(h^{-\frac{5}{2}} + (h-j)^{\frac{5}{2}}\biggr)\sum_{l=h+j+1}^\infty \biggl((l-h)^{-\frac{5}{2}} + (l-h-j)^{-\frac{5}{2}}\biggr), \\
        IV_4''(b) & = \frac{1}{k}\sum_{j=1}^k j^{-\frac{5}{2}}\biggl( \sum_{h=k+1}^{N-1} h^{-\frac{5}{2}} + \sum_{l=h+j+1}^\infty (l-j)^{-\frac{5}{2}}\biggr), \quad \text{and} \\
        IV_4''(c) &= \frac{1}{k}\sum_{j=1}^k \biggl(\sum_{h=k+1}^{N-1} (h+j)^{-\frac{5}{2}} +  \sum_{l=h+j+1}^\infty l^{-\frac{5}{2}}\biggr),
    \end{split}
    \end{equation*}
  all of which are of order $O(k^{-1/2})$.}{%
We again distinguish between $l=h+j$ and $l>h+j$, and each case may be treated by the same arguments used to establish \eqref{eq:stepIV:IV_2:l=j-rate} and 
    \eqref{eq:stepIV:IV_2:l>j-rate}.
    For the case $l=h+j$, we note that 
    \begin{equation*}
        \mathcal{P}_0\bigl(e_{h+j}\pino{h}{k-j}e_j\pino{0}{k-j}\bigr) =\mathcal{P}_0\bigl(Q^{h+j}_{j+1}(e_{h+j}\pino{h}{k-j})e_j\pino{0}{k-j}\bigr).
    \end{equation*}
    Using Lemma \ref{lemma:basic-props-AR(oo)} and Lemma \ref{lemma:Q}, we get
    \begin{equation*}
        \bigl\|\mathcal{P}_0\bigl(Q^{h+j}_{j+1}(e_{h+j}\pino{h}{k-j})e_j\pino{0}{k-j}\bigr)\bigr\|_p \lesssim \bigl\|Q^{h+j}_{j+1}(e_{h+j}\pino{h}{k-j})\bigr\|_{2p} \lesssim h^{-\frac{5}{2}} + (h-j)^{-\frac{5}{2}}. 
    \end{equation*}
    Plugging these estimates into \eqref{eq:stepIV-2:IV_4''-basic-est}, the right-hand side may be estimated by (we are only looking at $l=h+j$)
    \begin{equation*}
        \frac{1}{k}\sum_{j=1}^k \sum_{h=k+1}^\infty \biggl(h^{-\frac{5}{2}} + (h-j)^{-\frac{5}{2}}\biggr) = \sum_{h=k+1}^\infty h^{-\frac{5}{2}} + \frac{1}{k} 
 \sum_{j=1}^k \sum_{h=k+1-j}^\infty h^{-\frac{5}{2}} \lesssim k^{-\frac{1}{2}}.
    \end{equation*}
%

   Next, we move on to the case $l>h+j$. 
   First, we show that 
   \begin{equation*}
      \begin{split}
 \mathcal{P}_0\bigl(e_l\pino{l-j}{k-j}e_{l-h}\pino{l-h-j}{k-j}\bigr) & =\mathcal{P}_0\bigl(Q^l_{l-h+1}(e_l\pino{l-j}{k-j})Q^{l-h}_1(e_{l-h}\pino{l-h-j}{k-j})\bigr) \\
 & + \mathcal{P}_0\bigl(Q^l_{l-j+1}(e_l)Q^{l-j}_{1}(\pino{l-j}{k-j})H^{l-h}_1(e_{l-h}\pino{l-h-j}{k-j})\bigr) \\
 & + \mathcal{P}_0\bigl(Q^l_{1}(e_l)H^{l-j}_{1}(\pino{l-j}{k-j})H^{l-h}_1(e_{l-h}\pino{l-h-j}{k-j})\bigr).
      \end{split}
    \end{equation*}
    By Hölder's inequality and Lemma \ref{lemma:basic-props-AR(oo)}, we have
    \begin{equation}
        \begin{split}
            \|\mathcal{P}_0(e_l\pino{l-j}{k-j}e_{l-h}\pino{l-h-j}{k-j})\|_p & \lesssim\|Q^l_{l-h+1}(e_l\pino{l-j}{k-j})\|_{2p}\|Q^{l-h}_1(e_{l-h}\pino{l-h-j}{k-j})\|_{2p} \\
 & + \|Q^l_{l-j+1}(e_l)\|_{4p}\|Q^{l-j}_{1}(\pino{l-j}{k-j})\|_{4p}\\
 & + \|Q^l_{1}(e_l)\|_{4p}.
        \end{split}
    \end{equation}
    By Lemma \ref{lemma:Q} and the subsequent remark, we get the following rates
    \begin{equation*}
        \begin{split}
            \|Q^l_{l-h+1}(e_l\pino{l-j}{k-j})\|_{2p} & \lesssim h^{-\frac{5}{2}} + (h-j)^{-\frac{5}{2}}, \\
            \|Q^{l-h}_1(e_{l-h}\pino{l-h-j}{k-j})\|_{2p} & \lesssim (l-h)^{-\frac{5}{2}} + (l-h-j)^{-\frac{5}{2}},\\
            \|Q^l_{l-j+1}(e_l)\|_{4p} & \lesssim j^{-\frac{5}{2}},\\
            \|Q^{l-j}_{1}(\pino{l-j}{k-j})\|_{4p} & \lesssim (l-j)^{-\frac{5}{2}}, \quad \text{and}\\
            \|Q^l_{1}(e_l)\|_{4p} & \lesssim l^{-\frac{5}{2}}.
        \end{split}
    \end{equation*}
    
    Plugging this into \eqref{eq:stepIV-2:IV_4''-basic-est}, we may estimate $IV_4''$ by $IV_4'' \lesssim IV_4''(a) + IV_4''(b) + IV_4''(c)$, where
    \begin{equation*}
    \begin{split}
        IV_4''(a) &= \frac{1}{k}\sum_{j=1}^k \sum_{h=k+1}^{N-1}  \biggl(h^{-\frac{5}{2}} + (h-j)^{\frac{5}{2}}\biggr)\sum_{l=h+j+1}^\infty \biggl((l-h)^{-\frac{5}{2}} + (l-h-j)^{-\frac{5}{2}}\biggr), \\
        IV_4''(b) & = \frac{1}{k}\sum_{j=1}^k j^{-\frac{5}{2}}\biggl( \sum_{h=k+1}^{N-1} h^{-\frac{5}{2}} + \sum_{l=h+j+1}^\infty (l-j)^{-\frac{5}{2}}\biggr), \quad \text{and} \\
        IV_4''(c) &= \frac{1}{k}\sum_{j=1}^k \biggl(\sum_{h=k+1}^{N-1} (h+j)^{-\frac{5}{2}} +  \sum_{l=h+j+1}^\infty l^{-\frac{5}{2}}\biggr),
    \end{split}
    \end{equation*}
  all of which are of order $k^{-1/2}$.
}
  \end{proof}

\bibliographystyle{plain}
\bibliography{main}

\end{document}